\documentclass[draftl]{statsoc}

\usepackage{graphicx}
\usepackage{multirow}
\usepackage{subfigure}
\usepackage{amsmath,amssymb}
\usepackage{color}
\usepackage{url}
\usepackage{ulem}

\newcommand{\ve}{\boldsymbol{e}}

\newcommand{\vg}{\boldsymbol{g}}
\newcommand{\NN}{\mathbb{N}}
\newcommand{\ZZ}{\mathbb{Z}}

\newcommand{\RR}{\mathbb{R}}
\newcommand{\CC}{\mathbb{C}}
\newcommand{\EE}{\mathbb{E}}

\newcommand{\tr}{\mbox{tr}}
\newcommand{\ud}{\textup{d}}
\newcommand{\var}{\textup{var}}

\newtheorem{theorem}{Theorem}[section]
\newtheorem{defn}[theorem]{Definition}

\newtheorem{thm}{Theorem}[section]
\newtheorem{lem}[thm]{Lemma}

\title[Dynamical Seasonality and Trend]{Nonparametric  and adaptive modeling of dynamic seasonality and trend with heteroscedastic and dependent errors}

\author{Yu-Chun~Chen}
\address{Department of Medical Research and Education, National Yang-Ming University Hospital; Institute of Hospital and Health Care Administration, National Yang-Ming University, Taipei, Taiwan}
\email{yuchn.chen@gmail.com}
\author{Ming-Yen~Cheng}
\address{Department of Mathematics, National Taiwan University, Taipei, Taiwan}
\email{cheng@math.ntu.edu.tw}
\author[Chen {\it et al.}]{Hau-Tieng~Wu}
\address{Department of Statistics, University of California,  Berkeley, CA, USA.}
\email{hauwu@math.princeton.edu}

\begin{document}

\begin{abstract}
Seasonality (or periodicity) and trend are features describing an observed sequence, and extracting these features is an important issue in many scientific fields. 
However, it is not an easy task for existing methods to analyze simultaneously the trend and {\it dynamics} of the seasonality such as time-varying frequency and amplitude, and the {\it adaptivity} of the analysis to such dynamics and robustness to heteroscedastic, dependent errors is not guaranteed. These tasks become even more challenging when there exist multiple seasonal components.
We propose a nonparametric model to describe the dynamics of multi-component seasonality, and investigate the recently developed Synchrosqueezing transform (SST) in extracting these features in the presence of a trend and heteroscedastic, dependent errors. The identifiability problem of the nonparametric seasonality model is studied, and the adaptivity and robustness properties of the SST are theoretically justified in both discrete- and continuous-time settings. Consequently we have a new technique for de-coupling the trend, seasonality and heteroscedastic, dependent error process in a general nonparametric setup. Results of a series of simulations are provided, and the incidence time series of varicella and herpes zoster in Taiwan and respiratory signals observed from a sleep study are analyzed. 
\keywords{ARMA errors, Continuous-time ARMA processes, Cycles, Non-stationary processes, periodic functions, Synchrosqueezing transform, Instantaneous frequency, Time-frequency analysis}
\end{abstract}

\section{Introduction}

Seasonality (or periodicity) is a phenomenon commonly observed in a time series. 
For example, incidences of the following diseases are known to exert seasonality with high peaks in winter: cardiovascular disease \citep{Ishikawa:2012},  asthma \citep{Lin_Jones:2011}, varicella \citep{Chan:2011}, etc, which leads to higher mortality and high demand for medical resource in every winter season. Hence good understanding of the seasonality of a disease is important in both the clinical science and the public health \citep{Stone:2007}. %
Trend is another phenomenon commonly of interest in time series analysis; for instance, to determine if a general application of certain vaccine is effective in the society, we may like to investigate if overall trend of the disease incidence has changed. %
An explicit example which we will discuss in Section \ref{VHZ} is how the general application of varicella vaccine influences the seasonal behavior and the trend of the disease incidence, which is plotted in Figure \ref{varicella}. 
Seasonality and trend phenomena are not unique to disease incidence processes. For example, oscillatory patterns exist in different kinds of biomedical signals such as electrocardiogram signal, respiratory signal, blood pressure, circadian rhythm, etc., and it is well known that the period of the oscillation varies according to time and the time-varying period contains plentiful information about the underlying physiological {\it dynamics} \citep{hrv,Benchetrit:2000,Wysocki2006,golombek_rosenstein:2010,wang:2010,Lin_Hseu_Yien_Tsao:2011,wu:2011}. For example, in Section \ref{sleep} we demonstrate the correlation between the varying frequency of respiratory signal and sleep stage. Further examples in astronomy, climatology and econometrics have been extensively discussed in the literature \citep{Hall_Reimann_Rice:2000, Nott_Dunsmuir:2002, Oh_Nychka_Brown_Charbonneau:2004, Genton_Hall:2007, Park_Ahn_Hendry_Jang:2011, Rosen_Stoffer_Wood:2009, Pollock:2009,Bickel_Kleijn_Rice:2008}. 

There are abundant modern methods available to accommodate
both seasonality and trend in a time series, for example, seasonal autoregressive integrated moving average (SARIMA) \citep{Brockwell_Davis:2002}, and Trigonometric Box-Cox transform, ARMA errors, Trend and Seasonal components (TBATS) \citep{DeLivera_Alysha_Hyndman_Snyder:2011}, among others. Many of the existing models, including SARIMA and TBATS, focus on forecasting. Thus, although they are useful in many fields, they have some limitations when used to analyze historical data. First, it is hard for the methods to capture {\it the dynamical behavior} of the seasonality such as its diminishment or changes in the period and/or strength which, as mentioned above, is one of the main features in a time series in many fields. However, showing these features scientifically is not guaranteed by existing methods. Indeed, the global (parametric) model assumptions on the seasonality are often too restrictive for real world data. Violations of the parametric assumptions, in particular the fixed periods assumption, can cause not just large bias in the seasonality estimation but also spurious oscillations in the trend estimate,  as any unexplained seasonal dynamics would then have to be attributed to the trend. This has been an important issue in analysis of trend and seasonality. Another limitation follows immediately from the first one. In the conventional parametric methods, the seasonality analysis depends on the whole time series, rendering the methods sensitive to the length of the time series. For example, the result obtained from a $10$-year time series may be different from that obtained from a $5$-year sub-series.  Moreover, there may exist multiple seasonal components, which cannot be handled by most of the existing methods. Above all, the random errors are often dependent and may be heteroscedastic, and the innovations may be non-Gaussian, making the problem even more complicated and challenging. 

To tackle the above mentioned difficulties faced by existing methods and to understand more accurately about the dynamics of a system, we introduce a phenomenological nonparametric model which captures and offers a natural decomposition of the dynamical seasonal components, the changing trend, and the heteroscedastic, dependent errors. Each of the seasonal components has time-varying amplitude and time-varying frequency with bounded derivatives. We prove that functions in this  nonparametric seasonality class are identifiable up to a ``model bias'' of controlled order i.e. any two different representations of a member are the same up to the model bias. To the best of our knowledge, this identifiability result is so far the first theoretical justification for nonparametric modeling of multi-component dynamical seasonality. It is important in its own right. The trend is is characterized as a smooth function with ``very low-frequency,'' for example, any finite-degree polynomial belongs to this class. 
The random error term, or the noise, is modeled as a generalized stationary random process coupled with a heteroscedastic variance function. Notice that the nonparametric time series model we introduce is new; in particular, the functional classes for the seasonal component have not been considered before. 

To isolate the meaningful seasonal components based on noisy observations coming from the new model, we focus attention on a data-adaptive algorithm referred to as Synchrosqueezing transform (SST) \citep{daubechies_maes:1996, daubechies_lu_wu:2010}, originally designed to analyze dynamical seasonality without contamination of noise or coupling with trend.  
We prove that the SST method provides not only adaptive and robust estimators but also an easy visualization of the dynamical seasonal components. Hence, even in the presence of trend and (heteroscedastic) dependent errors, we can determine if any hazard occurs in disease incidence and obtain information about the underlying physiological dynamics based on the time-varying frequency in biomedical signals. In addition, since the seasonality is modeled nonparametrically and the SST algorithm is local in nature, it is insensitive to the length of the observed time series, in the sense that the estimate of a seasonal component does not change much as time goes even when it is dynamical. Furthermore, since our model allows multiple seasonal components, it can extract information about both the high- and low-frequency periodic components, without  even knowing the time-varying periods. The only requirements are the sampling interval between two successive observations is small enough for us to observe the high-frequency periodicity, and the length of the time interval is large enough for us to see the low-frequency ones.

After the oscillatory components are isolated from the time series, we can extract the trend and approximate accurately the heteroscedastic, dependent errors using the residuals obtained by subtracting from the time series the trend and seasonality estimates. Subsequently we can conduct further investigations on the error process, which are relevant in many directions including forecasting. Here we mention that the trend has not been taken into account in previous work on SST analysis of dynamical seasonality. In addition, \cite{brevdo_fuckar_thakur_wu:2012} limited the random error term to Gaussian white noise with the noise level being much smaller than the error in modeling the seasonality. By contrast, the assumptions we make on the error process are much milder, requiring only smoothness of the modulating variance function and boundedness of the power spectrum of the stationary component. Finally, while existing works investigated properties of SST in the continuous-time setup, we address the problems in both continuous- and discrete-time setups. The latter is equally relevant from the modeling viewpoint and is somehow more important from the practical viewpoint i.e. in reality we can only observe the process at discrete sampling time-points.  

This paper is organized as follows. In Section \ref{section:MandM}, both continuous- and discrete-time models are proposed to model processes that contain dynamical seasonal components, trend, and heteroscedastic, dependent errors. Functional classes used to model signals with dynamical seasonal components are introduced, and identifiability theory of the functional classes is provided. In Section \ref{method:theory} the Synchrosqueezing transform approach to separating the seasonal components, trend and dependent errors is introduced and theoretically studied. Also given is numerical implementation of the method. In Section \ref{simulation}, we demonstrate the efficacy of the proposed method  and compare it with the TBATS model by analyzing a series of simulations. In Section \ref{data} two medical examples are provided: incidence time series of varicella and herpes zoster extracted from the Taiwan's National Health Insurance Research Database (NHIRD) published by the National Health Research Institute of Taiwan, and respiratory signals of patients in a sleep stage study conducted in the Sleep Center of the Chang Gung Memorial Hospital in Taoyuan, Taiwan. 
Section \ref{discussion} contains discussions and  open problems. The Supplementary contains proofs of the theoretical results  and some further materials, including a formulae for the power spectrum of a general order continuous-time ARMA process and results of additional numerical studies.

\section{Model}\label{section:MandM}

Oscillatory signals are ubiquitous in many scientific fields. Seasonality, the term wildly used in the public health, economics, etc, describes the oscillatory behavior of a given time series $Y_t$. Here we list some interesting problems commonly raised in analyzing the oscillatory signals.
\begin{enumerate}
\item[Q1:] If there are multiple oscillatory components inside the signal, how to detect and estimate them? 
\item[Q2:] If there exists a trend in addition to the oscillatory components, how to extract it?
\item[Q3:] If the pattern of the oscillatory components is time-varying, how to quantify/identify it?
\item[Q4:] Since the length of the observed data elongates as time goes, how sensitive is the estimator to the length of the observed time series? 
\item[Q5:] If the errors across different time points are dependent, or if the variance of the error changes according to time, is the estimator robust to such dependent, heteroscedastic errors?    
\end{enumerate}

\subsection{Some related approaches}
In this subsection, we briefly review some existing  time series models that take into account seasonality, and discuss the need for a new model that can answer the above questions simultaneously.

\subsubsection{Trigonometric Seasonality and Trend Model}
A simple model for the seasonality reads:
\[
Y_t=f_t+T_t+\Phi_t,
\]
where $f_t$ is a deterministic, periodic function modeling the seasonality, $T_t$ is a deterministic function modeling the trend, and $\Phi_t$ is a stationary random process modeling the dependent errors. 
In the above model, $f_t$ is usually taken as a trigonometric function:
\begin{equation}\label{section:model:fourier}
f_t = \sum_{k=1}^KA_k\cos(2\pi \xi_kt),
\end{equation} 
where $K\in\NN$, and for each $k=1,\ldots, K$, $A_k>0$, $\xi_k>0$ and we call $A_k$ {\it the amplitude}, $\xi_kt$ {\it the phase function}, and $\xi_k$ {\it the frequency} of the $k$-th seasonal component. 
A special case, which consists of single-component seasonality i.e. $K=1$ and a linear trend, was considered in \cite{Pollock:2009}.

\subsubsection{BATS, TBATS}
To resolve the limitations of the Seasonal Autoregressive Integrated Moving Average (SARIMA) model \citep{Brockwell_Davis:2002}, and to improve the traditional single seasonal exponential smoothing methods, recently \cite{DeLivera_Alysha_Hyndman_Snyder:2011} introduced two algorithms BATS and TBATS. In particular, the BATS model includes a Box-Cox transformation, ARMA errors, and $n$ seasonal patterns as follows:
\begin{align*}
& y^{(\omega)}_t=\left\{\begin{array}{lc} (Y_t^\omega-1)/\omega &\mbox{when } \omega\neq 0\\ \log Y_t&\mbox{when }\omega=0 \end{array}\right.,\\
& y^{(\omega)}_t=\ell_{t-1}+\phi b_{t-1}+\sum_{i=1}^n s_{t-m_i}^{(i)}+d_t\\
& \ell_t=\ell_{t-1}+\phi b_{t-1}+\alpha d_t\\
&b_t = (1-\phi)b+\phi b_{t-1}+\beta d_t\\
&s^{(i)}_t=s_{t-m_i}^{(i)}+\gamma_i d_t
\end{align*}
where $\omega\in\RR$ is the Box-Cox transform parameter, $m_1,\ldots,m_n$ denote the {\it constant} seasonal periods, $b$ is the long-run trend, $\{d_t\}$ is an ARMA$(p,q)$ process with Gaussian white noise innovation process with zero mean and constant variance, and for $t=1,\ldots,T$, $\ell_t$ is the local stochastic level, $b_t$ is the short-term trend and $s^{(i)}_t$ is the stochastic level of the $i$-th seasonal component. 
Note that, while both the periodic component and trend in model (\ref{section:model:fourier}) are deterministic, in the BATS (and TBATS) model both $s^{(i)}_t$ and $\ell_t$ are coupled with the single-source error $d_t$.

Furthermore, \cite{DeLivera_Alysha_Hyndman_Snyder:2011}  introduced the trigonometric representation of the seasonal components based on Fourier series:
\begin{align}
&s^{(i)}_t=\sum_{j=1}^{k_i}s_{j,t}^{(i)}\label{model:TBATS:si}\\
&s_{j,t}^{(i)}=s^{(i)}_{j,t-1}\cos\lambda_j^{(i)}+s^*_{j,t-1}\sin\lambda_j^{(i)}+\gamma_1^{(i)}d_t\label{model:TBATS:sjt}\\
&s_{j,t}^{*(i)}=-s^{(i)}_{j,t-1}\sin\lambda_j^{(i)}+s^*_{j,t-1}\cos\lambda_j^{(i)}+\gamma_2^{(i)}d_t,\label{model:TBATS:sjtstar}
\end{align} 
where $\gamma_1^{(i)}$ and $\gamma_2^{(i)}$ are smoothing parameters, $\lambda_j^{(i)}=2\pi j/m_i$, 
$s_{j,t}^{*(i)}$  is the stochastic growth in the level of the $i$-th seasonal component, and $k_i$ is the number of harmonic components needed for the $i$-th seasonal component. The TBATS model is thus defined by replacing $s^{(i)}_t$ in the BATS model by (\ref{model:TBATS:si}). 
Note that we can rewrite (\ref{model:TBATS:sjt}) and (\ref{model:TBATS:sjtstar}) in the complex form:
\begin{align}\label{model:TBATS:sjtAM}
r^{(i)}_{j,t}e^{i\theta_{j,t}^{(i)}}=e^{i\lambda^{(i)}_j}r^{(i)}_{j,t-1}e^{i\theta_{j,t-1}^{(i)}}+c^{(i)}d_t,
\end{align}
where $r^{(i)}_{j,t}e^{i\theta_{j,t}^{(i)}}=s_{j,t}^{(i)}+is_{j,t}^{*(i)}$, $e^{i\lambda_j^{(i)}}$ is viewed as a rotational transformation with matrix form $\left[\begin{array}{cc}\cos \lambda_j^{(i)} & \sin \lambda_j^{(i)}\\-\sin \lambda_j^{(i)} & \cos \lambda_j^{(i)}\end{array}\right]$, and $c^{(i)}=\gamma_1^{(i)}+i\gamma_2^{(i)}$. Thus, (\ref{model:TBATS:si}) models the $i$-th seasonal component as the real part of $k_i$ complex harmonic functions. Note that the ``frequency'' $\lambda_j^{(i)}$ is fixed all the time. 

We can view BATS/TBATS as a model ``decoupling'' the seasonality and the trend which are modeled together in SARIMA. In TBATS the decoupled seasonality is modeled by introducing trigonometric functions. We refer to \cite{DeLivera_Alysha_Hyndman_Snyder:2011} for detailed discussion.

\subsubsection{Limitations}
As useful as the above models are, there are some limitations. In particular, questions Q3, Q4 and Q5 cannot be fully answered so far. First of all, the seasonal periods in the models are all fixed, which means that question Q3 cannot be answered fully (although TBATS allows stochastic seasonal components). 
Moreover, since the parameter estimation depends on the whole observed time series, the length of the observed time series plays a role in the result, thus it is not easy to answer Q4. As for Q5, although dependent errors can be handled, it is not guaranteed for heteroscedastic errors. In this paper we focus on relieving these limitations as well as providing an alternative approach in order to answer properly all of the questions Q1 -- Q5.

\subsection{New Models}

To answer questions Q1--Q5 properly, consider the following phenomenological model generalizing model (\ref{section:model:fourier}). 
First, consider a periodic function satisfying the following format:
\begin{equation}\label{section:model:new_model}
f(t) = \sum_{k=1}^KA_k(t)\cos(2\pi\phi_k(t)),
\end{equation}
where $A_k(t),\phi'_k(t)>0$ for all $t\in\RR$, $k=1,\ldots,K$.
We call $A_k(t)$ the {\it amplitude modulation}, $\phi_k(t)$ the {\it phase function} and $\phi'_k(t)$ the {\it instantaneous frequency} of the $k$-th component in $f$.
The mathematical rigor of the expression (\ref{section:model:new_model}) is discussed in Section \ref{identifiability}. 

The seasonality model (\ref{section:model:new_model}) is parallel to model (\ref{section:model:fourier}) -- the notion  ``instantaneous frequency'' and ``amplitude modulation'' in (\ref{section:model:new_model}) are respectively parallel to ``frequency'' and ``amplitude'' in (\ref{section:model:fourier}). The time-varying nature of $\phi'_k(t)$ and $A_k(t)$ allows us to capture the momentary behavior of the system. 
More precisely, in model (\ref{section:model:fourier}) the physical interpretation of frequency is how many oscillations are generated in a unit time period, while in model (\ref{section:model:new_model}) the physical interpretation of the time-varying function $\phi_k'(t)$ is how many oscillations are generated in an infinitesimal time period. Similar argument applies to the amplitude modulation function $A_k(t)$, which describes the ``amplitude modulation'' of the oscillation. Note that the frequency of the $k$-th harmonic function in (\ref{section:model:fourier}), $\xi_k$, is simply the derivative of $\xi_kt$, and $\xi_k>0$ can be interpreted as a constant function defined on $\RR$.

In general, the expression (\ref{section:model:new_model}) is not unique and the identifiability issue exists. In \cite{Genton_Hall:2007}, which studies single-compenoent seasonality without trend, this issue is noticed and both $A_1(t)$ and $\phi_1(t)$ are modeled by some parametric forms to avoid the identiÞability problem.  However, in general, parametric assumptions are restrictive and often need to be validated using nonparametric model-checking methods which are unfortunately unavailable so far. By contrast, we consider the following functional classes in which only nonparametric assumptions are imposed on the amplitude modulation functions and instantaneous frequency functions.
\begin{defn}[Intrinsic Mode Functions class $\mathcal{A}^{c_1,c_2}_{\epsilon}$]
For fixed choices of  $0<\epsilon\ll 1$ and $\epsilon\ll c_1<c_2<\infty$, the space $\mathcal{A}^{c_1,c_2}_{\epsilon}$ of  \textit{Intrinsic Mode
Functions (IMFs)} consists of functions $f:\RR\to\RR$, $f\in C^1(\RR)\cap L^\infty(\RR)$ having the
form \begin{equation}
f(t)=A(t)\cos(2\pi\phi(t)),
\end{equation}
where $A:\RR\to\RR$ and $\phi:\RR\to\RR$ satisfy the following conditions for all $t\in\RR$:
\begin{equation}\label{Aeps:cond:1}
A\in C^1(\RR)\cap L^\infty(\RR),~\inf_{t\in\RR}A(t)>c_1,~\sup_{t\in\RR}A(t)<c_2,
\end{equation}
\begin{equation}\label{Aeps:cond:2}
 \phi\in C^2(\RR),~\inf_{t\in\RR}\phi'(t)>c_1, ~\sup_{t\in\RR}\phi'(t)<c_2,
\end{equation}
\begin{equation}\label{Aeps:cond:3}
|A'(t)|\leq\epsilon \phi'(t),\quad |\phi''(t)|\leq\epsilon \phi'(t).
\end{equation}
\end{defn}

\begin{defn}[Superpositions of IMFs]\label{DefBClass}
Fix $0<d<1$. The space $\mathcal{A}^{c_1,c_2}_{\epsilon,d}$ of \textit{superpositions
of IMFs} consists of functions $f$ having the form
\begin{equation*}
f(t)=\sum_{k=1}^{K}f_{k}(t)
\end{equation*}
for some finite $K>0$ and for each $k=1,\ldots,K$, $f_{k}(t)=A_{k}(t)\cos(2\pi \phi_{k}(t))\in\mathcal{A}^{c_1,c_2}_{\epsilon}$ such that $\phi_{k}$ satisfies 
\begin{equation}\label{Aepsd:cond:1}
\phi'_k(t)>\phi'_{k-1}(t)\,\,\,\mbox{ and }\,\,\,\phi'_{k}(t)-\phi'_{k-1}(t)\geq d [\phi'_k(t)+\phi'_{k-1}(t)].
\end{equation}
\end{defn}
Intuitively, a signal in the $\mathcal{A}^{c_1,c_2}_{\epsilon}$ class is a single-component, periodic function having slowly varying amplitude modulation and instantaneous frequency. The $\mathcal{A}^{c_1,c_2}_{\epsilon,d}$ functional class models signals having multiple oscillatory components, with $A_k(t)$ and $\phi_k(t)$ characterizing together the dynamics of the $k$-th component. Here, condition (\ref{Aepsd:cond:1}) is needed because of the dyadic separation nature of the continuous wavelet transform. Note that both $\mathcal{A}^{c_1,c_2}_{\epsilon}$ and $\mathcal{A}^{c_1,c_2}_{\epsilon,d}$ are not vector spaces. 
The identifiability theory for functions belonging to $\mathcal{A}^{c_1,c_2}_{\epsilon}$ and $\mathcal{A}^{c_1,c_2}_{\epsilon,d}$ is given in Section \ref{identifiability}.

We then model a random process $Y(t)$ with multiple seasonal components and trend behaviors contaminated by heteroscedastic, dependent errors as below:
\begin{equation}\label{model:seasonality}
Y(t)=f(t)+T(t)+\sigma(t) \Phi(t),
\end{equation}
where the seasonality $f(t)=\sum_{k=1}^KA_{k}(t)\cos(2\pi \phi_{k}(t))$ is in $\mathcal{A}^{c_1,c_2}_{\epsilon}$ when $K=1$ and in $\mathcal{A}^{c_1,c_2}_{\epsilon,d}$ when $K>1$, the trend $T(t)$ is modeled as a $C^1$ real-valued function, $\Phi(t)$ is some stationary generalized random process (GRP) \citep{Gelfand:1964}, and $\sigma(t)>0$ so that $\sigma\in C^\infty\cap L^\infty$ is a real-valued smooth function used to model the heteroscedasticity of the error term. For example, $\Phi(t)$ can be taken as a continuous-time autoregressive moving average (CARMA) random process of order $(p,q)$, where $p,q\geq 0$, defined in Section \ref{section:CARMA} of the Supplementary. The heteroscedastic, dependent error process $\sigma(t) \Phi(t)$ specified here is a special case of the locally stationary processes introduced in \cite{Priestley:1965}, and fitting the time-varying spectra has been considered before \citep{Dahlhaus:1997,Hallin:1978,Hallin:1980,Rosen_Stoffer_Wood:2009}. Note that in our model the trend $T(t)$ is nonparametric in nature. Intuitively, it should be a very ``low-frequency'' function, and we specify the precise conditions in Section \ref{theory}. Combining these with the nonparametric models $\mathcal{A}^{c_1,c_2}_{\epsilon}$ and $\mathcal{A}^{c_1,c_2}_{\epsilon,d}$ for the dynamical seasonality, model (\ref{model:seasonality}) provides a general decomposition of the trend, the seasonality and the error process. It enables us to extract information for all of  $A_k(t)$, $\phi_k(t)$ and $\phi_k'(t)$, $k=1,\ldots,K$, and $T(t)$ based on observations on $Y(t)$ generated by model (\ref{model:seasonality}), which we will discuss in Section \ref{method:theory}.

 In practice, we can only access the continuous-time process $Y(t)$ given in model (\ref{model:seasonality}) on discrete sampling time-points $n\tau$, where $n\in\ZZ$ and $\tau>0$ is the sampling interval. So, we consider the following discrete-time model 
\begin{equation}\label{model:seasonality:discrete}
Y_n=f(n\tau)+T(n\tau)+\sigma(n\tau)\,\Phi_n, \, n\in \ZZ,
\end{equation}
where $f$, $T$ and $\sigma$ are as in model (\ref{model:seasonality}), and $\Phi_n$, $n\in\ZZ$, is a zero-mean stationary time series which can be taken as, for example, an ARMA time series discussed in Section \ref{section:ARMA} of the Supplementary. We can interpret (\ref{model:seasonality:discrete}) as a model for a discrete-time process $\{Y_n\}_{n\in\ZZ}$ in which the deterministic seasonality and trend are contaminated by the heteroscedastic errors $\sigma(n\tau)\,\Phi_n$, $n\in\ZZ$. 

Compared with the existing models like SARIMA and TBATS, our models (\ref{model:seasonality}) and (\ref{model:seasonality:discrete}) not only can take care of the problem of multiple seasonal components with non-integer periods, but also can cope with the system dynamics including the unknown time-varying frequencies and time-varying amplitude modulations. In the TBATS model, although the amplitude modulation (\ref{model:TBATS:sjtAM}) and the dynamics of the seasonality (\ref{model:TBATS:si}) may change according to time, the ``changes'' in the seasonality and trend are coupled with the ARMA error term, thus both the seasonality and the trend are stochastic. On the other hand, in our new models (\ref{model:seasonality}) and (\ref{model:seasonality:discrete}), both of the seasonality and the trend are modeled as deterministic terms which are independent of the error term. Hence the SARIMA and TBATS models are essentially different from ours.

\subsection{Identifiability of functions in $\mathcal{A}^{c_1,c_2}_{\epsilon}$, $\mathcal{A}^{c_1,c_2}_{\epsilon,d}$}\label{identifiability}
It is well known that for a given function there might be more than one representation. For example, a purely harmonic function can also be represented as a function having time-varying amplitude and time-varying phase:
$$
\cos(2\pi t)=(1+a(t))\cos(2\pi (t+b(t))),
$$
where $b'(t)$ and $a(t)$ might be ``large'' compared with $1$. Which of the two representations is ``good'' depends on the problem, and different representations lead to different interpretations. 
Thus, we start from asking the following question:
\newline\newline
\indent{\it Q: given a function $f(t)=A(t)\cos(2\pi\phi(t))\in\mathcal{A}^{c_1,c_2}_\epsilon$, how much can the different representations in $\mathcal{A}^{c_1,c_2}_\epsilon$ for $f$ differ from each other? }
\newline\newline
This is the {\it identifiability problem} we face when we introduce the $\mathcal{A}^{c_1,c_2}_\epsilon$ functional class. In the following theorem we claim that the amplitude modulation function $A(t)$, the instantaneous frequency function $\phi'(t)$ and the phase function $\phi(t)$ in {\it all} the different representations of a function in $\mathcal{A}^{c_1,c_2}_\epsilon$ can only differ from each other up to a smooth, small model bias of order $\epsilon$.  In this sense we say that a function in $\mathcal{A}^{c_1,c_2}_\epsilon$ is identifiable up to a model bias of order $\epsilon$. The proof of this theorem is postponed to the Supplementary.
\begin{thm}[Identifiability of single-component seasonality]\label{theorem:identifiability:single}
Suppose  that \\ $a(t)\cos\phi(t)\in \mathcal{A}^{c_1,c_2}_\epsilon$ can be represented in a different form which is also in $\mathcal{A}^{c_1,c_2}_\epsilon$, that is,
\begin{equation}\label{observation:identifiability:lemma:1}
a(t)\cos\phi(t)=A(t)\cos\varphi(t)\in \mathcal{A}^{c_1,c_2}_\epsilon.
\end{equation}
Define $t_m\in\RR$, $m\in\ZZ$, so that $\phi(t_m)=(m+1/2)\pi$, $s_m\in\RR$, $m\in\ZZ$, so that $\phi(s_m)=m\pi$, $\alpha(t)=A(t)-a(t)$, and $\beta(t)=\varphi(t)-\phi(t)$. Then $\alpha\in C^2(\RR)$, $\beta\in C^1(\RR)$, $\alpha(t_m)=0$ $\forall m$, $\beta(s_m)\geq 0$ $\forall m$ and $\beta(s_m)=0$ if and only if $\alpha(s_m)=0$. Moreover, we have $|\alpha'(t)|\leq 3\pi\epsilon$, $|\alpha(t)|\leq \frac{4\pi^2\epsilon}{c_1}$ and $|\beta(t)|< 3\pi \epsilon$ for all $t\in\RR$.
\end{thm}

Theorem \ref{theorem:identifiability:single} consists of two conclusions. The first conclusion is that the perturbations $\alpha$ and $\beta$ must have some ``hinge points'' and so they are restricted. This property comes from the positivity condition of the instantaneous frequency and amplitude modulation functions. The second conclusion is that the absolute values of $\alpha$, $\alpha'$ and $\beta$ cannot be large. This property comes from the ``slowly varying'' conditions of the $\mathcal{A}^{c_1,c_2}_\epsilon$ functional class and the existence of the hinge points. With these properties, the definition of instantaneous frequency and amplitude modulation is rigorous in the sense that they are unique up to a negligible error when $\epsilon$ is small enough. Theorem \ref{theorem:identifiability:single}  has its own interest and further study on this topic is beyond the scope of this paper. Similarly, the identifiability issue exists for functions in the class $\mathcal{A}^{c_1,c_2}_{\epsilon,d}$, and
in the following theorem we state the identifiability theory for $\mathcal{A}^{c_1,c_2}_{\epsilon,d}$. From the theorem, we conclude that any multi-component periodic function in $\mathcal{A}^{c_1,c_2}_{\epsilon,d}$ is again identifiable up to a model bias of order $\epsilon$.
\begin{thm}[Identifiability of multiple-component seasonality]\label{theorem:identifiability:multiple}
Suppose $f(t)=\sum_{l=1}^Na_l(t)\cos\phi_l(t)\in \mathcal{A}^{c_1,c_2}_{\epsilon,d}$ can be represented in a different form which is also in $\mathcal{A}^{c_1,c_2}_{\epsilon,d}$, that is,
\begin{equation*}
f(t)=\sum_{l=1}^Na_l(t)\cos\phi_l(t)=\sum_{l=1}^MA_l(t)\cos\varphi_l(t)\in \mathcal{A}^{c_1,c_2}_{\epsilon,d},
\end{equation*}
then $M=N$, $|\phi_l(t)-\phi_l(t)|\leq E_I\epsilon$, $|\phi'_l(b)-\varphi'_l(b)|\leq E_I\epsilon$ and $|a_l(t)-A_l(t)|\leq E_I\epsilon$ for all $l=1,\ldots,N$, where $E_I>0$ is a finite universal constant depending on $c_1$, $c_2$ and $d$ defined in (\ref{proof:identifiability:multiple:error_constant}).
\end{thm}

\section{Method and Theory}\label{method:theory}

We need a method to analyze observations generated from models (\ref{model:seasonality}) and (\ref{model:seasonality:discrete}) so as to extract information for the trend $T(t)$ and the dynamics of the seasonality, $A_k(t)$, $\phi_k(t)$, and $\phi_k'(t)$, $k=1,\ldots,K$. 
Time-frequency (TF) analysis \citep{flandrin:1999} is commonly applied to analyze the signal expressed in (\ref{section:model:new_model}). Reassignment approach  \citep{flandrin:1999,Chassande-MottinDaubechiesAuger:97,Chassande-MottinAugerFlandrin:03} is a technique in TF analysis aimed at giving a more accurate estimate of $\phi'_k(t)$ from the TF representation provided by, e.g., short time Fourier transform (STFT) or continuous wavelet transform (CWT). 
However, the estimation of $A_k(t)$, the reconstruction of each component $f_k(t)$  and the robustness to noise are not guaranteed in general. 

We consider a newly developed reassignment method referred to as the Synchrosqueezing transform (SST), which was introduced to study dynamical seasonality without coupling with trend or random errors \citep{daubechies_maes:1996,daubechies_lu_wu:2010}. In this section we show that theoretically SST can be used to accurately estimate the functions $\phi_k(t)$, $\phi'_k(t)$ and $A_k(t)$ when the trend is present, and its robustness to heteroscedastic, dependent random error processes, i.e. when the data are modeled by  (\ref{model:seasonality}) or (\ref{model:seasonality:discrete}). Before stating the SST algorithm, in the following subsection we introduce some notation first.

\subsection{Notation and Background Material}
Denote by $\mathcal{S}$ the Schwartz space and let $\mathcal{S}'$ be its dual (the tempered distribution space). When $g\in\mathcal{S}'$ and $h\in\mathcal{S}$, $g(h)$ means $g$ acting on $h$. Here, sometimes we use the notation $g(h):=\int gh\ud t$ which is consistent with the case when $g$ is an integrable function.
Given a function $h\in \mathcal{S}$, its Fourier transform is defined as $\widehat{h}(\xi):=\int_{-\infty}^\infty h(t)e^{-i2\pi\xi t}\ud t$. The Fourier transform of $g\in\mathcal{S}'$ exists in the distribution sense and is defined by $\widehat{g}(h):=g(\widehat{h})$. 

Take $\psi\in\mathcal{S}$. 
For $k\in \NN\cup\{0\}$, define the following abbreviations:
\begin{align*}
\psi_{a,b}^{(k)}(x):=\frac{1}{a^{k+1/2}}\psi^{(k)}\left(\frac{x-b}{a}\right),\,\,\psi_{a,b}(x):=\psi_{a,b}^{(0)}(x)\,\,\mbox{ and }\,\,\psi_{a}^{(k)}(x):=\psi_{a,0}^{(k)}(x),
\end{align*}
where $\psi^{(k)}$ is the $k$-th derivative of $\psi$, $a>0$ and $b\in\RR$. 
Recall that the  CWT \citep{daubechies:1992} of a given $f(t)\in \mathcal{S}'$ is defined by
\begin{equation}\label{cwt}
W_f(a,b)=\int_{-\infty}^\infty f(t)\overline{\psi_{a,b}(t)}\ud t,
\end{equation}
where $a>0$ and $b\in\RR$. Here we follow the convention in the wavelet literature that $\psi$ is called the mother wavelet, $a$ means scale and $b$ means time. To ease the notation, the moments of $\psi$ are denoted as $I^{(k)}_i=\int_\RR |x|^i|\psi^{(k)}(x)|\ud x$ for $k=0,1,\ldots$.

Let $W$ be the standard Brownian motion and $D$ be the differentiation operator in the general sense. Then $DW$ is the Gaussian white noise. Denote $W^{(n)}:=D^nW$, $n\geq 1$, be the $n$-th differentiation of the standard Brownian motion in the general sense. Recall that $DW$ is a special case of generalized random process (GRP) \citep{Gelfand:1964}. Fix a GRP $\Phi$. When $\psi\in\mathcal{S}$, $\Phi(\psi)$ is understood as a random variable modeling the measurement of $\Phi$ when it is characterized by the {\it measurement function} $\psi$. 

Next we recall the notion of the power spectrum of a stationary GRP $\Phi$. The correlation functional of $\Phi$, denoted as $B_\Phi$, is:
\begin{equation}
B_\Phi(\phi,\psi):=\mathbb{E}[\Phi(\phi)\overline{\Phi(\psi)}],
\end{equation}
where $\phi,\psi\in\mathcal{S}$ \citep{Gelfand:1964}. Then, there exists a functional $B_0$ so that
\begin{equation}
B_\Phi(\phi,\psi)=(B_0,\phi\star \psi^*),
\end{equation}
where $\star$ stands for convolution. Here, $B_0$ is a generalized function of one variable which is the Fourier transform of some positive tempered measure \cite[Equation 3 above Theorem 1 in Chapter III]{Gelfand:1964}. 
Moreover, by Theorem 1 in Chapter III of \cite{Gelfand:1964}, we have 
\begin{equation}
B_\Phi(\phi,\psi)=\int\widehat{\phi}(\xi)\overline{\widehat{\psi}(\xi)}\ud \eta(\xi),
\end{equation}
where $\eta$ is the unique positive tempered measure associated with $\Phi$ so that $\widehat{\eta}=B_0$. In general, we call $\ud \eta$ {\it the power spectrum} of the GRP $\Phi$. Thus, the variance of $\Phi(\psi_{a})$, where $\psi_a(t):=\frac{1}{\sqrt{a}}\psi\big(\frac{t}{a}\big)$, $a>0$, is simply:
\begin{align}
\text{Var}(\Phi(\psi_{a}))&\,=\mathbb{E}\Phi(\psi_{a})\overline{\Phi(\psi_{a})}=\int \widehat{\psi_{a}}(\xi)\overline{\widehat{\psi_{a}}(\xi)}\ud \eta(\xi)=a\|\widehat{\psi}(a\xi)\|_{L^2(\RR,\eta)}^2.\label{proof:thm:relationship_varX:1}
\end{align}
It is clear that the variance of $\Phi(\psi_{a})$ depends on both the scale $a$ and the power spectrum. Notice that in the special case where $\Phi=DW$, $\ud\eta(\xi)=\ud \xi$ and so the variance of $\Phi(\psi_a)$ does not depend on the scale $a$. 

\subsection{Synchrosqueezing transform approach}

In this subsection, we first briefly recall the main idea of reallocation methods and introduce the synchrosqueezing transform (SST) algorithm originally developed for reconstructing the seasonal components from a signal without contamination of noise. Then we introduce the SST to cope with the case when we have noisy observations from model (\ref{model:seasonality}) or (\ref{model:seasonality:discrete}). 

Take a TF representation, denoted as $R_f:(t,\xi)\in \RR^2\to \CC$, determined by $f(t)\in\mathcal{A}^{c_1,c_2}_{\epsilon,d}$ based on a TF analysis, for example, STFT or CWT. The reassignment methods ``sharpen'' $R_f(t,\xi)$ by ``re-allocating'' the value at $(t,\xi)$ to a different point $(t',\xi')$ according to some {\it reassignment rules} \citep{flandrin:1999}.  

The SST algorithm, a special case of the reassignment method tailored to analyze a clean function $f(t)\in \mathcal{A}^{c_1,c_2}_{\epsilon,d}$ without coupling with noise, is composed of three steps.  First, choose the mother wavelet $\psi\in\mathcal{S}$ so that  $\text{supp }\widehat{\psi}\subset[1-\Delta,1+\Delta]$, where $\Delta\ll 1$, and calculate $W_f(a,b)$, the CWT of $f(t)$ as given in (\ref{cwt}). Second, calculate the function $\omega_f(a,b)$ defined on $\RR^+\times\RR$, which plays the role of the reassignment rule: 
\begin{equation}\label{alogithm:sst:ressigment}
\omega_f(a,b):=\left\{\begin{array}{ll}
\frac{-i\partial_bW_f(a,b)}{2\pi W_f(a,b)} & \mbox{when}\quad |W_f(a,b)|\neq 0;\\
\infty&\mbox{when}\quad |W_f(a,b)|=0.
\end{array}\right.
\end{equation}
By its definition, $\omega_f(a,b)$ contains abundant information about the instantaneous frequency functions in $f$.  Indeed, when $f$ is a purely harmonic function, $\omega_f(a,b)$ takes on the value of the frequency of $f$ if it is finite. We refer to \cite{daubechies_lu_wu:2010} for the details.
Third, the SST of $f(t)$ is defined by re-assigning the TF representation $W_f(a,b)$ according to the reassignment rule $\omega_f(a,b)$:
\begin{equation}\label{alogithm:sst:formula}
S^{\Gamma}_f(t,\xi):=\lim_{\alpha\to 0}\hspace{-10pt}\int\limits_{\{(a,t):~|W_f(a,t)|\geq {\Gamma}\}}\hspace{-20pt}h_\alpha(|\omega_f(a,t)-\xi|)W_f(a,t)a^{-3/2}\ud a
\end{equation}
where $(t,\xi)\in\RR\times\RR^+$, $\alpha,\Gamma>0$, $h_\alpha(t):=\frac{1}{\alpha}h(\frac{t}{\alpha})$, $h\in L^1(\RR)$, and $h_\alpha\to \delta$ weakly when $\alpha\to 0$ with $\delta$ denoting the Dirac delta function. Thus, at each time point $t$, $S^{\Gamma}_f(t,\xi)$ collects all CWT coefficients with scales $a$ at which the CWT detects a seasonal component with frequency close to $\xi$. As we will see in Theorem \ref{section:theorem:stability}, according to the reassignment rule (\ref{alogithm:sst:ressigment}), $S^{\Gamma}_f(t,\xi)$ will only have dominant values around $\phi_k'(t)$ which allows us an accurate estimate of $\phi_k'(t)$; see Figure \ref{fig:simulation:clean} for a numerical illustration. We refer to Section \ref{implementation} for details of the construction  and implementation of $\widetilde\phi'_k(t)$. To reconstruct the $k$-th component $f_k(t)=A_k(t)\cos(2\pi\phi_k(t))$ in $f$, its amplitude modulation $A_k(t)$ and phase $\phi_k(t)$, we resort to the reconstruction formula of CWT and consider the following estimators: 
\begin{align}
& \widetilde{f}^{\Gamma,\CC}_{k}(t):=\mathcal{R}_\psi^{-1}\int_{\frac{1-\Delta}{\phi_k'(t)}}^{\frac{1+\Delta}{\phi_k'(t)}}  W_f(a,t){\boldsymbol{\chi}}_{|W_f(a,t)| >\Gamma}(a) a^{-3/2} \ud a,\,\,\,\widetilde{f}^{\Gamma}_{k}(t):=\mathfrak{Re} \widetilde{f}^{\Gamma,\CC}_{k}(t),\label{alogithm:sst:reconstruction}
\end{align}
 where $\mathcal{R}_\psi:=\int \frac{\widehat{\psi}(\zeta)}{\zeta}\ud \zeta$, $\boldsymbol{\chi}$ denotes the indicator function, and $\mathfrak{Re}$ means taking the real part,
\begin{equation*}
\widetilde{A}_k(t):=|\widetilde{f}^{\Gamma,\CC}_{k}(t)|
\end{equation*}
and an estimator for $\phi_k(t)$ can then be obtained by unwrapping the phase of the complex-valued signal $\frac{\widetilde{f}^{\Gamma,\CC}_{k}(t)}{\widetilde{A}_k(t)}$. We mention that the reconstruction formulae (\ref{alogithm:sst:reconstruction}) is slightly different from that in Estimate 3.9 in \cite{daubechies_lu_wu:2010}. These formula are actually equivalent, as is shown in  its proof in the paper. In practice we find that (\ref{alogithm:sst:reconstruction}) performs slightly better numerically. Moreover, it can be applied to other time-frequency analysis techniques which provide accurate instantaneous frequency estimation. Thus we suggest it as our reconstruction formulae.

Now, consider that we have discrete-time observations of $f\in \mathcal{A}^{c_1,c_2}_{\epsilon,d}$, that is, $\boldsymbol{f}:=\{f(n\tau)\}_{n\in\ZZ}$ and $\boldsymbol{f_n}=f(n\tau)$, where $\tau>0$ is the sampling interval. In this case we model the discrete-time observations as a delta chain, $f_\tau =\tau\sum_{i\in\ZZ}f(t)\delta_{n\tau}$, where $\delta_{n\tau}$ is the delta measure at $n\tau$, which is a distribution, and plug it into (\ref{cwt}). Since $\psi\in\mathcal{S}$, the CWT $W_{f_\tau}(a,b)$ is well-defined and is equal to $\tau \sum_{m\in\ZZ} f(m\tau)\frac{1}{a^{1/2}}\psi\big(\frac{m\tau -b}{a}\big)$. This is simply the discretization of (\ref{cwt}), so for $a>0$ and $n\in\ZZ$ denote
\begin{equation}\label{algorithm:sst:discrete:CWT}
\begin{split} 
&W_{\boldsymbol{f}}(a,n\tau):=\,\tau \sum_{m\in\ZZ} \boldsymbol{f_m}\frac{1}{a^{1/2}}\psi\big(\frac{m\tau -n\tau}{a}\big),\\
&\partial_bW_{\boldsymbol{f}}(a,n\tau):=\,\tau \sum_{m\in\ZZ} \boldsymbol{f_m}\frac{1}{a^{3/2}}\psi'\big(\frac{m\tau -n\tau}{a}\big).
\end{split}
\end{equation}
Similarly, we have the discretization of (\ref{alogithm:sst:ressigment}) and (\ref{alogithm:sst:formula}), which are denoted as  
\begin{equation} \label{alogithm:sst:discrete:ressigment}
\begin{split}
&\omega_{\boldsymbol{f}}(a,n\tau):=\left\{\begin{array}{ll}
\frac{-i\partial_bW_{\boldsymbol{f}}(a,n\tau)}{2\pi W_{\boldsymbol{f}}(a,n\tau)} & \mbox{when}\quad |W_{\boldsymbol{f}}(a,n\tau)|\neq 0;\\
\infty&\mbox{when}\quad |W_{\boldsymbol{f}}(a,n\tau)|=0.
\end{array}\right.\\
&S^{\Gamma}_{\boldsymbol{f}}(n\tau,\xi):=\lim_{\alpha\to 0}\hspace{-10pt}\int\limits_{\{(a,n\tau):~|W_{\boldsymbol{f}}(a,n\tau)|\geq {\Gamma}\}}\hspace{-20pt}h_\alpha(|\omega_{\boldsymbol{f}}(a,n\tau)-\xi|)W_{\boldsymbol{f}}(a,n\tau)a^{-3/2}\ud a,
\end{split}
\end{equation}
where $\Gamma>0$ and $n\in\ZZ$. Then the estimation of $f_k$, $A_k$ and $\phi_k$, $k=1,\ldots,K$, follows immediately, for example, for $n\in\ZZ$ we have  
\begin{align}
\widetilde{f}^{\Gamma,\CC}_{k,n}:=\mathcal{R}_\psi^{-1}\int_{\frac{1-\Delta}{\phi_k'(n\tau)}}^{\frac{1+\Delta}{\phi_k'(n\tau)}}  W_{\boldsymbol{f}}(a,n\tau){\boldsymbol{\chi}}_{|W_{\boldsymbol{f}}(a,n\tau)| >\Gamma} a^{-3/2} \ud a,\,\,\,\widetilde{f}^{\Gamma}_{k,n}:=\mathfrak{Re} \widetilde{f}^{\Gamma,\CC}_{k,n}.\label{alogithm:sst:discrete:reconstruction}
\end{align}
 The above discussions concern the cases when the observations are not contaminated with noise and do not contain trend. If we observe $Y$ satisfying model  (\ref{model:seasonality}), we simply replace $f$  in (\ref{alogithm:sst:ressigment}), (\ref{alogithm:sst:formula}) and (\ref{alogithm:sst:reconstruction}) by $Y$, and we consider the following the trend estimator:
\begin{equation*}
\widetilde{T}:=Y-\mathfrak{Re}\int_{\frac{1-\Delta}{c_2}}^{\frac{1+\Delta}{c_1}} W_Y(a,b)a^{-3/2}\ud a,
\end{equation*}   
which is a GRP in general. Suppose we have discrete-time observations $\boldsymbol{Y}=\{\boldsymbol{Y_n}\}_{n\in\ZZ}$ from model (\ref{model:seasonality:discrete}) so that $\boldsymbol{Y_n}=f(n\tau)+T(n\tau)+\sigma(n\tau)\Phi_n$, where $\tau>0$ is the sampling interval. Then we replace $\boldsymbol{f}$ in (\ref{algorithm:sst:discrete:CWT}), (\ref{alogithm:sst:discrete:ressigment}) and (\ref{alogithm:sst:discrete:reconstruction}) by $\boldsymbol{Y}$, and then reconstruct the trend at time $n\tau$, $n\in\ZZ$, by the following estimator:
\begin{equation*}
\widetilde{T}_n:=\boldsymbol{Y_n}-\mathfrak{Re}\int_{\frac{1-\Delta}{c_2}}^{\frac{1+\Delta}{c_1}} W_{\boldsymbol{Y}}(a,n\tau)a^{-3/2}\ud a.
\end{equation*}

\subsection{Theory}\label{theory}

In this subsection we state theoretical properties of the above SST approach and summarize its advantages over other methods, especially for our purpose, the seasonality and trend analysis.
Before stating the robustness theorems, we make the following assumptions and define some notation.

\begin{description}
\item[\textbf{Assumption (A1):}] Assume the mother wavelet $\psi\in\mathcal{S}$ is chosen such that $\widehat{\psi}\subset[1-\Delta,1+\Delta]$, where $\Delta<d/(1+d)$, and $\mathcal{R}_\psi=1$. Also assume that $T:\RR\to\RR$ is in $C^1(\RR)$ so that its Fourier transform exists in the distribution sense, and $|T(\psi_{a,b})|,\,|T'(\psi_{a,b})|\leq C_T\epsilon$ for all $b\in\RR$ and $a\in (0,\frac{1+\Delta}{c_1}]$, for some $C_T\geq0$.
\item[\textbf{Assumption (A2):}] Suppose the power spectrum $\ud \eta$ of the given GRP $\Phi$ satisfies $\int (1+|\xi|)^{-2l}\ud\eta<\infty$ for some $l>0$. Also assume $\sigma\in C^\infty$ so that $\|\sigma\|_{L^\infty}\ll 1$ and $\epsilon_\sigma:=\max_{\ell=1,\ldots,\max\{1,l\}}\{\|\sigma^{(\ell)}\|_{L^\infty}\}\ll 1$, and $\var\Phi(\psi)=1$. 
\item[\textbf{Notation (N1):}] Denote by $E_0$ the universal constant depending on the moments of $\psi$ and $\psi'$, $c_1,c_2$ and $d$ defined in (\ref{proof:definition_error_constant_E0})  in the Supplementary. Denote by $E_T$ the universal constant depending on the moments of $\psi$ and $\psi'$, $C_T,c_1,c_2$ and $d$ defined in  (\ref{proof:definition:ET}) in the Supplementary.  They are related to the model bias introduced by the model $\mathcal{A}^{c_1,c_2}_{\epsilon,d}$.
 Denote by $E_i$, $i=1,\ldots,6$, constants depending on the power spectrum of $\Phi$, $c_1,c_2,d$ and the zeros and first moments of $\psi^{(k)}$, $k=1,\ldots,l+1$. These constants are related to the error process and are specified in (\ref{proof:definition_error_constant_E1E2}), (\ref{proof:definition_error_constant_E3E4}) and (\ref{proof:definition_error_constant_E5E6}) in the proof of Theorem \ref{section:theorem:stability}.
\end{description}
 
We now state the robustness property of SST when the seasonality plus trend signal is contaminated by an ``almost'' stationary GRP. The proof is postponed to the Supplementary.

\begin{thm}\label{section:theorem:stability}
Suppose $Y(t)$ follows model (\ref{model:seasonality}) and assumption  (A1) and (A2) hold. 
Then, when $\epsilon$ is small enough, for each $b\in\mathbb{R}$ and $\gamma>1$ we have the following results.
\begin{enumerate}
\item[(i)] For each $a\in [\frac{1-\Delta}{c_2},\frac{1+\Delta}{c_1}]$, with probability higher than $1-\gamma^{-2}$, we have
\begin{align}
\Big|W_Y(a,b)&-\sum_{l=1}^KA_l(b)e^{i2\pi \phi_l(b)}\sqrt{a}\overline{\widehat{\psi}\left(a\phi'_l(b)\right)}\Big|\leq \gamma(E_1\sigma(b)+E_2\epsilon_\sigma)+E_0\epsilon;\nonumber
\end{align}
\item[(ii)] for each $a\in Z_k(b):=\big[\frac{1-\Delta}{\phi_k'(t)},\frac{1+\Delta}{\phi_k'(t)}\big]$ with $|W_{f}(a,b)|>\gamma(E_3\sigma(b)+E_4\epsilon_\sigma)+(C_T+1)\epsilon^{1/3}$, where $k=1,\ldots,K$, with probability higher than $1-\gamma^{-2}$, we have
\begin{align}
|\omega_Y(a,b)-\phi'_k(b)|\leq&\, \frac{\gamma(E_3\sigma(b)+E_4\epsilon_\sigma)+E_0\epsilon}{|W_f(a,b)|};\nonumber
\end{align}
\item[(iii)] with probability higher than $1-\gamma^{-2}$, for $k=1,\ldots,K$ we have
\begin{align}
&\big|\widetilde{f}^{E_1\sigma(b)+E_2\epsilon_\sigma,\CC}_{k}(b)-A_k(b)e^{2\pi i \phi_k(b)}\big| \leq \, \big[\gamma(E_5\sigma(b)+E_6\epsilon_\sigma)+E_0\epsilon\big]\Delta. \nonumber
\end{align}
\item[(iv)] the trend estimator satisfies
\[
|\EE\widetilde{T}(\varphi_{h,b})-T(b)|\leq 2E_T \|\varphi\|_{1}\epsilon
\]
when $h>0$ is small enough, where $\varphi\in\mathcal{S}$ and $\varphi_{h,b}(t):=\frac{1}{h}\varphi(\frac{b-t}{h})\to \delta_b$ in the distribution sense when $h\to 0$.
\end{enumerate}
\end{thm}

We have some remarks about the theorem. 
\begin{enumerate}
\item From (ii) in Theorem \ref{section:theorem:stability}, it is clear that when we estimate the instantaneous frequency $\phi_k'(b)$, the larger the $|W_f(a,b)|$ is the smaller the estimation error is.
\item Notice that each of the error bounds in Theorem \ref{section:theorem:stability} consists of two terms. The first term is related to the error process $\sigma\Phi$ and its heteroskedasticity modeled by $\sigma(b)$, and the second term is related to the model bias $\epsilon$ introduced when we use $\mathcal{A}^{c_1,c_2}_{\epsilon}$/$\mathcal{A}^{c_1,c_2}_{\epsilon,d}$ to model the seasonality. In the first term, $\sigma(b)$  and $\sigma_\epsilon$  respectively characterize the noise level and the ``non-stationarity'' of the error process. Notice that when $\sigma(b)$ dominates $\epsilon$ and $\epsilon_\sigma$, the estimation error is of the same order as $\sigma(b)$, which is the standard deviation of the error process. In this case, if we have interest in recovering the error process, we may choose a smaller $\Delta$ so that the estimation error is smaller than $\sigma(b)$ and hence with high probability the realization of the error process at hand can be approximated accurately. We also comment that in the proof of Theorem  \ref{section:theorem:stability} the autocorrelation structure of the error process is not used; by taking this structure into consideration, we may achieve a better estimation scheme.
\item Notice that in result (iv) of Theorem \ref{section:theorem:stability} we do not give a probability bound statement for the trend estimator as we do for the seasonal components estimators, instead we only bound the pointwise estimation bias. Although the mother wavelet $\psi$ help us to ``measure'' a GRP when we estimate the seasonal components, in general the value of a GRP cannot be accessed at any point and have to ``measure'' it by a Schwartz function $\varphi$. Specifically, to access the value of $T$ at time $b$ we can take $h\to0$ in $T(\varphi_{b,h})$; however, as $h\to 0$, the GRP $\widetilde{T}$ blow up. Therefore, we have no access to the variance of $\widetilde{T}$. For example, consider the case $T=0$ and $\Phi$ is the Gaussian white noise and ask if we are able to confirm $T(b)=0$ given the GRP $T+\Phi$. By a direct calculation, $\EE(\Phi(\varphi_{b,h}))=0$ but $\textup{var}(\Phi(\varphi_{b,h}))=\int|\widehat{\varphi}(h\xi)|^2\ud\xi$, which blows up as $h\to0$.  We emphasize that in the discrete setup, which is the case in practice, this problem disappears and we are able to provide a probability bound statement, given in Theorem \ref{section:theorem:ARMAstability}, (iv).
\end{enumerate}

Note that Theorem \ref{section:theorem:stability} may not be always applicable. For example, although all discretized CARMA GRP are ARMA time series \citep{Brockwell:2010}, not every ARMA$(p,q)$ time series can be embedded into a CARMA$(p',q')$ GRP for some $p',q'\in \NN\cup\{0\}$ \citep{Brockwell:1995}. 
The following theorem states the robust property of the SST approach when the data come from model (\ref{model:seasonality:discrete}). First, we introduce the following additional assumptions and further notation.

\begin{description}
\item[\textbf{Assumption (A3):}] For the time series $\boldsymbol{Y}=\{Y_n\}_{n\in\ZZ}$ in model (\ref{model:seasonality:discrete}), we assume $A_k(t)\in C^2(\RR)$ and $\sup_{t\in\RR}|A_k''(t)|\leq \epsilon c_2$ for all $k=1,\ldots,K$. We also assume in addition to Assumption (A1) that $T\in C^2$ so that $|T''(\psi_{a,b})| \leq C_T\epsilon$ for all $b\in\RR$ and $a\in (0,\frac{1+\Delta}{c_1}]$.     Suppose the sampling interval $\tau$ satisfies $0<\tau\leq \frac{1-\Delta}{(1+\Delta)c_2}$. 
\item[\textbf{Assumption (A4):}] Assume $\textup{var}(\Phi_n)=1$ and $\sigma\in C^2$ so that $\|\sigma\|_{L^\infty}\ll 1$, $\epsilon_\sigma:=\max\{\|\sigma^{(1)}\|_{L^\infty},\|\sigma^{(2)}\|_{L^\infty} \}\ll 1$.

\item[\textbf{Notation (N3):}] Denote by $E_{\tau,0}$ the universal constant depending on $\tau$, the moments of $\psi$ and $\psi'$, $c_1$, $c_2$ and $d$. Denote by $E_{T,0}$ the universal constant  depending on $\tau$, the moments of $\psi$ and $\psi'$, $C_T, c_1$, $c_2$ and $d$.  These constants are related to the trend and the model bias introduced by the $\mathcal{A}^{c_1,c_2}_{\epsilon,d}$ class and are influenced by the sampling interval $\tau$.      Let $E_{\tau,i}$, $i=1,\ldots,6$, $E_{T,1}$ and $E_{T,2}$ be constants depending on $\tau$, the  given error process $\Phi$, $c_1,c_2,d$ and the zeros and first moments of $\psi$ and $\psi'$. 

\end{description}

\begin{thm}\label{section:theorem:ARMAstability}
Take a time series $\boldsymbol{Y}=\{Y_n\}_{n\in\ZZ}$ following model (\ref{model:seasonality:discrete}) and suppose assumptions (A1), (A3) and (A4) hold. 
Then, if $\epsilon$ is small enough, for each $n\in\mathbb{Z}$ and $\gamma>1$ we have the following results.
\begin{enumerate}
\item[(i)] For each $a\in [\frac{1-\Delta}{c_2},\frac{1+\Delta}{c_1}]$, with probability higher than $1-\gamma^{-2}$, we have
\begin{align}
\Big|W_{\boldsymbol{Y}}(a,n\tau)-\sum_{l=1}^KA_l(n\tau)e^{i2\pi \phi_l(n\tau)}\sqrt{a}\overline{\widehat{\psi}\left(a\phi'_l(n\tau)\right)}\Big|\leq\, \gamma(E_{\tau,1}\sigma(n\tau)+\tau^2E_{\tau,2}\epsilon_\sigma)+E_{\tau,0}\epsilon;\nonumber
\end{align}
\item[(ii)] for each $a\in Z_k(n\tau):=\big[\frac{1-\Delta}{\phi_k'(n\tau)},\frac{1+\Delta}{\phi_k'(n\tau)}\big]$ with $|W_{f}(a,n\tau)|>\gamma(E_{\tau,3}\sigma(n\tau)+\tau^2E_{\tau,4}\epsilon_\sigma)+(C_T+1)\epsilon^{1/3}$, where $k=1,\ldots,K$, with probability greater than $1-\gamma^{-2}$, we have
\begin{align}
|\omega_{\boldsymbol{Y}}(a,n\tau)-\phi'_k(n\tau)|\leq&\, \frac{\gamma(E_{\tau,3}\sigma(n\tau)+E_{\tau,4}\epsilon_\sigma)+E_{\tau,0}\epsilon}{|W_{\boldsymbol{f}}(a,n\tau)|};\nonumber
\end{align}
\item[(iii)] with probability higher than $1-\gamma^{-2}$, for $k=1,\ldots,K$ we have
\begin{align}
&\big|\widetilde{f}^{E_{\tau,1}\sigma(n\tau)+\tau^2E_{\tau,2}\epsilon_\sigma,\CC}_{k,n}-A_k(n\tau)e^{2\pi i \phi_k(n\tau)}\big| \leq \, \big[\gamma(E_{\tau,5}\sigma(n\tau)+E_{\tau,6}\epsilon_\sigma)+E_{\tau,0}\epsilon\big]\Delta. \nonumber
\end{align}
\item[(iv)] with probability higher than $1-\gamma^{-2}$, we have
\[
\big|\widetilde{T}_n-T(n\tau)\big|\leq \gamma (E_{T,1}\sigma(n\tau) +E_{T,2}\epsilon_\sigma )+E_{T,0}\epsilon. 
\]
\end{enumerate}
\end{thm}

Comments (a), (b) and (c) given immediately after Theorem \ref{section:theorem:stability} still hold for Theorem \ref{section:theorem:ARMAstability}. But we have more comments for Theorem \ref{section:theorem:ARMAstability} regarding the discrete-time case.
\begin{enumerate}
\item[(d)] The regularity conditions on $A_l(t)$  given in Assumption (A3) are added for the purpose of demonstrating the interaction between the model bias parameter $\epsilon$ and the discretization effect, 
as can be seen from (\ref{only_place_for_regularity_of_Al}) of the Supplementary.
\item[(e)] The condition on the sampling interval $\tau $ rings a bell of the Nyquist rate. Indeed, since locally the signal oscillates in a way close to harmonics, we expect to see that its spectrum is ``essentially supported'' on the frequency range $[c_1,c_2]$. This fact can be seen in the proof of the theorem. Thus, the sampling interval $\tau$ has to be shorter than $1/c_2$ in order to avoid the aliasing effect introduced by the discretization.
\end{enumerate}

With Theorem \ref{section:theorem:stability} and Theorem \ref{section:theorem:ARMAstability}, we summarize the main properties of SST that render it suitable for use in determining the seasonality.
\begin{itemize}
\item[(P1)] Fix a harmonic function $f(t)=\sum_{k=1}^KA_k\cos(2\pi\xi_k t)$, where $A_k>0$ and $\xi_k>0$. It is well known that its Fourier transform is $\widehat{f}(\xi)=\sum_{k=1}^KA_k(\delta_{\xi_k}+\delta_{-\xi_k})/2$, where $\delta$ denotes the Dirac delta function, which leads to the time-frequency representation, or time-varying spectrum\footnote{Note that we use $A_k$ instead of $A_k^2$ to simply the discussion.}, $R_f(t,\xi)=\frac{1}{2}\sum_{k=1}^KA_k\delta_{\xi_k}(\xi)$ when $\xi>0$. In this case, the time-varying spectrum does not depend on time indeed. Ideally, given a function $f(t)=\sum_{k=1}^KA_k(t)\cos(2\pi\phi_k(t))$ so that $A_k(t)>0$ and $\phi'_k(t)>0$ for $t\in\RR$, we would expect to have the ``time-varying spectrum'' $R_f(t,\xi)$ reading like $R_f(t,\xi)=\frac{1}{2}\sum_{k=1}^KA_k(t)\delta_{\phi'_k(t)}(\xi)$. This expectation can be fulfilled to some extent according to (ii) in Theorem \ref{section:theorem:stability} and Thereom \ref{section:theorem:ARMAstability} when $f\in\mathcal{A}^{c_1,c_2}_{\epsilon,d}$. Indeed, it tells us that the SST provides an approximation to this ``ideal spectrum'' when $f\in\mathcal{A}^{c_1,c_2}_{\epsilon,d}$ since the value of $|S_f^\gamma(t,\xi)|$ is dominant only if $(t,\xi)$ is close to $(t,\phi_k'(t))$. 
This property further allows an easy visualization of the instantaneous frequency $\phi'_k(t)$, if the seasonality exists. Similarly, the intensity of the dominant value reflects the value of the amplitude modulation $A_k(t)$.
\item[(P2)] SST is an invertible transformation in the sense that we can reconstruct each component of $f(t)$ accurately, as is shown in Theorem \ref{section:theorem:stability} and Theorem \ref{section:theorem:ARMAstability}. Once the existence of seasonality is confirmed by SST and the period of time within which the seasonality exists, this property allows us to recover the seasonal oscillation, for example, of the epidemic system so that we can determine when the incidence of the disease is highest.
\item[(P3)] It follows from Theorem \ref{section:theorem:stability} and Theorem \ref{section:theorem:ARMAstability} that the existence of the trend modeled in (\ref{model:seasonality}) and (\ref{model:seasonality:discrete}) do not interfere with the seasonality estimation. This property allows us to estimate the trend even when seasonality exists, which is important since in some situations the main focus is trend estimation and the seasonality is regarded as a nuisance parameter. Also note that a smooth function $T\in C^\infty\cap \mathcal{S}'$ so that its Fourier transform $\widehat{T}$ is compactly supported in $\big(-\frac{1-\Delta}{1+\Delta}c_1,\frac{1-\Delta}{1+\Delta}c_1\big)$ is a special case of what we consider in the theory. Indeed, for such a trend function we have by the Plancheral theorem $\int T(t)\psi_{a,b}(t)\ud t=\widehat{T}(\sqrt{a}\widehat{\psi}(a\xi)e^{i2\pi\xi b})=0$ for all $a\in (0,\frac{1+\Delta}{c_1}]$. 
\item[(P4)] Properties (P1)--(P3) are robust to the existence of the heteroscedastic, dependent errors in both the continuous- and discrete- models (\ref{model:seasonality}) and (\ref{model:seasonality:discrete}), as the requirements on the error process are mild. For example, $\Phi_n$ in  (\ref{model:seasonality:discrete}) can be taken as ARMA errors. Also, by Lemma \ref{lemma:powerspectrum} in the Supplementary, a stationary CARMA$(p,q)$ process, where $p,q\geq 0$, satisfies the conditions on its power spectrum given in Assumption (A1) for some $l\geq 0$, so Theorem \ref{section:theorem:stability} applies when $\Phi$ is taken as a stationary CARMA$(p,q)$ process. 
\item[(P5)] Since the estimation procedure is local in nature, it is insensitive to the length of the observed time series, and so it can answer partially Q4. 
\item[(P6)] The constants appearing in the estimation errors, for example those defined in (N1)--(N4), depend only on the higher order moments of the chosen mother wavelet $\psi$ but not on the profiles (or shape) of $\psi$. Thus, choice of the mother wavelet is not crucial to ensure properties (P1)--(P5). In this sense, we say that the method is adaptive. Indeed, one can even show that CWT is not essential in the whole algorithm in the sense that the variational approach is possible \citep{daubechies_lu_wu:2010}. Furthermore, the reconstruction formula (iii) in Theorem \ref{section:theorem:stability} (and in Theorem \ref{section:theorem:ARMAstability}) can be viewed as an {\it adaptive} bandpass filter which removes the energy of the noise out of the range of interest.  
\end{itemize}


\subsection{Numerical Implementation}\label{implementation}

Here we summarize how we numerically implement SST based on discretization. We refer the readers to \cite{brevdo_fuckar_thakur_wu:2012} for further details of implementing SST and its application to paleoclimatic data. 
Given a time series $\boldsymbol{Y}:=\{Y_n\}_{n=1}^N$ consisting of either a discretization of a process $Y$ satisfying model (\ref{model:seasonality}), with $\tau>0$ as the sampling interval, or observations from the discrete-time model (\ref{model:seasonality:discrete}). 
To prevent boundary effects, we pad $\boldsymbol{Y}$ on both sides (using, e.g., reflecting boundary conditions) so that its length is $N'=2^{L+1}$, where  $L$ is the minimal integer such that $N'>N$.  We use the same notation $\boldsymbol{Y}$ to denote the padded signal. Notice that although it works well in practice, doing so is not the optimal solution in coping with the boundary effect, but is only for our convenience. Denote the numerical implementation of CWT based on $\boldsymbol{Y}$ by an $N'\times n_a$ matrix $\widetilde{W}_{\boldsymbol{Y}}$ with the discretization interval $\Delta_a$ in $\log_2(a)$. To be more precise, we discretize the scale axis $a$ by $a_j=2^{j/n_v}\Delta t$, $j=1,\ldots, Ln_v$, where the ``voice number'' $n_v$ is a user-defined parameters that affects the number of scales we work with. In practice we choose $n_v=32$.
Also, denote the numerical implementation of SST based on $\boldsymbol{Y}$ by an $N'\times n_\xi$ matrix $\widetilde{S}^\Gamma_{\boldsymbol{Y}}$, where $n_\xi=\lfloor\frac{\frac{1}{2\tau}-\frac{1}{N'\tau}}{\Delta_\xi}\rfloor$ is the number of the discretization of the frequency domain $[\frac{1}{N'\tau},\frac{1}{2\tau}]$ by equally spaced intervals of length $\Delta_\xi=\frac{1}{N'\tau}$. In formula (\ref{alogithm:sst:formula}) and (\ref{alogithm:sst:reconstruction}), the number $\Gamma$ plays the role of a thresholding parameter. When the random error is Gaussian white noise, we may follow the suggestion provided in \cite{brevdo_fuckar_thakur_wu:2012} to choose $\Gamma$. 

With the implemented CWT and SST, we first estimate $\phi_k'$ by fitting a discretized curve $c^*\in Z_{n_\xi}^{N'}$, where $Z_{n_\xi}=\{1,\ldots,n_\xi\}$, to the dominant area of $\widetilde{S}^\Gamma_{\boldsymbol{Y}}$ by maximizing the following functional:
\begin{align}
c^*&=\max_{c\in Z_{n_\xi}^{N'}}\Big[ \sum_{m=1}^{N'} \log\left(\frac{|\widetilde{S}^\Gamma_{\boldsymbol{Y}}(c(m),m)|}{\sum_{i=1}^{n_\xi}\sum_{j=1}^{N'}|\widetilde{S}^\Gamma_{\boldsymbol{Y}}(j,i)|}\right)-\lambda\sum_{m=2}^{N'} |c(m)-c(m-1)|^2\Big], \label{algorithm:sst}
\end{align}
where the user-defined parameter $\lambda$ determines the ÒsmoothnessÓ of the resulting curve estimate. The main motivation of maximizing this functional is actually curve fitting. The first term is fitting a curve on the TF plane so that the SST over the curve is maximized.
However, the fitted $c^*$ might be wildly deviated if we do not impose any regularity condition on it, so we add the  penalty term to enforce the smoothness of $c^*$, that is, the larger $\lambda$ is, the smoother the curve $c^*$ is. 
Then for $n=1,\ldots,N'$ we can calculate the estimator of $\phi'_k(n\tau)$ by
\begin{equation}\label{algorithm:if}
\widetilde{\phi}_k'(n\tau):=\frac{c^*(n)}{N'\tau},
\end{equation}
and the estimators for the $k$-th seasonal component, its amplitude modulation and phase are respectively computed by: 
\begin{align}
&\widetilde{f}_k(n\tau):=\mathfrak{Re} \widetilde{f}^\CC_k(n\tau), \,\, \mbox{where } \,\widetilde{f}^\CC_k(n\tau):=\mathcal{R}_\psi^{-1}\frac{1}{\Delta_a}\sum_{i=\lfloor\frac{1-\Delta}{c^*(n)}\rfloor}^{\lceil\frac{1+\Delta}{c^*(n)}\rceil} \widetilde{W}_{\boldsymbol{Y}}(n,i){\boldsymbol{\chi}}_{|\widetilde{W}_{\boldsymbol{Y}}(n,i)|>\Gamma}2^{i\Delta_a},\nonumber
\end{align}
$\widetilde{A}_k(n\tau):=|\widetilde{f}^\CC_k(n\tau)|$ and $\widetilde{\phi}_k(n\tau)$ as the unwrapped  phase function of $\frac{\widetilde{f}^C_k(n\tau)}{|\widetilde{f}^\CC_k(n\tau)|}$.  
In addition, we estimate the trend at time $n\tau$, $n=1,\ldots,N'$, by
\begin{equation*}
\widetilde{T}(n\tau):=Y_n-\mathfrak{Re}\mathcal{R}_\psi^{-1} \frac{1}{\Delta_a}\sum_{i=\lfloor \frac{(1-\Delta)c_1}{(1+\Delta)\Delta_a}\rfloor}^{n_a} \widetilde{W}_{\boldsymbol{Y}}(n,i)2^{i\Delta_a}.
\end{equation*}
Here we remark that the above estimators may be noisy to some extent since they are pointwise in nature, and we can apply some smoothing techniques to the above preliminary estimates in order to obtain more stable estimators. For example, we may also consider another reconstruction formula, also equipped with the CWT \citep{daubechies_lu_wu:2010}:
\[
f(t)=C_\psi\int^\infty_{-\infty}\int_0^\infty W_f(a,b)a^{-5/2}\psi\Big(\frac{t-b}{a}\Big)\ud a\ud b,
\]
where $C_\psi$ is the constant for the reconstruction and the integration with respect to $b$ helps to smooth the reconstruction estimator in time. 
We will not get into these numerical details in this paper, however.

\section{Simulated Examples}\label{simulation}

To demonstrate the capability of SST to detect dynamical seasonality and other properties discussed in Section \ref{theory}, we tested it and compared it with the TBATS model on two simulation examples. We compared SST with TBATS in that, to the best of our knowledge,  TBATS is so far the algorithm closest to our purpose in seasonality analysis. Since TBATS is not designed for seasonality with time-varying periods, 
in the first simulation example the seasonality is composed of multiple pure trigonometric components (thus no dynamics exists in the seasonality). In the second simulation setting we consider the seasonality modeled by $\mathcal{A}^{c_1,c_2}_{\epsilon,d}$ and show the main difference between SST and TBATS. The code for implementation of SST is in the authors' homepage\footnote{\url{http://www.math.princeton.edu/~hauwu}}. We called the R forecast package to run TBATS\footnote{\url{http://robjhyndman.com/software/forecast/}}. We ran the simulation and data analysis on a macbook having $4$GB $1333$ MHz DDR3 ram, $1.7$ GHz Intel Core i5 CPUs.

\subsection{Simulation settings}

Define the following two functions modeling the seasonality:
\begin{align}
&s_{1,1}(t):=2.5\cos(2\pi t),\,\,s_{1,2}:=3\cos(2\pi^2 t)\nonumber\\
&s_1(t):=s_{1,1}(t)+s_{1,2}(t)\nonumber
\end{align}
and
\begin{align}
&A_1(t):=2+0.5(1+0.1\cos(t))\arctan(t-13),\,\,A_2(t):=3.5\boldsymbol{\chi}_{[0,7.5]}(t)+2\boldsymbol{\chi}_{(7.5,10]}(t)\nonumber\\
&\phi_1(t):=t+0.1\sin(t),\,\,\phi_2(t):=3.4t-0.02t^{2.3}\nonumber\\
&s_{2,1}(t):=A_1(t)\cos\big(2\pi\phi_1(t)\big),\,\,s_{2,2}:=A_2(t)\cos\big(2\pi\phi_2(t)\big)\nonumber\\
&s_2(t):=s_{2,1}(t)+s_{2,2}(t)\nonumber,
\end{align}
where $\boldsymbol{\chi}$ is the indicator function. 
Note that $s_1$ is composed of two harmonic functions and their frequencies are not integer multiple of each other, and $s_2$ models the seasonality with time-varying behavior. By definition, $s_2$ is composed of two IMFs with instantaneous frequencies $\phi_1'(t)=1+0.1\cos(t)$ and $\phi_2'(t)=3.4-0.046t^{1.3}$. 
 We considered the following two trend functions:
\begin{align}
&T_1(t):=8\Big(\frac{1}{1+(t/5)^2}+\exp(-t/10)\Big)\nonumber
&T_2(t):=2t+10\exp(-(t-4)^2/6).\nonumber
\end{align}
The clean $A_1(t)$, $A_2(t)$, $\phi_1'(t)$, $\phi'_2(t)$, $T_1(t)$, $T_2(t)$, $s_2(t)+T_1(t)$ and $s_2(t)+T_2(t)$ are shown in Figure \ref{fig:simulation:cleansig}.

\begin{figure}[ht]
\includegraphics[width=1\textwidth]{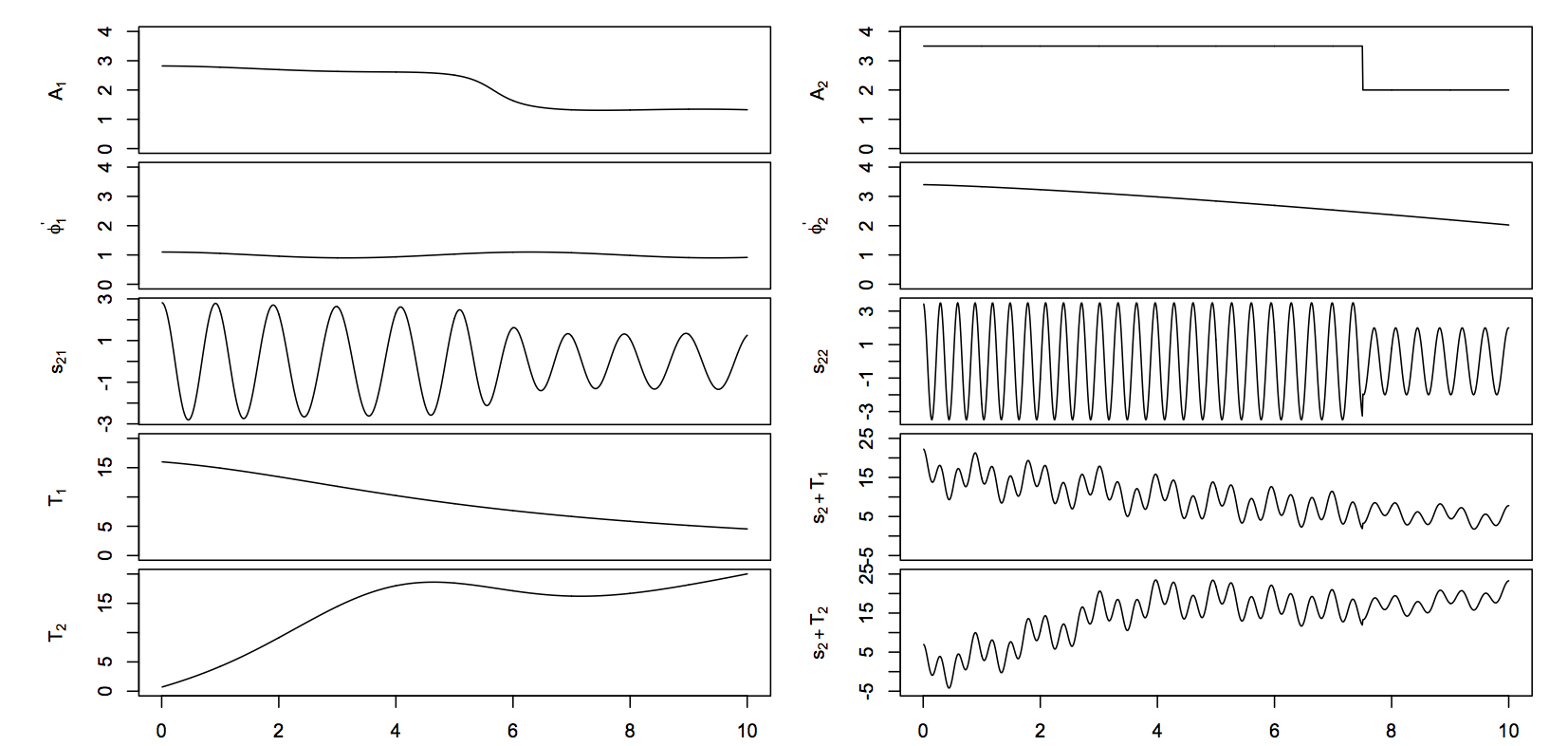}
\caption{{\it The clean signals in $s_2(t)+T_1(t)$ and $s_2(t)+T_2(t)$.} Left column, from top to bottom: $A_1(t)$, $\phi'_1(t)$, $s_{2,1}(t)$, $T_1(t)$ and $T_2(t)$; right column, from top to bottom: $A_2(t)$, $\phi'_2(t)$, $s_{2,2}(t)$, $s_2(t)+T_1(t)$ and $s_2(t)+T_2(t)$. Note that the instantaneous frequency and the amplitude modulation functions are not constant, which model the dynamics of the system.}
\label{fig:simulation:cleansig}
\end{figure}

We discretized the clean signals in the time period $[0,10]$ with the sampling interval $\tau=1/100$ so that we have $N=1000$ sampling points. The sampled time series of a given function $f$ defined on $\RR$ is denoted as $\boldsymbol{f}\in\RR^{N}$ so that the $l$-th entry of $\boldsymbol{f}$ is $f(l\tau)$ for all $l=1,\ldots, N$. Similarily, the sampled time series on $s_1$, $s_2$, $T_1$ and $T_2$ are denoted as $\boldsymbol{s_1}$, $\boldsymbol{s_2}$, $\boldsymbol{T_1}$ and $\boldsymbol{T_2}$, respectively. 

We consider the following three random processes to model the noise. The first one is 
\begin{align*}
&X_1(n):=2\sigma(n\tau)X_{\textup{ARMA1}}(n),
\end{align*}
where $\sigma(t)=1+0.1\cos(\pi t)$, and $X_{\textup{ARMA1}}$ is an ARMA(1,1) time series determined by the autoregression polynomial $a(z)=0.5z+1$ and the moving averaging polynomial $b(z)=0.4z+1$, with the innovation process taken as i.i.d. student $t_4$ random variables; the second one is
\begin{align*}
&X_2(n):=\sigma(n\tau)\big(4X_{\textup{ARMA1}}(n)\boldsymbol{\chi}_{n\in[1,N/2]}(n)+X_{\textup{ARMA2}}(n)\boldsymbol{\chi}_{n\in[N/2+1,N]}(n)\big),
\end{align*}
where $X_{\textup{ARMA2}}$ is an ARMA(1,1) time series determined by the autoregression polynomial $a(z)=-0.2z+1$ and the moving averaging polynomial $b(z)=0.51z+1$, with the innovation process taken as i.i.d. student $t_4$ random variables; 
the third one is
\begin{align*}
&X_3(n):=2X_{\textup{GARCH}}(n)
\end{align*}
where $X_{\textup{GARCH}}$ is a 
GARCH$(1,2)$ time series with ARCH coefficients $(1,0.2)$, GARCH coefficients $(0.2, 0.3)$ and N$(0,1)$ disturbances. 
The sampled time series on $X_i$ is denoted as $\boldsymbol{X_i}\in\RR^{N}$. Note that $X_1$ and $X_2$ are heteroscedastic and non-stationary. 
We then tested our algorithm on the following time series: 
\begin{align*}
\boldsymbol{Y_{j,k,\sigma_0}}:=\boldsymbol{s_2}+\boldsymbol{T_j}+\sigma_0\boldsymbol{X_k},
\end{align*}
where $j=1,2$, $k=1,2,3$ and $\sigma_0\geq0$.

\subsection{The SST tested on the clean signal $s_2+T_1$}
We started from examining the performance of SST when applied to the clean signal $\boldsymbol{s}_2+\boldsymbol{T}_1$. We took $\psi\in\mathcal{S}$ so that $\widehat{\psi}(\xi)=\exp\big(\frac{1}{((\xi-1)/0.3)^2-1}\big)$. The results are shown in Figure \ref{fig:simulation:clean}, 
from which we have the following findings. First, the instantaneous frequency functions of $s_{2,1}$ and $s_{2,2}$ can be seen clearly from the dominant curves in the SST representation depicted in the left panel. The time-varying amplitude modulation functions are also visually clear in the SST representation: the smaller the amplitude is, the lighter the intensity of the dominant curve is. The reconstruction of each component and the trend are shown in the right column. It can be seen that except for $s_{2,2}$ near the change-point $7.5$, the reconstruction is satisfactory. Note that this kind of ``sudden changes'' in the  signal $s_{2,2}$ is not theoretically analyzed nor numerically improved in the current paper.  

\begin{figure}[ht]
\includegraphics[width=1\textwidth]{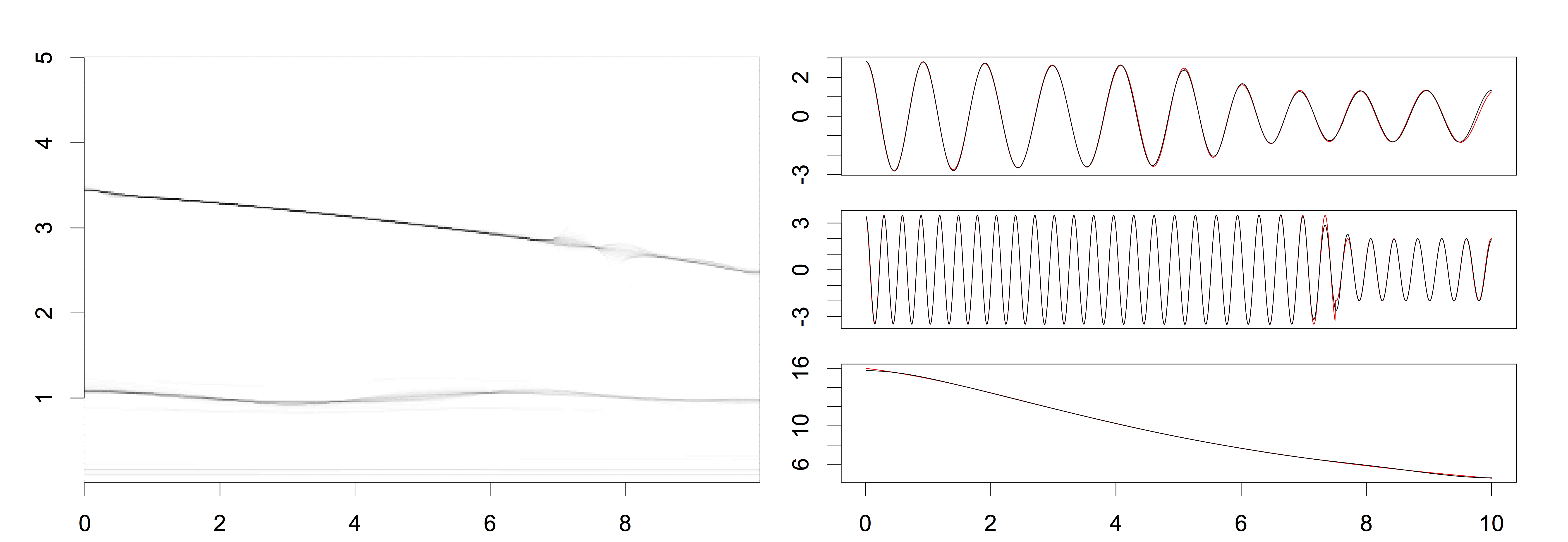}
\caption{{\it Synchrosqueezing Transform applied to the clean signal $\boldsymbol{s_2}+\boldsymbol{T_1}$.} Left: The SST results on the clean signal $\boldsymbol{s_2}+\boldsymbol{T_1}$, shown in Figure \ref{fig:simulation:cleansig}. The $x$-axis is time and the $y$-axis is frequency. Right, from top to bottom: the resulting $\widetilde{\boldsymbol{s_{2,1}}}$, $\widetilde{\boldsymbol{s_{2,2}}}$ and $\widetilde{\boldsymbol{T_1}}$ are shown as the black curves, and $\boldsymbol{s_{2,1}}$, $\boldsymbol{s_{2,2}}$ and $\boldsymbol{T_1}$ are respectively superimposed as the red curves. It is clear from the figure that the reconstruction is almost exact except for $\boldsymbol{s_{2,2}}$ at time $7.5$.}
\label{fig:simulation:clean}
\end{figure}

\subsection{Comparison of SST and TBATS on $s_1+T_1+X_1$}

We analyzed $200$ realizations of $\boldsymbol{Y_{0}}:=s_1+T_1+X_1$ using SST and TBATS. When we ran TBATS, we took the seasonal periods to be the true values $100$ and $100/\pi$.
We report the relative root average square estimation error (RRASE) to measure the estimation accuracy of the two different estimators. Denote by $\widetilde{\boldsymbol{s_{1,1}}}$, $\widetilde{\boldsymbol{s_{1,2}}}$ and $\widetilde{\boldsymbol{T_1}}$  generic estimators of $\boldsymbol{s_{1,1}}$, $\boldsymbol{s_{1,2}}$ and $\boldsymbol{T_1}$, respectively, and denote by $\widetilde{\boldsymbol{r}}:=\boldsymbol{Y_0}-\widetilde{\boldsymbol{s_{1,1}}}-\widetilde{\boldsymbol{s_{1,2}}}-\widetilde{\boldsymbol{T_1}}$ the residuals, which can be used to approximate the random errors $\boldsymbol{X_1}$. The RRASE, and its standard deviation, results given by SST and TBATS for $\boldsymbol{Y_0}$ are shown in Table \ref{table:simulation:Y0}. The computational time  (in seconds) of SST and TBATS, and its standard deviation, are reported as well.
In Figure \ref{fig:simulation:sst_recon0}, we demonstrate the results for the realization which yielded the median RRASE value among all the realizations. Note that the seasonal components in $\boldsymbol{Y_0}$ are purely harmonic and the true values of the seasonal periods were used when we ran TBATS, thus it estimated well the signals. On the other hand, SST resulted in smaller RRASE standard deviation, although it yielded larger RRASE (because it has to estimate the seasonal periods nonparametrically from the noisy data). Also, notice that the seasonal component with low frequency determined by TBATS contains artificial local extrema inside each oscillation. These artificial local extrema might lead to misinterpretation of the system dynamics and have to be taken into consideration when using TBATS for the purpose of system dynamics analysis.

\begin{figure}[ht]
\includegraphics[width=1\textwidth]{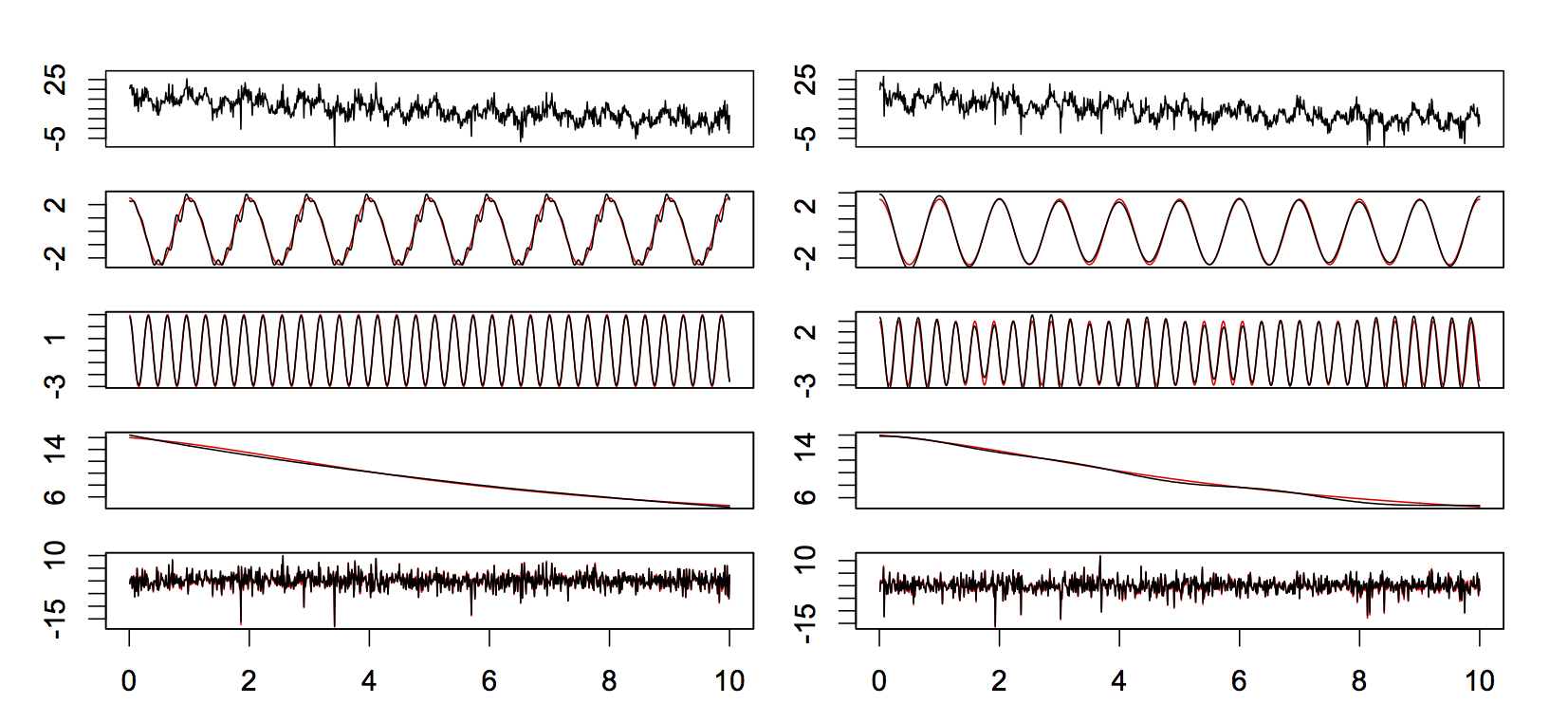}
\caption{{\it Performances of TBATS and SST on $\boldsymbol{Y_{0}}$.} Left (resp. right): The TBATS (resp. SST) results for the realization of $\boldsymbol{Y_0}$ with the median RRASE value.  
The first row shows the realization of $\boldsymbol{Y_0}$. The second to fifth rows show respectively $\widetilde{\boldsymbol{s_{1,1}}}$, $\widetilde{\boldsymbol{s_{1,2}}}$, $\widetilde{\boldsymbol{T_1}}$ and $\widetilde{\boldsymbol{r}}$ (black curves), with the $\boldsymbol{s_{1,1}}$, $\boldsymbol{s_{1,2}}$, $\boldsymbol{T_1}$ and $\boldsymbol{X_1}$ superimposed as red curves. }
\label{fig:simulation:sst_recon0}
\end{figure}

\begin{table}
\caption{\label{table:simulation:Y0}$\boldsymbol{Y_0}$: Two harmonic seasonal components and trend contaminated by nearly stationary noise.}
\centering
\fbox{%
{\small
\begin{tabular}{| c | c | c | c | c | c |}
\hline
\multicolumn{6}{|c|}{Results for $\boldsymbol{Y_0}$} \\ 
\cline{1-6}
  & $\widetilde{\boldsymbol{s_{1,1}}}$ & $\widetilde{\boldsymbol{s_{1,2}}}$ &  $\widetilde{\boldsymbol{T_1}}$ & $\widetilde{\boldsymbol{r}}$ & Time\\
\hline
SST & $0.13 \pm 0.031$ & $0.164 \pm 0.022$ &  $0.019 \pm 0.004$ & $0.162 \pm 0.017$ & $8.11 \pm 0.37$ \\  
TBATS & $0.105\pm 0.062$ &  $0.094\pm 0.057$ &  $0.025\pm 0.006$ & $0.141\pm0.034$ & $2.86\pm 0.96$ \\   
\hline
\end{tabular}
}
}
\end{table}


\subsection{The SST tested on signals with dynamics and heteroscedastic, dependent noise}
We analyzed $200$ realizations of $\boldsymbol{Y_{j,k,\sigma_0}}$, $j=1,2$, $k=2,3$ and $\sigma_0>0$, which have dynamical seasonal periods, using SST and TBATS. 
When we applied TBATS, based on the ground truth we set $2$ seasonal components with the period lengths ranging in $I_{1}=[100/1.05,100/0.95]$ and $I_{2}=[100/3.2,100/2.6]$ respectively. Notice that the chosen $I_{1}$ and $I_{2}$ respectively contain the ranges of $\phi_1'(t)$ and $\phi'_2(t)$. We divided each of $I_{1}$ and $I_{2}$ into $5$ equally spaced points, and determine the ``best'' seasonal periods based on the AIC values of the fitted TBATS models. The chosen optimal seasonal periods varied from time to time, among the 200 realizations.  
In this case, TBATS does not perform well and the obtained results are different from time to time. Indeed, since the signal does not satisfy its model assumptions, TBATS tends to smooth over sudden changes as those in  $\boldsymbol{Y_{j,k,\sigma_0}}$. 
Please see Figure \ref{fig:simulation:TBATS_recon} for results of TBATS on one realization of $\boldsymbol{Y_{1,2,1}}$ and $\boldsymbol{Y_{1,3,1}}$. 

\begin{figure}[ht]
\includegraphics[width=1\textwidth]{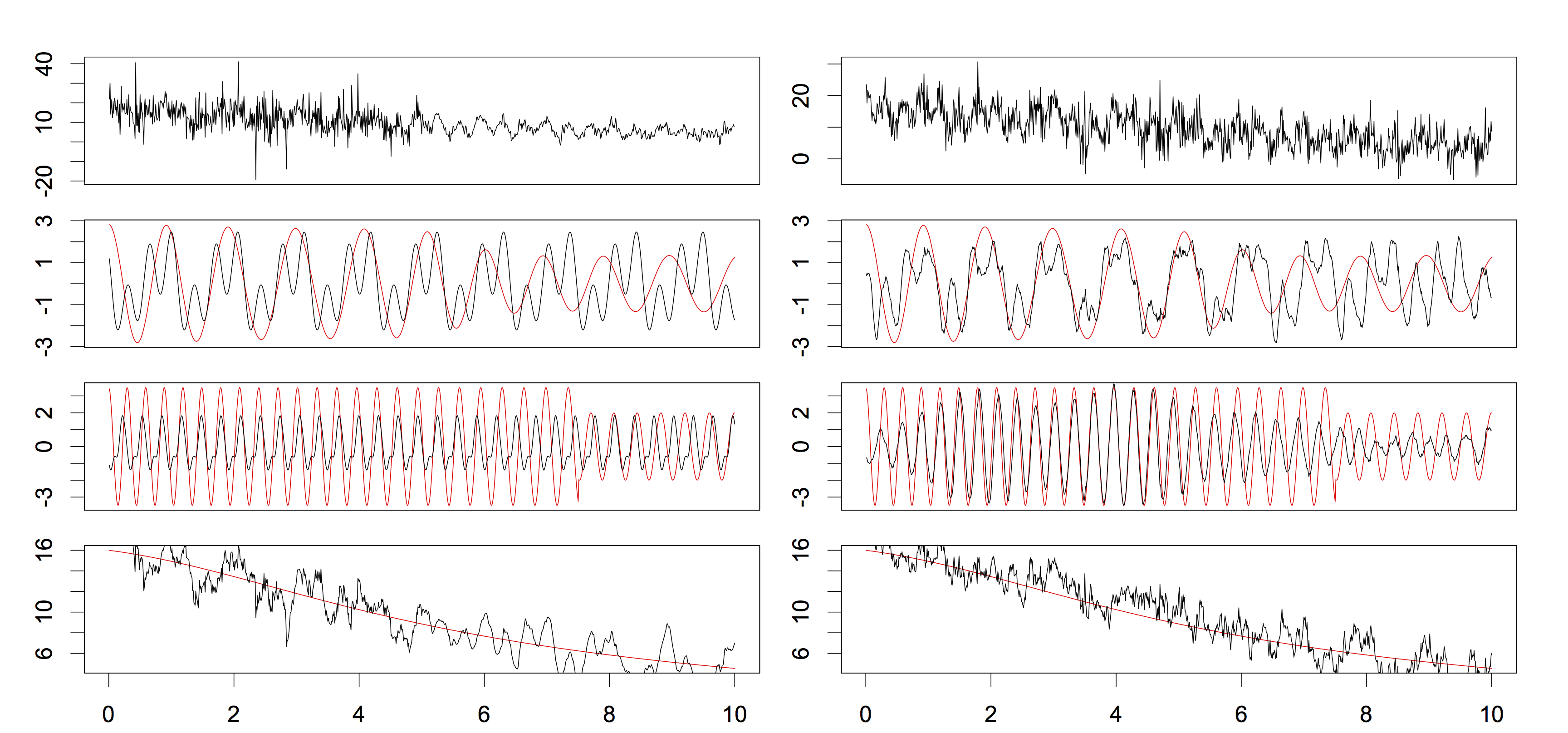}
\caption{{\it Performance of TBATS on $\boldsymbol{Y_{1,2,1}}$ and $\boldsymbol{Y_{1,3,1}}$.} Left (resp. right): The TBATS results of $\boldsymbol{Y_{1,2,1}}$ (resp. $\boldsymbol{Y_{1,3,1}}$). First row: the realizations of $\boldsymbol{Y_{1,2,1}}$ and $\boldsymbol{Y_{1,3,1}}$.  Second (resp. third) row: the $\boldsymbol{s_{2,1}}$ and $\widetilde{\boldsymbol{s_{2,1}}}$ (resp. $\boldsymbol{s_{2,2}}$ and $\widetilde{\boldsymbol{s}_{2,2}}$) are respectively represented by the red curve and the black curve. Bottom row: $\boldsymbol{T_1}$ (red) and $\widetilde{\boldsymbol{T_1}}$ (black). 
TBATS tends to fit oscillations even if they do not exist. }
\label{fig:simulation:TBATS_recon}
\end{figure}

The RRASE, and the standard deviation, of the results by SST  for $\boldsymbol{Y_{j,k,\sigma_0}}$ with two different trends ($j=1,2$), two different error types ($k=2,3$) and two different noise levels of $\sigma_0$, are shown in Table \ref{table:simulation:Y1}. The average computational time (in seconds) and the standard deviation are reported as well.
Among all the 200 realizations of $\boldsymbol{Y_{1,2,1}}$ (and $\boldsymbol{Y_{1,3,1}}$), we demonstrate the results for the realization which gave the median RRASE in Figure \ref{fig:simulation:sst_recon} (and Figure \ref{fig:simulation:sst_recon2}). 
From Table \ref{table:simulation:Y1}, we can conclude that the performance of $\widetilde{\boldsymbol{s}_{2,1}}$, $\widetilde{\boldsymbol{s}_{2,2}}$, and $\widetilde{\boldsymbol{T_j}}$ become better as the noise level $\sigma_0$ decreases, and the estimators are robust to different error types. Note that, when $\sigma_0$ decreases, the RRASE of the residual $\widetilde{\boldsymbol{r}}:=\boldsymbol{Y_{j,k,\sigma_0}}-\widetilde{\boldsymbol{s_{2,1}}}-\widetilde{\boldsymbol{s_{2,2}}}-\widetilde{\boldsymbol{T_j}}$, which approximates $\sigma_0 \boldsymbol{X_k}$, deteriorates to some extent due to the existence of the model bias $\epsilon$, as indicated by Theorem \ref{section:theorem:ARMAstability}. Indeed, the model bias is kept fixed; thus as RRASE measures the difference between $\widetilde{\boldsymbol{r}}$ and $\sigma_0 \boldsymbol{X_k}$ relative to $\sigma_0 \boldsymbol{X_k}$, it increases as $\sigma_0$ decreases.

\begin{table}
\caption{\label{table:simulation:Y1}$\boldsymbol{Y_{j,k,\sigma_0}}$, $j=1,2,k=2,3$: Performance of SST on two dynamic seasonal components and trend, contaminated by different kinds of error process with different noise levels.}
\centering
{\small
\begin{tabular}{| c | c | c | c | c | c |}
\hline
\cline{1-6}
$(j,k,\sigma_0)$  & $\widetilde{\boldsymbol{s}_{2,1}}$ & $\widetilde{\boldsymbol{s}_{2,2}}$ &  $\widetilde{\boldsymbol{T_j}}$ & $\widetilde{\boldsymbol{r}}$ & Time\\
\hline
$(1,2,0.5)$ & $0.144\pm 0.030$ & $0.133 \pm 0.020$ &  $ 0.018\pm 0.005$ & $ 0.201\pm 0.023$ & $ 9.6\pm 1.124$ \\  
\hline
$(2,2,0.5)$ & $0.151\pm0.031 $ & $ 0.133\pm 0.020$ &  $0.015 \pm 0.002$ & $0.208 \pm 0.022$ & $ 9.01\pm 0.83 $ \\  
\hline
$(1,2,1)$ & $0.263 \pm 0.062$ & $0.234 \pm 0.040$ &  $0.036 \pm 0.009$ & $0.185 \pm 0.023$ & $8.17 \pm 0.55$ \\  
\hline
$(2,2,1)$ & $0.266\pm 0.063$ & $ 0.234\pm 0.040$ &  $ 0.025\pm 0.006$ & $0.187 \pm 0.022$ & $ 9.16\pm 0.9$ \\  
\hline\hline
$(1,3,0.5)$ & $0.132\pm 0.024$ & $ 0.120\pm 0.016$ &  $ 0.016\pm 0.004$ & $ 0.205\pm 0.021$ & $ 9.61\pm1.11 $ \\  
\hline
$(2,3,0.5)$ & $0.140\pm 0.025$ & $ 0.120\pm 0.016$ &  $0.014 \pm 0.002$ & $ 0.214\pm 0.020$ & $ 9.57\pm 0.97$ \\  
\hline
$(1,3,1)$ & $ 0.239\pm 0.048$ & $0.209 \pm 0.032$ &  $0.030 \pm 0.008$ & $0.186 \pm 0.021$ & $8.82 \pm 0.33$ \\	
\hline
$(2,3,1)$ & $0.244\pm 0.049$ & $ 0.209\pm 0.032$ &  $ 0.022\pm 0.005$ & $ 0.189\pm 0.021$ & $ 8.81\pm 0.38$ \\  
\hline
\end{tabular}
}
\end{table}

\begin{figure}[ht]
\includegraphics[width=1\textwidth]{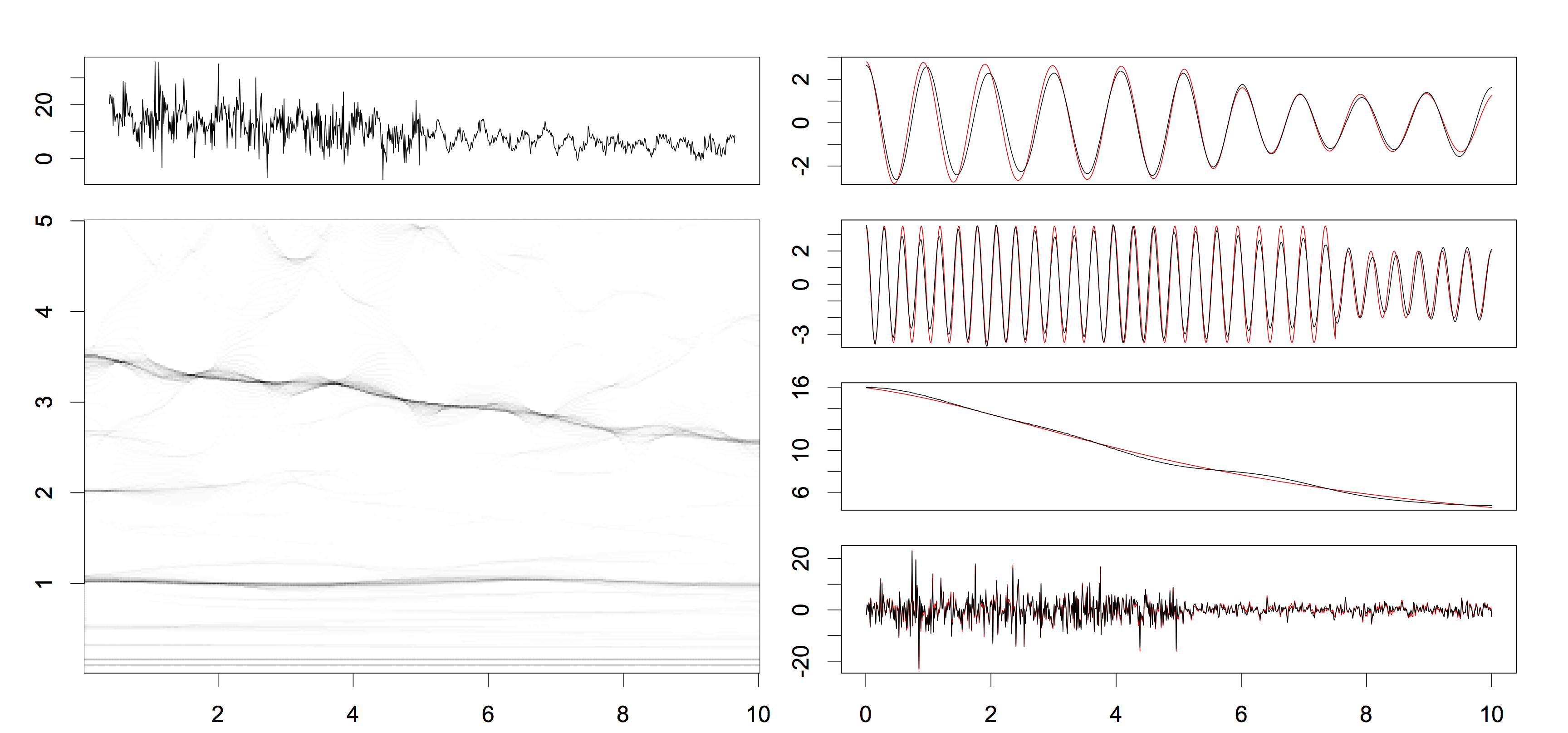}
\caption{{\it Performance of SST on $\boldsymbol{Y_{1,2,1}}$.}  Left: The realization of $\boldsymbol{Y_{1,2,1}}$ which yielded the median RRASE (upper) and the SST of that realization (lower). Visually we can see the two seasonal components, how their instantaneous frequency functions change, and the intensity of the second component changes at about $t=7.5$. Right: From top to bottom the black curves are $\widetilde{\boldsymbol{s_{2,1}}}$, $\widetilde{\boldsymbol{s_{2,2}}}$, $\widetilde{\boldsymbol{T_1}}$ and $\widetilde{\boldsymbol{X_2}}$, with the respective $\boldsymbol{s}_{2,1}$, $\boldsymbol{s}_{2,2}$, $\boldsymbol{T}_1$ and $\boldsymbol{X}_2$ superimposed as red curves.}
\label{fig:simulation:sst_recon}
\end{figure}

\begin{figure}[ht]
\includegraphics[width=1\textwidth]{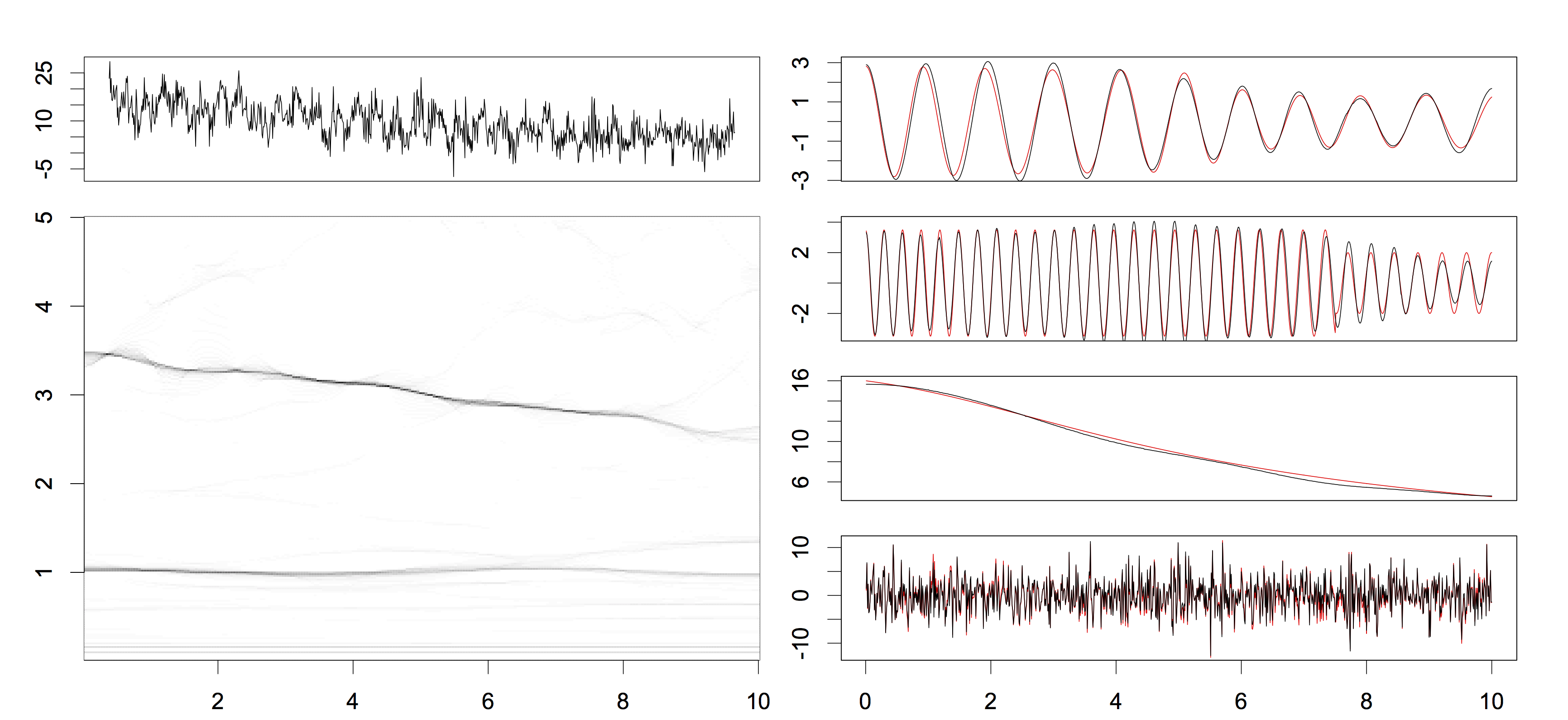}
\caption{{\it Performance of SST on $\boldsymbol{Y_{1,3,1}}$.} Left: The realization of $\boldsymbol{Y_{1,3,1}}$ which yielded the median RRASE (upper) and the SST of that realization (lower). Visually we can see the two seasonal components, how their instantaneous frequency functions change, and the second component changes at about $t=7.5$. Right, from top to bottom are $\widetilde{\boldsymbol{s_{2,1}}}$, $\widetilde{\boldsymbol{s_{2,2}}}$, $\widetilde{\boldsymbol{T_1}}$ and $\widetilde{\boldsymbol{X}}_{3}$ (black curves), with the $\boldsymbol{s}_{2,1}$, $\boldsymbol{s}_{2,2}$, $\boldsymbol{T}_1$ and $\boldsymbol{X}_{3}$ superimposed as red curves.}
\label{fig:simulation:sst_recon2}
\end{figure}

\subsection{Summary}

Notice that although SST does not outperform TBATS when tested on $\boldsymbol{Y_0}$, TBATS collapses while SST can provide reasonable results in analyzing $\boldsymbol{Y_{1,2,1}}$ and $\boldsymbol{Y_{1,3,1}}$. In the case of $\boldsymbol{Y_0}$, it should be noted that although TBATS can accurately estimate the periods of both of the seasonal components, the oscillation pattern tends to deviate from the cosine function, which is the ground truth. On the other hand, due to noise, as is expected from the results in Theorem \ref{section:theorem:ARMAstability}, the performance of SST in estimating the amplitude modulation is not as good as it is in estimating the instantaneous frequency; nonetheless the oscillation is not distorted. 
We should also point out that since $\boldsymbol{Y_0}$ satisfies the TBATS model assumptions, TBATS (a parametric model) outperforms SST (a nonparametric model) when analyzing $\boldsymbol{Y_0}$.
On the other hand, from Theorem \ref{theorem:identifiability:single} and Theorem \ref{theorem:identifiability:multiple}, the representation of a function composed of harmonic functions is not unique, and SST can only determine the seasonal components up to some degree of accuracy. Nevertheless, Table \ref{table:simulation:Y0} conveys a clear message that the SST still enjoys reasonable performance relative to the TBATS in this case. Also, we can see from Table \ref{table:simulation:Y1} that the performance of SST on the different dynamical signals, for example, $\boldsymbol{Y_{1,2,1}}$ and $\boldsymbol{Y_{1,3,1}}$, does not differ much. This  can be explained by Theorem \ref{section:theorem:ARMAstability}. On the other hand, when the seasonality is modeled by the $\mathcal{A}^{c_1,c_2}_{\epsilon,d}$ class and the noise is heteroscedastic and dependent, TBATS fails constantly since these kinds of signals and noise violate its parametric model assumptions. This explains the deteriorated performance of TBATS when tested on $\boldsymbol{Y_{1,2,1}}$ and $\boldsymbol{Y_{1,3,1}}$. Therefore, when we have a priori knowledge that the seasonal periods are not dynamical, TBATS is  preferred; however, in general we would suggest to try SST, at least to explore if the seasonal periods are dynamical or not. More simulation results of SST, including different noise types, sensitivity issue and the existence of local bursts, can be found in the Supplementary.

\section{Real Data Examples}\label{data}

\subsection{Seasonal Dynamics of Varicella and Herpes Zoster}\label{VHZ}
The varicella-zoster virus (VZV) causes two distinct diseases, varicella (chickenpox) and herpes zoster (HZ) i.e. shingles. Varicella occurs primarily in children and adolescents and features a seasonal pattern with the peak incidence happening in the winter \citep{Chan:2011}.  
In contrast, HZ occurs mostly in adults and elders who had varicella in their childhood so that VZV resides in their sensory ganglia. When the subject's immune function declines, the reactivation of the residential VZV may lead to HZ. The existence of seasonality in the incidence of HZ has been less studied and contradictory conclusions were reported \citep{Gallerani:2000, PF:2007}.

Beyond the existence of seasonality, the dynamics of the seasonality, for example, the relationship between the strength of seasonality and incidence rate of the varicella disease has been less quantified in the literature. In particular, the effect of the public vaccination program on the seasonal dynamics has been  less studied. 
As varicella is a highly contagious disease but can be effectively prevented, by 70\%--80\%, using varicella vaccines, see for example \cite{Marin:2008}, free varicella vaccination was made available to certain areas in Taiwan starting from 2003. A nationwide vaccination program was then launched in Taiwan in 2004. In the program  children aged $12$-$18$ months were encouraged to receive free vaccine against varicella and the vaccination rate was as high as 85\% in 2004 and then reached a plateau at 95\% afterward. It has been noted in  \cite{Chao:2012} that the public vaccination program was a considerable success and the incidence rate of varicella  dropped sharply by 70--80\% after 2004. For HZ, a vaccination program was promoted since 2008, and its effect on the public health is not yet clear. To investigate how the seasonal patterns of varicella and HZ were influenced by the vaccination, for example, whether the periods or amplitudes had changed or not after the vaccination, we carried out the following data analysis.

All out-patient visit records of an one-million representative cohort derived from the Taiwan's National Health Insurance Research Database (NHIRD) were analyzed. The Taiwan's NHIRD consists of de-identified and encrypted medical claims made by its 23 million inhabitants and is publicly available to medical researchers in Taiwan. This nationwide database provides accurate estimates of disease incidences because of the high coverage rate of 
Taiwan's National Health Insurance Program, which has been above 99\%  since 2000. Moreover, the Bureau of National Health Insurance (BNHI) of Taiwan performs regular cross-check and validation of the medical charts and claims to ensure the reliability of diagnosis coding, thus enabling NHIRD as a reliable resource for public health studies \citep{Chen:2011}. For research purposes, the one million (i.e. 4.3\%) representative inhabitants were randomly selected from Taiwan's 23 million inhabitants by the National Health Research Institute (NHRI) of Taiwan (\url{http://nhird.nhri.org.tw/en/index.htm}) so that the age- and gender- structures are identical to the whole population, and their full medical records were collected.

The weekly cumulative incidence rates of varicella and HZ were calculated using out-patient visit records of the one-million representative cohort dataset of NHIRD. Since varicella often occurs during pre-school and school age, all subjects aged below $10$ between 1 January 2000 and 31 December 2009 were enrolled. Every patient in this group with the first-time diagnosis of varicella, under the International Classification of Disease, 9th Version (ICD-9) code 052.x, was identified and included in the calculation of the incidence of varicella. The date of the out-patient visit was designated as each patient's incidence date of varicella. Those 
diagnosed with varicella before 1 January 2000 were excluded to ensure the validity of the incidence dates. We denote the induced weekly varicella incidence time series multiplied by $1000$ as $Y_{\textup{V}}[n]$, $n=1,\ldots,540$.
While varicella occurs mostly during childhood, HZ mainly occurs in elder patients. Subjects aged over 65 between 1 January  2000 and 31 December 2009 were enrolled and traced for their first-time diagnosis of HZ (ICD-9 code 053.x). The dates of incident of HZ were collected to calculate the incidence of HZ. Those 
diagnosed with HZ prior to 1 January 2000 were also excluded to ensure the validity of the incidence dates. We denote the induced weekly HZ incidence time series multiplied by $10000$ by $Y_{\textup{HZ}}[n]$, $n=1,\ldots,540$.

Viewing each of the incidence time series, $Y_{\textup{V}}[n]$ and $Y_{\textup{HZ}}[n]$, as the discretization of a single-component periodic function in the $\mathcal{A}^{c_1,c_2}_{\epsilon,d}$ function class contaminated by a heteroscedastic dependent error process, we analyzed them using SST and TBATS  and the results are shown in Figures \ref{varicella}, \ref{HZ} and \ref{TBATS:varicella}. The reconstructed seasonality and trend are denoted as $\widetilde{s}_{\textup{V}}[n]$ (or $\widetilde{s}_{\textup{HZ}}[n]$) and $\widetilde{T}_{\textup{V}}[n]$ (or $\widetilde{T}_{\textup{HZ}}[n]$) respectively.

\begin{figure}[ht]
\includegraphics[width=1\textwidth]{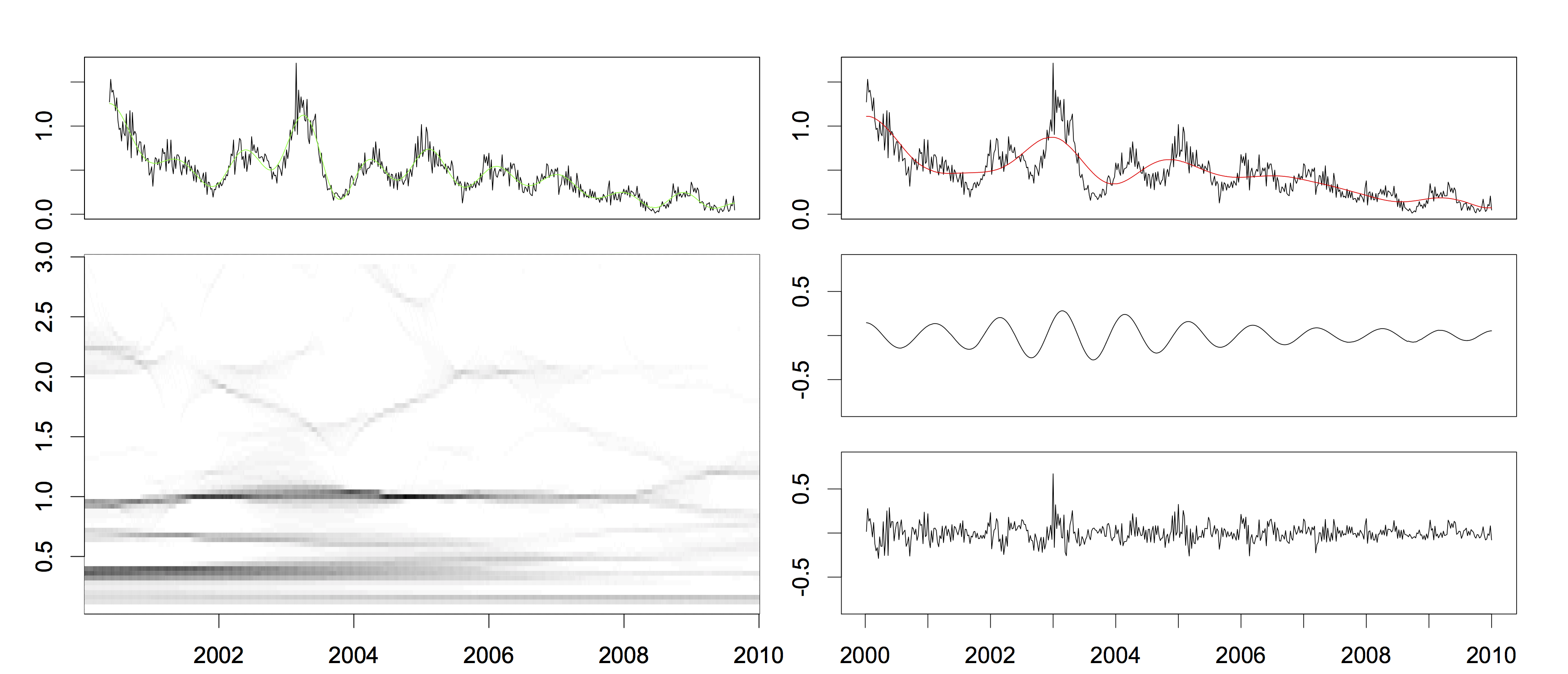}
\caption{{\it Varicella data analyzed using SST.} Upper left: The varicella incidence time series $Y_{\textup{V}}$ (black) is superimposed with $\widetilde{s}_{\textup{V}}+\widetilde{T}_{\textup{V}}$ (green). Lower left: The SST result for $Y_{\textup{V}}$, where the $y$-axis is the frequency and the intensity of the graph is the absolute value of the SST of $Y_{\textup{V}}$ given in (\ref{alogithm:sst:formula}). Right column: From top to bottom are the estimated trend $\widetilde{T}_{\textup{V}}$ (red) superimposed with $Y_{\textup{V}}$ (black), the estimated seasonality $\widetilde{s}_{\textup{V}}$, and  the residual term $Y_{\textup{V}}-\widetilde{s}_{\textup{V}}-\widetilde{T}_{\textup{V}}$. It is clear that the seasonality pattern between 2006 and 2008 is different from that in the other years. In both columns, the $x$-axes is indexed by year ranging from 2000 to 2010.}
\label{varicella}
\end{figure}

\begin{figure}[ht]
\includegraphics[width=1\textwidth]{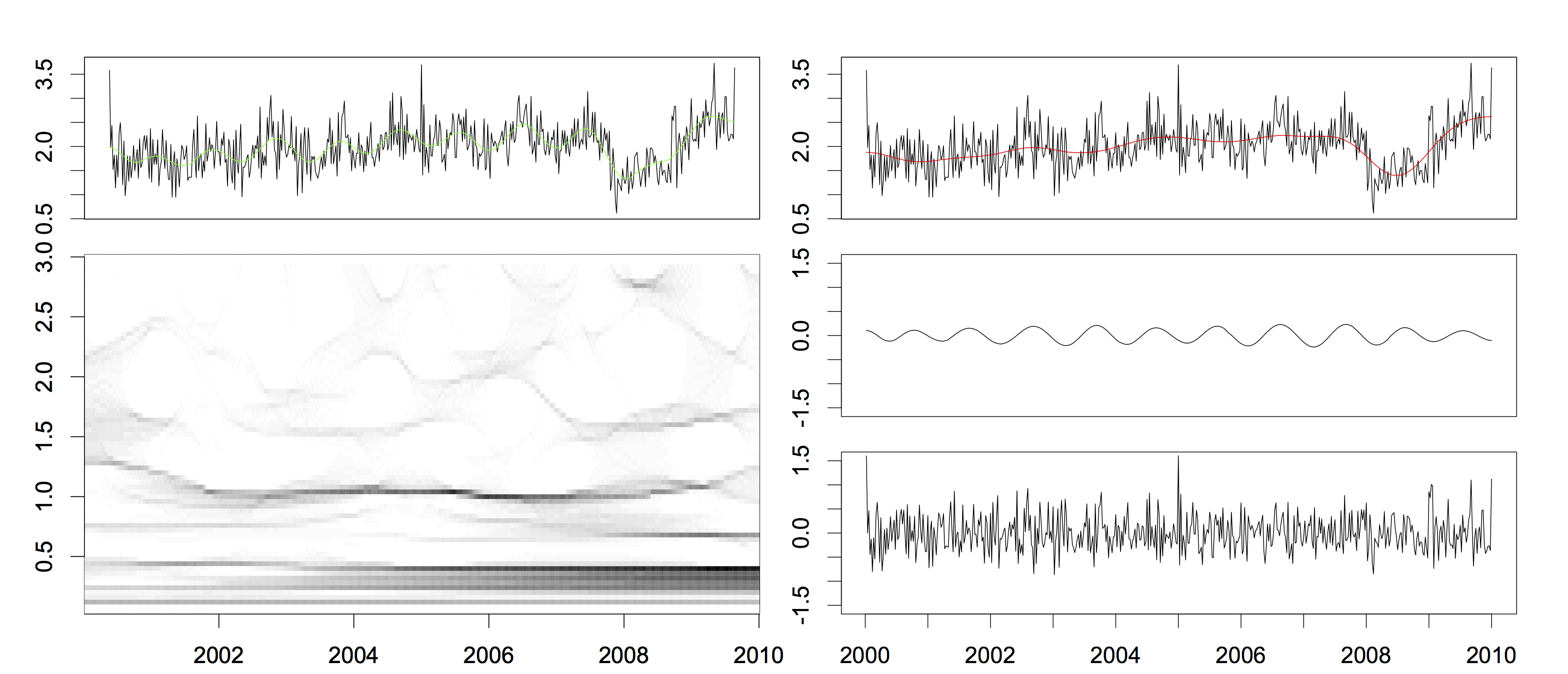}
\caption{{\it Herpes Zoster data analyzed using SST.} Upper left: The Herpes Zoster incidence time series $Y_{\textup{HZ}}$ (black) superimposed with $\widetilde{s}_{\textup{HZ}}+\widetilde{T}_{\textup{HZ}}$ (green). Lower left: The SST result for $Y_{\textup{HZ}}$, where the $y$-axis is the frequency and the intensity of the graph is the absolute value of the SST of $Y_{\textup{HZ}}$ given in (\ref{alogithm:sst:formula}). Right: From top to bottom are $Y_{\textup{HZ}}$ (black) superimposed with the estimated trend $\widetilde{T}_{\textup{HZ}}$ (red), the estimated seasonality $\widetilde{s}_{\textup{HZ}}$, and the residual term $Y_{\textup{HZ}}-\widetilde{s}_{\textup{HZ}}-\widetilde{T}_{\textup{HZ}}$. }
\label{HZ}
\end{figure}

Here, we summarize the findings from Figure \ref{varicella}. First, notice that the seasonality, the dominant curve on the time-frequency plane, is graphically visible based on the SST analysis, as is expected from the properties (P1)-(P4). Second,  we can tell from the estimate $\widetilde{s}_{\textup{V}}[n]$ the dynamics of the seasonality. Before the nationwide public vaccination program was launched in 2004, the seasonal behavior of varicella was stable and evident: it climbed gradually after September or October, reached the peak level in December and the next January, and then declined down to the base during June and July.  This finding is compatible with that of previous studies in Hong Kong and Denmark without public vaccination program \citep{Chan:2011, Metcalf:2009}. After the launch of public vaccination program in 2004, accordant with the increase in the vaccination rate, the winter peak shifted slightly toward spring between 2004 and 2008 while the period remained the same. This finding is consistent with the result of vaccination program in the United States \citep{Seward:2002}. More importantly, less oscillatory seasonality is observed after the launch of public immunization in 2004. This finding is important and less reported before.
The estimated trend of varicella incidence of is compatible with the finding in \cite{Chang:2011}, and we refer the readers to the paper for more discussion. The obvious drop in the trend starting from 2003 may be explained by the free varicella vaccination program in 2003, which was subsequently accelerated by the nationwide public vaccination program commenced in 2004.
Both the sharp decline during 2000-2001 and the increase during 2002 in the trend are less conclusive by this analysis; instead they may have been simply artifacts caused by the transition of the coding system. In Taiwan, the whole medical claim system had undergone a transition from a localized coding scheme (A-code) to the international standardized coding scheme (International Classification of Disease, ICD-9),  which was not completed until 2002. The gradual decrease in the trend starting from 2005 and the fact that the trend seems to level off starting from 2008 may be interpreted as the expected impact of the vaccination program. Clearly, the SST analysis showed its robustness to the coding bias problem in 2000-2002, when recovering the trend in 2003-2010.

Although it is not easy to tell directly from the HZ incidence time series, visually we can detect the existence of the seasonality from the time frequency representation of $Y_{\textup{HZ}}[n]$ provided by SST, as can be seen in Figure \ref{HZ}. The existence of the seasonality is supported by the fact that the occurrence of HZ in elders is a response to the T cell-mediated immunity which is weakened during winter, see for example \cite{Altizer:2006}. From the reconstructed seasonality $\widetilde{s}_{\textup{HZ}}[n]$ and the reconstructed trend, we see that while there was a 50\% increase in the HZ incidence rate among elders after the implementation of nationwide varicella vaccination program in 2004, its seasonality is dynamical -- the frequency was slightly higher before 2002 and slightly lower around 2007. (Notice that the boundary effect should contribute little to this finding since we symmetrically reflect the data on both sides separately.) The finding is rarely reported before but supported by the host immunity (human immunity) mechanism \citep{Dowell:2001}.  However, at this stage we cannot draw the conclusion and further study is required. Another  finding about the HZ incidence time series also coincides with that reported in \cite{Chao:2012}, that is, the incidence of the herpes zoster in elders increased after the implementation of the free varicella vaccination programme. Notice that around 2008 there was a significant drop in both $Y_{\textup{HZ}}[n]$ and the estimated trend, which might have come from a systematic change. However, we do not have enough evidence to explain this significant drop, and the underling mechanism deserves further investigation. Further study of the herd immunity, the interaction between varicella and herpes zoster and the host immunity mechanism is needed but out of scope of this paper.

To compare the findings based on the analyses of varicella and HZ incidences using SST with that using TBATS, we now look at the results obtained from TBATS, which are summarized in Figure \ref{TBATS:varicella}. As the seasonal period is expected to be $1$-year, we ran TBATS by taking the seasonal period to be $52$ weeks. For the varicella time series $Y_{\textup{V}}$, we can see that the estimated seasonality together with the estimated trend, i.e. $\widetilde{s}_{\textup{V}}+\widetilde{T}_{\textup{V}}$, follows $Y_{\textup{V}}$ well, while the estimated trend itself ($\widetilde{T}_{\textup{V}}$) does not. In particular, after 2005, the trend seems to oscillate in the opposite direction with respect to $Y_{\textup{V}}$. This can be explained by the fact that the unexplained part of the deterministic signal, that is, the dynamics in the seasonality, is absorbed by the trend estimate so as that TBATS can fit a nice seasonality to the time series $Y_{\textup{V}}$. In addition, from the public health viewpoint, it is somehow difficult to interpret the second peak in March in the seasonal component.  For the herpes zoster time series $Y_{\textup{HZ}}$, we see that while both the sum of the trend and seasonality estimates ($\widetilde{s}_{\textup{HZ}}+\widetilde{T}_{\textup{HZ}}$) and the trend estimate ($\widetilde{T}_{\textup{HZ}}$) follow $Y_{\textup{HZ}}$ closely, the trend estimate appears to oscillate to same extent. Again, this is inevitable for TBATS because any unexplained dynamic oscillation is absorbed by the trend estimate. To sum up, we find the results given by TBATS are harder to interpret.

\begin{figure}[ht]
\includegraphics[width=1\textwidth]{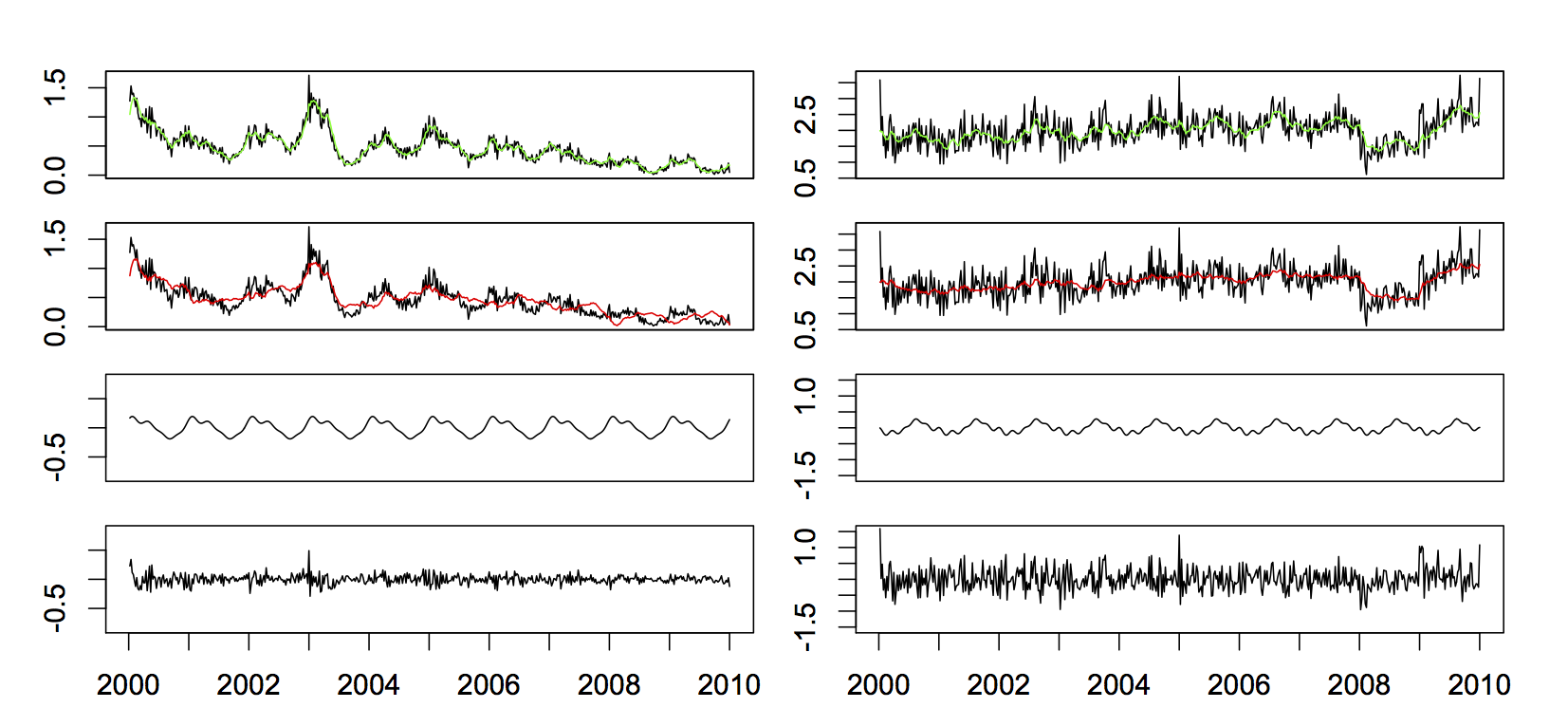}
\caption{{\it Varicella and Herpes Zoster data analyzed by TBATS.} Left: From top to bottom are $Y_{\textup{V}}$ (black) and $\widetilde{s}_{\textup{V}}+\widetilde{T}_{\textup{V}}$ (green); $Y_{\textup{V}}$ (black) and $\widetilde{T}_{\textup{V}}$ (red); $\widetilde{s}_{\textup{V}}$ (black); and  $Y_{\textup{V}}-\widetilde{s}_{\textup{V}}-\widetilde{T}_{\textup{V}}$ (black). Right: From top to bottom are $Y_{\textup{HZ}}$ (black) and $\widetilde{s}_{\textup{HZ}}+\widetilde{T}_{\textup{HZ}}$ (green); $Y_{\textup{HZ}}$ (black) and $\widetilde{T}_{\textup{HZ}}$ (red); $\widetilde{s}_{\textup{HZ}}$ (black); and $Y_{\textup{HZ}}-\widetilde{s}_{\textup{HZ}}-\widetilde{T}_{\textup{HZ}}$ (black). In both columns, the $x$-axes is indexed by year ranging from 2000 to 2010.}
\label{TBATS:varicella}
\end{figure}

\subsection{Seasonal Dynamics of Respiratory Signals and Sleep Stages}\label{sleep}

At first glance, it might be difficult to imagine the existence of seasonality with time-varying period. 
In this section we demonstrate an example from the medical field. Physiological signals contain abundant information, for example, to evaluate a person's health condition we can use information extracted from his/her electrocardiographic (ECG) signals, respiratory signals, blood pressure, and so on \citep{hrv,Benchetrit:2000,Wysocki2006,Lin_Hseu_Yien_Tsao:2011}. Many such signals are oscillatory, and the frequency (or equivalently period) is commonly used to quantify behavior of the oscillation. However, it has been well known that the frequency, as a global quantity,  cannot fully capture the oscillatory behavior of these signals \citep{hrv,Lin_Hseu_Yien_Tsao:2011,Wysocki2006,wu:2011}. In particular, the period between sequential heart beats (resp. respiratory cycles) varies in normal subjects, which is referred to as heart rate variability (HRV) \citep{hrv} (resp. respiratory rate variability (RRV)). HRV (resp. RRV) has been shown to be related to physiological dynamics -- the less variability, the more ill the subject is \citep{hrv,Wysocki2006}. Below, we provide an example of this kind showing that instantaneous frequency (IF) of respiratory signals provide meaningful information about the underlying physiological dynamics.

While sleep is a naturally recurring physiological dynamical state ubiquitous among mammals, its biological nature is so far only partially understood. Based on the physiological and neurological features, sleep is divided into five stages: rapid eye movement (REM) and stage 1 to stage 4 (or as a whole called non-rapid eye movement (NREM)). Sleep stages 1 and 2 are referred to as {\it shallow sleep} and stages 3 and 4 are referred to as {\it deep sleep}. In clinics, the sleep stages are determined by reading the recorded electroencephalography (EEG) based on the R\&K criteria \citep{RK}. Normally, sleep cycles among REM, shallow sleep and deep sleep, and each cycle takes about 90 -- 110 minutes. By analyzing the following data, we demonstrate that the IF of the recorded respiratory signal determined by SST contains sleep stage information.
Standard overnight Polysomnography (Alice 5, Respironics) were performed 
in the Sleep Center of the Chang Gung Memorial Hospital in Taoyuan, Taiwan, to document sleep parameters in four adult patients without sleep apnea diagnosis (age: $43.8\pm10.8$). The sleep length was $368\pm 35$ minutes.  The expert physicians determined the sleep stage by reading the recorded EEG based on the R\&K criteria. We take the determined sleep stage as the gold standard. In addition, the respiratory signal was recorded by the thermistor measuring the nasal airflow at the sampling rate $100$ Hz and down sampled to $5$ Hz in order to speed up the analysis. 

We ran SST on the respiratory signals and extracted the IF. A piece of typical respiratory signal and its SST are shown in Figure \ref{sleep_resp_sig} for reference. Notice that the oscillatory pattern of the respiratory signal is not a cosine function. \cite{wu:2011} referred to this kind of oscillatory pattern as the {\it shape function} and analyzed it by SST . Here we briefly summarize the result. We model the respiratory signal as $A(t)s(\phi(t))$, where $A(t)$ and $\phi(t)$ satisfy the same conditions as those functions in the $\mathcal{A}_{\epsilon}^{c_1,c_2}$ class, and $s\in C^{1,\alpha}$, $\alpha>1/2$, is a $1$-periodic function. Since by Taylor's series expansion $A(t)s(\phi(t))=A(t)\cos(2\pi\phi(t))+\sum_{l=2}^\infty A(t)\cos(2\pi l\phi(t))$, by viewing $\sum_{l=2}^\infty A(t)\cos(2\pi l\phi(t))$ as a component having much higher ``frequency'' than $A(t)\cos(2\pi\phi(t))$, 
SST can estimate $A(t)$ and $\phi'(t)$ from $A(t)s(\phi(t))$ as accurately as if the respiratory signal was $A(t)\cos(2\pi\phi(t))$. Thus, we can ignore the non-harmonic oscillatory pattern and estimate the IF $\phi'(t)$ using SST. We refer the readers to \cite{wu:2011} for more detailed discussion and the technical details.

The extracted IF from the whole night respiratory signal and the K\&R sleep stages of one subject are illustrated in Figure \ref{sleep_resp_sst}, from which we {\it visually} observe high correlation between sleep stage and variation of IF. To further study this correlation, we built up data indicating the IF behavior for each subject in the following way. We divided a respiratory signal into $30$-second sub-intervals $\{I_i\}$, according to the sub-intervals for which the R\&K sleep stage was determined. We evaluated the standard deviation of the estimated IF of the respiratory signal in each $I_i$ and denoted it as $s_i$. Then, for each subject, we grouped all the $s_i$'s according to the sleep stage into four groups --  awake, REM, shallow stage and deep stage. Then, for each subject we ran the F-test on the four groups, and on each pair of the four groups. The $p$-values are listed in Table \ref{sleep:p_value_table}, which are all well below $0.01$ except for the pair REM v.s. awake. Therefore, the oscillatory pattern of the respiratory signal, in particular the dynamical period, contains plentiful information about the sleep depth. It is not surprising that the F-test based on only IF of the respiratory signal cannot detect the difference between REM and awake stages, as distinguishing between them has been considered as a difficult problem \citep{Lo_Jordan_White:2007}. Yet, since sleep is a complicated physiological activity which involves the whole brain, the information obtained from the respiratory signal might compensate that obtained from EEG. Indeed, the information contained in EEG is mainly the activity of the cortex, while the involuntary respiratory neural control center is located in the subcortical area. Further study on the sleep cycle, for example, classification by taking both IF and AM of the respiratory signals into account, analyzing the multiple time series including ECG, EEG and the respiratory signal, and its clinical applications, is beyond the scope here; a more detailed study will be reported in a future paper.

\begin{figure}[ht]
\includegraphics[width=0.95\textwidth]{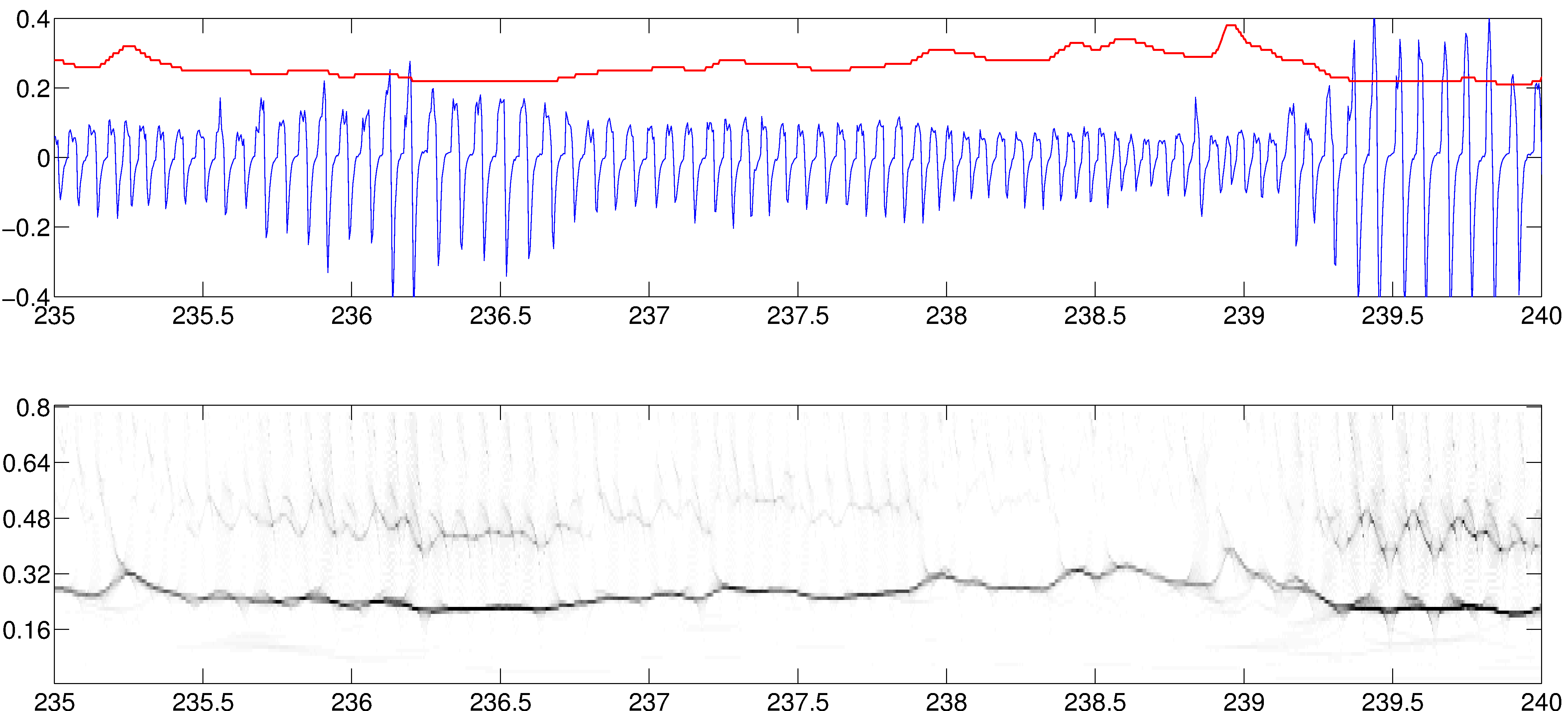}
\caption{Upper panel: The blue curve is a $5$-minute piece of the recorded respiratory signal of subject 1 during shallow sleep, and the superimposed red curve is the IF extracted from this piece using SST. 
Notice that the faster the breathing is, the higher the IF is. Lower panel: The SST of the 
respiratory signal shown in the upper panel. 
Note a dominant ``band'' in the time-frequency plane around 0.3Hz, which is actually a time-varying curve with the $y$-location indicating the IF. Notice that the amplitude modulation (AM) information is encoded on the intensity of the plot -- the deeper the breathing is, the darker the curve is. The unit of the x-axis is minute and that of the y-axis is Hz.}
\label{sleep_resp_sig}
\end{figure}

\begin{figure}[ht]
\includegraphics[width=0.95\textwidth]{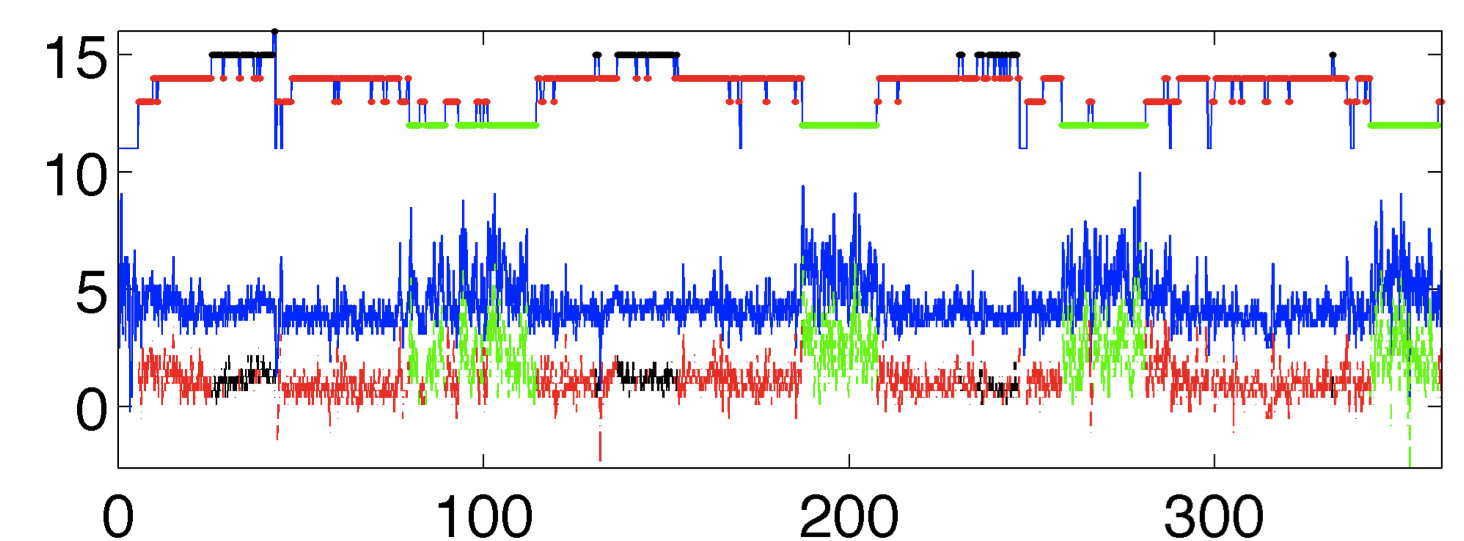}
\caption{In the upper half the K\&R sleep stage in each 30-second interval is indicated by color, with green, red, black and blue colors respectively standing for REM, shallow sleep, deep sleep and awake. The IF is shown as the blue curve in the lower half. The high correlation between the variation of IF and sleep stage is clear visually. To enhance the information encoded in IF, we superimpose on the IF the color attached to the sleep stage and shift it downwards. The differences between the four groups are tested by F-test and the results are shown in Table \ref{sleep:p_value_table}. The unit of x-axis is minute.}
\label{sleep_resp_sst}
\end{figure}

\begin{table}
\caption{\label{sleep:p_value_table} \small Results of the F-test on the difference between sleep stages in IF variation of respiratory signal.
The $p$-values for each of the four individual are shown. Here $8e-12$ means $8\times 10^{-12}$, for example.}
\centering
{\small
\begin{tabular}{| c | c | c | c | c |  }
\hline
\multirow{3}{*}{} & \multicolumn{4}{|c|}{subject} \\
\cline{2-5}
  & $1$ & $2$ & $3$ & $4$  \\
\hline
REM v.s. shallow & $8e-12$ & $0$ & $4.3e-12$ & $0$\\
REM v.s. deep & $0$ & $0$ & $0$ & $6.7e-16$ \\
deep v.s. shallow & $0$ & $1.2e-6$ & $1.4e-6$ & $5e-4$ \\
wake v.s. shallow & $0$ & $0$ & $0$ & $0$\\
wake v.s. deep & $0$ & $9e-15$ & $0$ & $0$ \\
wake v.s. REM & $0.81$ & $0.03$ & $0$ & $0.4$ \\
four groups & $0$ & $0$ & $0$ & $0$\\
\hline
\end{tabular}
}
\end{table}

\section{Discussion}\label{discussion}
We introduce a new nonparametric model and an adaptive time frequency analysis technique, referred to as Synchrosqueezing transform (SST), to analyze time series with dynamical seasonal behavior and smooth trend, which is contaminated by stationary time series modulated by slowly varying variability. Apart from the identifiability results for the nonparametric seasonality model, theoretical results are provided to justify and quantify the capability of SST to extract the seasonal component  and the trend from the noisy observations. In our numerical study, besides testing SST on simulated data and comparing it with the recently developed method TBATS, we applied SST to evaluate two real medical data. First, we studied the influence of the general application of the vaccine on the seasonal dynamics of two highly related diseases, varicella and herpes zoster. Second, we studied if the IF estimated from the respiratory signal provides information about the sleep cycle.

We list below a set of open problems to complete the discussion. 
\begin{itemize}
\item Although the influence of the heteroscedastic, dependent error process on SST is carefully analyzed, more studies on the statistical side are needed. For example, how to {\it adaptively} choose the optimal threshold $\Gamma$ in SST (\ref{alogithm:sst:formula}), and in which sense? Also, to choose the smoothing parameter $\lambda$ in (\ref{algorithm:sst}) automatically, we may explore ideas in the curve fitting literature, for example, cross-validation.
\item It is sometimes interesting to understand the structure of the error process; for example, this is important when we want to forecast future values. There exist several algorithms aiming at identifying CARMA (or ARMA) processes \citep{Brockwell_Davis:2002,Brockwell:2001}. In addition, when the noise is modeled as a CARMA or ARMA random process modulated by a slowly varying function $\sigma(t)$, we may apply to the residuals the algorithms discussed in the literature, for example \cite{Dahlhaus:1997}, \cite{Hallin:1978,Hallin:1980} and \cite{Rosen_Stoffer_Wood:2009}, to estimate the heteroscedastic variance $\sigma^2(t)$.  
However, although the error process can be efficiently separated from the observed time series by SST (indicated by Theorem \ref{section:theorem:stability} result (iii), Theorem \ref{section:theorem:ARMAstability} result (iii), Tables \ref{table:simulation:Y0} and \ref{table:simulation:Y1}, and Figures \ref{fig:simulation:sst_recon0}, \ref{fig:simulation:sst_recon} and \ref{fig:simulation:sst_recon2}), theoretical quantification of the influence of SST on these approaches remains unknown, and further investigation is needed. 
\item How to identify the existence of the seasonality, and how to decide the number of components if seasonality exists? 
Although we do not have an ultimate solution, a possible approach is the following. First, test the null hypothesis that the observed process is stationary, modulated by the heteroscedastic variance and the trend. If the null hypothesis is not rejected, stop and conclude that there is no seasonality in the signal. Otherwise, in the next step, we conduct the following forward procedure to determine the number of seasonal components $K$. Starting from $k=1$, we visually determine {\it k} components from the SST TF plane, reconstruct the seasonal components and the trend, and then test if the residuals (obtained by subtracting from the signal the reconstructed $k$ seasonal components and trend) is stationary, modulated by the heteroscedastic variance. If the stationarity null hypothesis is not rejected we stop and decide $K$ equals the current value of $k$; otherwise we increase the value of $k$ by $1$, and then find another seasonal component from the TF plane and repeat the above estimation and testing procedure. With this iterative approach, we are able to determine $K$. A relevant issue is to analyze behavior of SST on random error processes. We leave these important problems as a future work.
\item Forecasting is another important issue in time series analysis. By extrapolating the seasonal components and trend, and extrapolating the residuals (which approximates the random error term) using the Kalman filter, we can achieve short-term prediction based on the SST approach. Currently we do not have a theoretical result justifying this approach. Also, we do not have yet a way to extend this for long-term forecasting, as existing models such as TBATS can do.   
\item Since the proposed model for the seasonality is nonparametric in nature, the results from the SST algorithm may be used to check the validity of some specific sub-models such as constant seasonal periods. Further investigation in this direction is needed.
\item In the classical time series analysis, the trend is usually modeled as a stochastic component as in SARIMA and TBATS. In our models, the trend is a deterministic component instead of a stochastic one, and  the only stochastic part is the stationary or ``almost stationary'' random error process. In general, it is possible that a stochastic trend exists; in that case, how to distinguish it from the heteroscedastic, dependent error requires further study. 
\item Sometimes local bursts occur in real data. We include in the Supplementary a stimulation study; it indicates that the SST algorithm can distinguish local bursts from trend and the dynamical seasonality, and the SST reconstruction is not significantly affected. On the other hand, SST views local bursts as part of the error process; as a result, approximation of the error process by the residuals is affected. We will study how to distinguish between local bursts and the error process in the future.
\item The SST method can be used to study seasonal patterns with constant periods as well. When the period of a seasonal component is a known constant, we can skip the step to estimate the instantaneous frequency $\phi'_k(t)$ by replacing its estimate $\widetilde{\phi}'_k$ in (\ref{algorithm:if}) with its known value. The resulting estimator is expected to be more stable and would be useful in situations where the assumption holds.  
 \item Although a preliminary interpretation and survey of the medical signals in Section \ref{data} is provided, from the clinical viewpoint, further studies are needed to better understand the underlying dynamical mechanism. For example, the host/herd immunity and the interaction between varicella and herpes zoster are not yet clear, and we may combine the respiratory signal information with other biomedical signals to predict sleep stages more accurately. 
\end{itemize}

\section{Acknowledgements}
The study in Section \ref{VHZ} was based on data from the NHRID provided by the BNHI, Department of Health, and managed by NHRI in Taiwan. The interpretation and conclusions contained therein do not represent those
of the BNHI, the Department of Health, or the NHRI. Hau-Tieng Wu acknowledges support by AFOSR grant FA9550-09-1-0551, NSF grant CCF-0939370 and FRG grant DSM-1160319, the valuable discussion with Professor Ingrid Daubechies and Dr. Yu-Lun Lo and he thanks Dr. Yu-Lun Lo for providing the sleep data. Cheng's research is supported in part by the National Science Council grants NSC97-2118-002-001-MY3 and NSC101-2118-002-001-MY3, and the Mathematics Division, National Center for Theoretical Sciences (Taipei Office). The authors thank the anonymous reviewers, Professor Liudas Giraitis, Professor Peter Robinson and Professor William Dunsmuir for their valuable comments which significantly improve the presentation of the paper.

\bibliographystyle{Chicago}
\bibliography{seasonality_rev}

\begin{thebibliography}{}

\bibitem[\protect\citeauthoryear{Altizer, Dobson, Hosseini, Hudson, Pascual,
  and Rohani}{Altizer et~al.}{2006}]{Altizer:2006}
Altizer, S., A.~Dobson, P.~Hosseini, P.~Hudson, M.~Pascual, and P.~Rohani
  (2006).
\newblock Seasonality and the dynamics of infectious diseases.
\newblock {\em Ecology letters\/}~{\em 9\/}(4), 467--484.

\bibitem[\protect\citeauthoryear{Benchetrit}{Benchetrit}{2000}]{Benchetrit:2000}
Benchetrit, G. (2000).
\newblock Breathing pattern in humans: diversity and individuality.
\newblock {\em Respiration Physiology\/}~{\em 122\/}(2-3), 123 -- 129.

\bibitem[\protect\citeauthoryear{Bickel, Kleijn, and Rice}{Bickel
  et~al.}{2008}]{Bickel_Kleijn_Rice:2008}
Bickel, P., B.~Kleijn, and J.~Rice (2008).
\newblock {Event weighted tests for detecting periodicity in photon arrival
  times}.
\newblock {\em The Astrophysical Journal\/}~{\em 685}, 384--389.

\bibitem[\protect\citeauthoryear{Brockwell}{Brockwell}{1995}]{Brockwell:1995}
Brockwell, P.~J. (1995).
\newblock {A note on the embedding of discrete-time ARMA processes}.
\newblock {\em Journal of Time Series Analysis\/}~{\em 16\/}(5), 451--460.

\bibitem[\protect\citeauthoryear{Brockwell}{Brockwell}{2001}]{Brockwell:2001}
Brockwell, P.~J. (2001).
\newblock {Continuous-Time ARMA Processes}.
\newblock In {\em Handbook of Statistics}, Volume~19, pp.\  249--276. Elsevier.

\bibitem[\protect\citeauthoryear{Brockwell}{Brockwell}{2009}]{Brockwell:2009}
Brockwell, P.~J. (2009).
\newblock {L\'{e}vy-Driven Continuous-Time ARMA Processes}.
\newblock In {\em Handbook of Financial Time Series}, pp.\  457--486.
  Springer-Verlag.

\bibitem[\protect\citeauthoryear{Brockwell and Davis}{Brockwell and
  Davis}{2002}]{Brockwell_Davis:2002}
Brockwell, P.~J. and R.~A. Davis (2002).
\newblock {\em Introduction to Time Series and Forecasting}.
\newblock Springer.

\bibitem[\protect\citeauthoryear{Brockwell and Hannig}{Brockwell and
  Hannig}{2010}]{Brockwell:2010}
Brockwell, P.~J. and J.~Hannig (2010).
\newblock {CARMA(p,q) generalized random processes}.
\newblock {\em J. Stat. Plan. Infer.\/}~{\em 140\/}(12), 3613--3618.

\bibitem[\protect\citeauthoryear{Brockwell and Lindner}{Brockwell and
  Lindner}{2009}]{Brockwell_Lindner:2009}
Brockwell, P.~J. and A.~Lindner (2009).
\newblock {Existence and uniqueness of stationary L\'{e}vy-driven CARMA
  processes}.
\newblock {\em Stoch. Proc. Appl.\/}~{\em 119\/}(8), 2660--2681.

\bibitem[\protect\citeauthoryear{Brockwell and Lindner}{Brockwell and
  Lindner}{2010}]{Brockwell_Lindner:2010}
Brockwell, P.~J. and A.~Lindner (2010).
\newblock {Strictly stationary solutions of autoregressive moving average
  equations}.
\newblock {\em Biometrika\/}~{\em 97\/}(3), 765--772.

\bibitem[\protect\citeauthoryear{Chan, Tian, Kwan, Chan, and Leung}{Chan
  et~al.}{2011}]{Chan:2011}
Chan, J., L.~Tian, Y.~Kwan, W.~Chan, and C.~Leung (2011).
\newblock Hospitalizations for varicella in children and adolescents in a
  referral hospital in {H}ong {K}ong, 2004 to 2008: A time series study.
\newblock {\em BMC Public health\/}~{\em 11\/}(1), 366.

\bibitem[\protect\citeauthoryear{Chang, Huang, Chang, and Tsai}{Chang
  et~al.}{2011}]{Chang:2011}
Chang, L.-Y., L.-M. Huang, I.-S. Chang, and F.-Y. Tsai (2011).
\newblock Epidemiological characteristics of varicella from 2000 to 2008 and
  the impact of nationwide immunization in {T}aiwan.
\newblock {\em BMC Infectious disease\/}~{\em 11\/}(1), 352.

\bibitem[\protect\citeauthoryear{Chao, Chien, Yeh, Hsu, and Lian}{Chao
  et~al.}{2012}]{Chao:2012}
Chao, D.-Y., Y.-Z. Chien, Y.-P. Yeh, P.-S. Hsu, and I.-B. Lian (2012).
\newblock The incidence of varicella and herpes zoster in {T}aiwan during a
  period of increasing varicella vaccine coverage, 2000--2008.
\newblock {\em Epidemiology and infection\/}~{\em 140\/}(6), 1131--1140.

\bibitem[\protect\citeauthoryear{Chassande-Mottin, Auger, and
  Flandrin}{Chassande-Mottin et~al.}{2003}]{Chassande-MottinAugerFlandrin:03}
Chassande-Mottin, E., F.~Auger, and P.~Flandrin (2003).
\newblock Time-frequency/time-scale reassignment.
\newblock In {\em Wavelets and signal processing}, Appl. Numer. Harmon. Anal.,
  pp.\  233--267. Birkh\"auser.

\bibitem[\protect\citeauthoryear{Chassande-Mottin, Daubechies, Auger, and
  Flandrin}{Chassande-Mottin et~al.}{1997}]{Chassande-MottinDaubechiesAuger:97}
Chassande-Mottin, E., I.~Daubechies, F.~Auger, and P.~Flandrin (1997).
\newblock Differential reassignment.
\newblock {\em Signal Processing Letters, IEEE\/}~{\em 4\/}(10), 293--294.

\bibitem[\protect\citeauthoryear{Chen, Yeh, Wu, Haschler, Chen, and
  Wetter}{Chen et~al.}{2011}]{Chen:2011}
Chen, Y.-C., H.-Y. Yeh, J.-C. Wu, I.~Haschler, T.-J. Chen, and T.~Wetter
  (2011).
\newblock Taiwan's national health insurance research database: Administrative
  health care database as study object in bibliometrics.
\newblock {\em Scientometrics\/}~{\em 86\/}(1), 365--380.

\bibitem[\protect\citeauthoryear{Dahlhaus}{Dahlhaus}{1997}]{Dahlhaus:1997}
Dahlhaus, R. (1997).
\newblock {Fitting Time Series Models to Nonstationary Processes}.
\newblock {\em Ann. Stat.\/}~{\em 25\/}(1), 1--37.

\bibitem[\protect\citeauthoryear{Daubechies}{Daubechies}{1992}]{daubechies:1992}
Daubechies, I. (1992).
\newblock {\em Ten Lectures on Wavelets}.
\newblock Philadelphia: SIAM.

\bibitem[\protect\citeauthoryear{Daubechies, Lu, and Wu}{Daubechies
  et~al.}{2010}]{daubechies_lu_wu:2010}
Daubechies, I., J.~Lu, and H.-T. Wu (2010).
\newblock Synchrosqueezed wavelet transforms: An empirical mode
  decomposition-like tool.
\newblock {\em Appl. Comput. Harmon. Anal.\/}~{\em 30}, 243--261.

\bibitem[\protect\citeauthoryear{Daubechies and Maes}{Daubechies and
  Maes}{1996}]{daubechies_maes:1996}
Daubechies, I. and S.~Maes (1996).
\newblock A nonlinear squeezing of the continuous wavelet transform based on
  auditory nerve models.
\newblock In {\em Wavelets in Medicine and Biology}, pp.\  527--546. CRC-Press.

\bibitem[\protect\citeauthoryear{{De Livera}, Hyndman, and Snyder}{{De Livera}
  et~al.}{2011}]{DeLivera_Alysha_Hyndman_Snyder:2011}
{De Livera}, A.~M., R.~J. Hyndman, and R.~D. Snyder (2011).
\newblock {Forecasting Time Series With Complex Seasonal Patterns Using
  Exponential Smoothing}.
\newblock {\em J. Am. Stat. Assoc.\/}~{\em 106\/}(496), 1513--1527.

\bibitem[\protect\citeauthoryear{Dowell}{Dowell}{2001}]{Dowell:2001}
Dowell, S. (2001).
\newblock Seasonal variation in host susceptibility and cycles of certain
  infectious diseases.
\newblock {\em Emerging Infectious Diseases\/}~{\em 7\/}(3), 369--374.

\bibitem[\protect\citeauthoryear{Flandrin}{Flandrin}{1999}]{flandrin:1999}
Flandrin, P. (1999).
\newblock {\em Time-frequency/time-scale Analysis}.
\newblock Academic Press.

\bibitem[\protect\citeauthoryear{Gallerani and Manfredini}{Gallerani and
  Manfredini}{2000}]{Gallerani:2000}
Gallerani, M. and R.~Manfredini (2000).
\newblock Seasonal variation in herpes zoster infection.
\newblock {\em British journal of dermatology\/}~{\em 142\/}(3), 588--589.

\bibitem[\protect\citeauthoryear{Gel'fand and Vilenkin}{Gel'fand and
  Vilenkin}{1964}]{Gelfand:1964}
Gel'fand, I. and N.~Y. Vilenkin (1964).
\newblock {\em Generalized function theory Vol 4}.
\newblock Academic Press.

\bibitem[\protect\citeauthoryear{Genton and Hall}{Genton and
  Hall}{2007}]{Genton_Hall:2007}
Genton, M.~G. and P.~Hall (2007).
\newblock {Statistical inference for evolving periodic functions}.
\newblock {\em J. Roy. Stat. Soc. B\/}~{\em 69\/}(4), 643--657.

\bibitem[\protect\citeauthoryear{Golombek and Rosenstein}{Golombek and
  Rosenstein}{2010}]{golombek_rosenstein:2010}
Golombek, D. and R.~Rosenstein (2010).
\newblock Physiology of circadian entrainment.
\newblock {\em Physiol. Rev.\/}~{\em 90}, 1063--1102.

\bibitem[\protect\citeauthoryear{Hall, Reimann, and Rice}{Hall
  et~al.}{2000}]{Hall_Reimann_Rice:2000}
Hall, P., J.~Reimann, and J.~Rice (2000).
\newblock {Nonparametric estimation of a periodic function}.
\newblock {\em Biometrika\/}~{\em 87\/}(3), 545--557.

\bibitem[\protect\citeauthoryear{Hallin}{Hallin}{1978}]{Hallin:1978}
Hallin, M. (1978).
\newblock {Mixed autoregressive moving-average multivariate processes with time
  dependent coefficients}.
\newblock {\em J Multivariate Anal\/}~{\em 8}, 567--572.

\bibitem[\protect\citeauthoryear{Hallin}{Hallin}{1980}]{Hallin:1980}
Hallin, M. (1980).
\newblock {Invertibility and Generalized Invertibility of Time Series Models}.
\newblock {\em J. Roy. Stat. Soc. B\/}~{\em 42}, 210--212.

\bibitem[\protect\citeauthoryear{Ishikawa, Niwa, and Tanaka}{Ishikawa
  et~al.}{2012}]{Ishikawa:2012}
Ishikawa, K., M.~Niwa, and T.~Tanaka (2012).
\newblock Difference of intensity and disparity in impact of climate on several
  vascular diseases.
\newblock {\em Heart and Vessels\/}~{\em 27\/}(1), 1--9.

\bibitem[\protect\citeauthoryear{Lin, Jones, Liu, and Hwang}{Lin
  et~al.}{2011}]{Lin_Jones:2011}
Lin, S., R.~Jones, X.~Liu, and S.~Hwang (2011).
\newblock Impact of the return to school on childhood asthma burden in new york
  state.
\newblock {\em Int. J. Occup. Env. Heal.\/}~{\em 17\/}(1), 9--16.

\bibitem[\protect\citeauthoryear{Lin, Hseu, Yien, and Tsao}{Lin
  et~al.}{2011}]{Lin_Hseu_Yien_Tsao:2011}
Lin, Y.-T., S.-S. Hseu, H.-W. Yien, and J.~Tsao (2011).
\newblock {Analyzing autonomic activity in electrocardiography about general
  anesthesia by spectrogram with multitaper time-frequency reassignment}.
\newblock {\em Proc. IEEE-BMEI\/}~{\em 2}, 628--632.

\bibitem[\protect\citeauthoryear{Lo, Jordan, Malhotra, Wellman, Heinzer,
  Eikermann, Schory, Dover, and White}{Lo et~al.}{2007}]{Lo_Jordan_White:2007}
Lo, Y.-L., A.~Jordan, A.~Malhotra, A.~Wellman, R.~Heinzer, M.~Eikermann,
  K.~Schory, L.~Dover, and D.~White (2007).
\newblock Influence of wakefulness on pharyngeal airway muscle activity.
\newblock {\em Thorax\/}~{\em 62}, 798--804.

\bibitem[\protect\citeauthoryear{Malik and Camm}{Malik and Camm}{1995}]{hrv}
Malik, M. and A.~J. Camm (1995).
\newblock {\em Heart rate variability}.
\newblock Wiley-Blackwell.

\bibitem[\protect\citeauthoryear{Marin, Meissner, and Seward}{Marin
  et~al.}{2008}]{Marin:2008}
Marin, M., H.~Meissner, and J.~Seward (2008).
\newblock Varicella prevention in the united states: A review of successes and
  challenges.
\newblock {\em Pediatrics\/}~{\em 122\/}(3), 744--751.

\bibitem[\protect\citeauthoryear{Metcalf, Bjornstad, Grenfell, and
  Andreasen}{Metcalf et~al.}{2009}]{Metcalf:2009}
Metcalf, C., O.~Bjornstad, B.~Grenfell, and V.~Andreasen (2009).
\newblock Seasonality and comparative dynamics of six childhood infections in
  pre-vaccination copenhagen.
\newblock {\em Proceedings of the royal society B\/}~{\em 276\/}(1),
  4111--4118.

\bibitem[\protect\citeauthoryear{Nott and Dunsmuir}{Nott and
  Dunsmuir}{2002}]{Nott_Dunsmuir:2002}
Nott, D.~J. and W.~T. Dunsmuir (2002).
\newblock {Estimation of nonstationary spatial covariance structure}.
\newblock {\em Biometrika\/}~{\em 89}, 819--829.

\bibitem[\protect\citeauthoryear{Oh, Nychka, Brown, and Charbonneau}{Oh
  et~al.}{2004}]{Oh_Nychka_Brown_Charbonneau:2004}
Oh, H.-S., D.~Nychka, T.~Brown, and P.~Charbonneau (2004).
\newblock {Period analysis of variable stars by robust smoothing}.
\newblock {\em J. Roy. Stat. Soc. B\/}~{\em 53\/}(1), 15--30.

\bibitem[\protect\citeauthoryear{Park, Ahn, Hendry, and Jang}{Park
  et~al.}{2011}]{Park_Ahn_Hendry_Jang:2011}
Park, C., J.~Ahn, M.~Hendry, and W.~Jang (2011).
\newblock {Analysis of long period variable starts with nonparametric tests for
  trend detection}.
\newblock {\em J. Am. Stat. Assoc.\/}~{\em 106\/}(495), 832--845.

\bibitem[\protect\citeauthoryear{Perez-Farinos, Ordobas, Garcia-Fernandez,
  Garc'a-Comas, Canellas, Rodero, Gutierrez-Rodriguez, Garcia-Gutierrez, and
  R.}{Perez-Farinos et~al.}{2007}]{PF:2007}
Perez-Farinos, N., M.~Ordobas, C.~Garcia-Fernandez, L.~Garc'a-Comas,
  S.~Canellas, I.~Rodero, A.~Gutierrez-Rodriguez, J.~Garcia-Gutierrez, and
  R.~R. (2007).
\newblock Varicella and herpes zoster in madrid, based on the sentinel general
  practitioner network: 1997 -- 2004.
\newblock {\em BMC Infectious Diseases\/}~{\em 7\/}(1), 59.

\bibitem[\protect\citeauthoryear{Pollock}{Pollock}{2009}]{Pollock:2009}
Pollock, D. (2009).
\newblock Investigating economic trends and cycles.
\newblock In {\em Vol. 2 Applied Econometrics, T.C. Mills and K. Patterson
  (eds.)}, Palgrave Handbook of Econometrics. Palgrave Macmillan.

\bibitem[\protect\citeauthoryear{Priestley}{Priestley}{1965}]{Priestley:1965}
Priestley, M.~B. (1965).
\newblock {Evolutionary spectra and non-stationary processes}.
\newblock {\em J. Roy. Stat. Soc. B\/}~{\em 27\/}(2), 204--237.

\bibitem[\protect\citeauthoryear{Rechtschaffen and Kales}{Rechtschaffen and
  Kales}{1968}]{RK}
Rechtschaffen, A. and A.~Kales (1968).
\newblock {\em A Manual of Standardized Terminology, Techniques and Scoring
  System for Sleep Stages of Human Subjects}.
\newblock Washington: Public Health Service, US Government Printing Office.

\bibitem[\protect\citeauthoryear{Rosen, Stoffer, and Wood}{Rosen
  et~al.}{2009}]{Rosen_Stoffer_Wood:2009}
Rosen, O., D.~S. Stoffer, and S.~Wood (2009).
\newblock {Local spectral analysis via a Bayesian mixture of smoothing
  splines}.
\newblock {\em J. Am. Stat. Assoc.\/}~{\em 104\/}(485), 249--262.

\bibitem[\protect\citeauthoryear{Seward, Watson, Peterson, Mascola, Pelosi,
  Zhang, Maupin, Goldman, Tabony, Brodovicz, Jumaan, and Wharton}{Seward
  et~al.}{2002}]{Seward:2002}
Seward, J., B.~Watson, C.~Peterson, L.~Mascola, J.~Pelosi, J.~Zhang, T.~Maupin,
  G.~Goldman, L.~Tabony, K.~Brodovicz, A.~Jumaan, and M.~Wharton (2002).
\newblock Varicella disease after introduction of varicella vaccine in the
  united states, 1995-2000.
\newblock {\em JAMA\/}~{\em 287\/}(5), 606--611.

\bibitem[\protect\citeauthoryear{Stone, Olinky, and Huppert}{Stone
  et~al.}{2007}]{Stone:2007}
Stone, L., R.~Olinky, and A.~Huppert (2007).
\newblock Seasonal dynamics of recurrent epidemics.
\newblock {\em Nature\/}~{\em 446\/}(7135), 533--536.

\bibitem[\protect\citeauthoryear{Thakur, Brevdo, Fuckar, and Wu.}{Thakur
  et~al.}{2013}]{brevdo_fuckar_thakur_wu:2012}
Thakur, G., E.~Brevdo, N.~S. Fuckar, and H.-T. Wu. (2013).
\newblock The synchrosqueezing algorithm for time-varying spectral analysis:
  robustness properties and new paleoclimate applications.
\newblock {\em Signal Processing\/}~{\em 93}, 1079--1094.

\bibitem[\protect\citeauthoryear{Wang}{Wang}{2010}]{wang:2010}
Wang, X. (2010).
\newblock Neurophysiological and computational principles of cortical rhythms
  in cognition.
\newblock {\em Physiol. Rev.\/}~{\em 90}, 1195--1268.

\bibitem[\protect\citeauthoryear{Wu}{Wu}{2012}]{wu:2011}
Wu, H.-T. (2012).
\newblock Instantaneous frequency and wave shape functions {(I)}.
\newblock {\em Appl. Comput. Harmon. Anal.\/}.
\newblock In press.

\bibitem[\protect\citeauthoryear{Wysocki, Cracco, Teixeira, Mercat, Diehl,
  Lefort, Derenne, and Similowski}{Wysocki et~al.}{2006}]{Wysocki2006}
Wysocki, M., C.~Cracco, A.~Teixeira, A.~Mercat, J.~Diehl, Y.~Lefort,
  J.~Derenne, and T.~Similowski (2006).
\newblock Reduced breathing variability as a predictor of unsuccessful patient
  separation from mechanical ventilation.
\newblock {\em Crit. Care. Med.\/}~{\em 34}, 2076--2083.

\end{thebibliography}

\newpage

\centerline{\bf {\large Supplementary Materials for ``Nonparametric modeling and adaptive}}
\centerline{\bf {\large estimation for dynamical seasonality and trend with heteroscedastic }}
\centerline{\bf {\large and dependent errors''}}

\centerline{\large  by Yun-Chun Chen, Ming-Yen Cheng, and  Hau-Tieng Wu}

\setcounter{equation}{0}
\setcounter{section}{0}
\setcounter{table}{0}
\setcounter{figure}{0}
\setcounter{page}{1}
\renewcommand{\theequation}{S.\arabic{equation}}
\renewcommand{\thesection}{S.\arabic{section}}
\renewcommand{\thetheorem}{S.\arabic{theorem}}
\renewcommand{\thelemma}{S.\arabic{lemma}}

\section{Proof of Theorem \ref{theorem:identifiability:single}}

The proof contains two parts. The first part is showing the restrictions on the perturbations $\alpha$ and $\beta$ based on the positivity condition of the instantaneous frequency and amplitude modulation functions. The second part is to control the amplitude of $\alpha$ and $\beta$ and their derivatives by the ``slowly varying'' conditions of $\mathcal{A}^{c_1,c_2}_\epsilon$ functional class.

First, observe that if (\ref{observation:identifiability:lemma:1}) holds, then by the definition of $\mathcal{A}^{c_1,c_2}_\epsilon$, we know 
\begin{align}
&\alpha\in C^2(\RR),~ \beta\in C^1(\RR)\label{proof:2.1:Aeps:cond5},\\
&\inf_{t\in\RR}\phi'(t)>c_1,~\sup_{t\in\RR}\phi'(t)<c_2,\label{proof:2.1:Aeps:phicond1}\\
&\inf_{t\in\RR}a(t)>c_1,~\sup_{t\in\RR}a(t)<c_2,\label{proof:2.1:Aeps:acond1}\\
&|a'(t)|\leq \epsilon \phi'(t),~|\phi''(t)|\leq \epsilon \phi'(t),\label{proof:2.1:Aeps:boundcond1}
\end{align}
and 
\begin{align}
&\inf_{t\in\RR}[\phi'(t)+\alpha'(t)]>c_1,~\sup_{t\in\RR}[\phi'(t)+\alpha'(t)]<c_2,\label{proof:2.1:Aeps:phicond2}\\
&\inf_{t\in\RR}[a(t)+\beta(t)]>c_1,~\sup_{t\in\RR}[a(t)+\beta(t)]<c_2,\label{proof:2.1:Aeps:acond2}\\
&|a'(t)+\beta'(t)|\leq \epsilon (\phi'(t)+\alpha'(t)),~|\phi''(t)+\alpha''(t)|\leq \epsilon (\phi'(t)+\alpha'(t)).\label{proof:2.1:Aeps:boundcond2}
\end{align}
From (\ref{proof:2.1:Aeps:boundcond1}) and (\ref{proof:2.1:Aeps:boundcond2}) we have 
\begin{align}
|\beta'(t)|\leq \epsilon(2\phi'(t)+\alpha'(t))\mbox{ and }|\alpha''(t)|\leq \epsilon(2\phi'(t)+\alpha'(t)). \label{proof:2.1:Aeps:cond6}
\end{align}

Now, based on the conditions (\ref{proof:2.1:Aeps:cond5}), (\ref{proof:2.1:Aeps:acond1}), (\ref{proof:2.1:Aeps:phicond1}), (\ref{proof:2.1:Aeps:acond2}) and (\ref{proof:2.1:Aeps:phicond2}), we prove that $\alpha$ and $\beta$ must have some ``hinging points'' that restrict their behaviors. Notice that the definition of $t_m$ and $s_m$ depends on the monotonicity of $\phi(t)$, which is true by condition (\ref{proof:2.1:Aeps:phicond1}). Observe that, for any $n\in\ZZ$, when $t=t_n$, 
\begin{eqnarray}
&&(a(t_n)+\beta(t_n))\cos(\phi(t_n)+\alpha(t_n))\nonumber\\
&=&(a(t_n)+\beta(t_n))\cos[n\pi+\pi/2+\alpha(t_n)]\label{before_proof:identifiability:eq1}\\
&=&a(t_n)\cos(n\pi+\pi/2)=0,\nonumber
\end{eqnarray}
where the second equality comes from (\ref{observation:identifiability:lemma:1}). This leads to $\alpha(t_n)=k_n\pi$, $k_n\in\ZZ$, since $\phi(t_n)=(n+1/2)\pi$ and $a(t_n)+\beta(t_n)>0$ by (\ref{proof:2.1:Aeps:phicond2}). We show that $k_n$ are the same for all $n\in\ZZ$. First, suppose there exists $t_n$ so that $\alpha(t_n)=k\pi$ and $\alpha(t_{n+1})=(k+2l)\pi$, where $k,l\in \ZZ$ and $l>0$. Note that $k$ should be an even number, otherwise we have
$$
a(t_n)\cos(\phi(t_n))=(a(t_n)+\beta(t_n))\cos(\phi(t_n)+\alpha(t_n))=-(a(t_n)+\beta(t_n))\cos(\phi(t_n)),
$$
which is absurd due to (\ref{proof:2.1:Aeps:acond1}) and (\ref{proof:2.1:Aeps:acond2}). Then, it follows from (\ref{proof:2.1:Aeps:cond5}) that 
there exists $t'\in(t_n,t_{n+1})$ so that $\alpha(t')=(k+1)\pi$ and hence 
\[
a(t')\cos(\phi(t'))= (a(t')+\beta(t'))\cos(\phi(t')+\alpha(t'))=-(a(t')+\beta(t'))\cos(\phi(t')),
\]
which is absurd since $\cos(\phi(t'))\neq 0$, (\ref{proof:2.1:Aeps:acond1}) and (\ref{proof:2.1:Aeps:acond2}). Similar argument holds when $l<0$. Second, suppose there exists $t_n$ so that $\alpha(t_n)=k\pi$ and $\alpha(t_{n+1})=(k+2l-1)\pi$, where $k,l\in \ZZ$ and $l>0$. If $l>1$, the same argument holds. If $l=1$, it is absurd again since $a(t')>0$ and $a(t')+\beta(t')>0$ but the sign changes.


Thus we know $\alpha(t_n)=k\pi$ for all $n\in\ZZ$, for some fixed $k\in \ZZ$. We now show that $k=0$.
Suppose $\alpha(t_n)=k_0\pi$, where $k_0$ is a fixed even integer, for all $n\in\ZZ$, then we set $\alpha$ to be $\alpha-k_0\pi$ and the claim is done since they are equivalent after being composed with the cosine function. Suppose $k_0$ is a fixed odd integer, then since $\alpha\in C^2(\RR)$ and $\alpha(t_n)=\alpha(t_{n+1})=k_0\pi$, there exists $t'\in(t_n,t_{n+1})$ so that $\alpha(t')=k_0\pi$ and hence 
\begin{align*}
a(t')\cos(\phi(t'))&\,= (a(t')+\beta(t'))\cos(\phi(t')+\alpha(t'))=-(a(t')+\beta(t'))\cos(\phi(t')),
\end{align*}
which is again absurd since $\cos(\phi(t'))\neq 0$, (\ref{proof:2.1:Aeps:acond1}) and (\ref{proof:2.1:Aeps:acond2}). As a result, we get $\alpha(t_n)=0$ for all $n\in\ZZ$. Furthermore, by the fundamental theorem of calculus, we know
\begin{equation}\label{proof:2.1:alpha:changesign}
0=\alpha(t_{n+1})=\int^{t_{n+1}}_{t_n}\alpha'(u)\ud u,
\end{equation}
which implies that $\alpha'(t)$ changes sign inside $[t_n,\,t_{n+1}]$ for all $n\in\ZZ$. Also, due to the monotonicity of $\phi+\alpha$ (\ref{proof:2.1:Aeps:phicond2}), 
\begin{equation}\label{proof:2.1:alpha:bound}
|\alpha(t')-\alpha(t'')|<\pi
\end{equation} 
for any $t',t''\in[t_m,t_{m+1}]$ for all $m\in\ZZ$. Indeed, if $|\alpha(t')-\alpha(t'')|\geq \pi$, for some $t,t''\in[t_n,t_{n+1}]$, we contradict (\ref{proof:2.1:Aeps:phicond2}) since $(n+1/2)\pi<\phi(t')+\alpha(t')<(n+3/2)\pi$ and $\phi(t_{n+1})+\alpha(t_{n+1})=(n+3/2)\pi$. 

We next claim that $\beta(s_m)\geq 0$ for all $m\in\ZZ$. When $t=s_m$, we have
\begin{align}
(-1)^ma(s_m)=&\,a(s_m)\cos(m\pi)\label{before_proof:identifiability:eq2}\\
=&\,(a(s_m)+\beta(s_m))\cos[m\pi+\alpha(s_m)]\nonumber\\
=&\,(-1)^m(a(s_m)+\beta(s_m))\cos(\alpha(s_m))\nonumber,
\end{align}
where the second equality comes from (\ref{observation:identifiability:lemma:1}), which leads to $\beta(s_m)\geq0$ since $|\cos(\alpha(s_m))|\leq 1$. Notice that (\ref{before_proof:identifiability:eq2}) implies that $\alpha(s_m)=2k_m\pi$, where $k_m\in\ZZ$, if and only if $\beta(s_m)=0$. We now claim that for all $m\in \ZZ$, if $\beta(s_m)=0$, then $k_m=0$. Without loss of generality, assume $k_m>0$. Since $\alpha\in C^2(\RR)$, there exists $t'\in(t_{m-1},s_m)$ so that $\alpha(t')=\pi$ and hence 
\begin{align*}
a(t')\cos(\phi(t'))&\,= (a(t')+\beta(t'))\cos(\phi(t')+\alpha(t'))=-(a(t')+\beta(t'))\cos(\phi(t')),
\end{align*}
which is absurd since $\cos(\phi(t'))\neq 0$, (\ref{proof:2.1:Aeps:acond1}) and (\ref{proof:2.1:Aeps:acond2}). 
Next we claim that for all $m\in \ZZ$, if $\beta(s_m)>0$, then $|\alpha(s_m)|<\pi/2$. Indeed, since $0<\cos(\alpha(s_m))=\frac{a(s_m)}{a(s_m)+\beta(s_m)}<1$ by (\ref{before_proof:identifiability:eq2}), we know $\alpha(s_m)\in(-\pi/2,\pi/2)+2n_m\pi$, where $n_m\in\ZZ$. By the same argument as in the above, if $n_m>0$, there exists $t'\in(t_{m-1},s_m)$ so that $\alpha(t')=\pi$ and hence 
\begin{align*}
a(t')\cos(\phi(t'))&\,=\, (a(t')+\beta(t'))\cos(\phi(t')+\alpha(t'))=\,-(a(t')+\beta(t'))\cos(\phi(t')),
\end{align*}
which is absurd since $\cos(\phi(t'))\neq 0$, (\ref{proof:2.1:Aeps:acond1}) and (\ref{proof:2.1:Aeps:acond2}). We have thus complete the first part of the proof.

Before going to the second part of the proof, notice that by (\ref{before_proof:identifiability:eq2}), the behavior of $\frac{\beta(s_m)}{a(s_m)}$ can be further bounded by Taylor's expansion. Indeed, for each $m\in\ZZ$, since $|\alpha(s_m)|<\pi/2$, we have
\begin{equation}\label{observation:identifiability:cos1}
1-\frac{1}{2}\alpha(s_m)^2<\frac{1}{1+\frac{\beta(s_m)}{a(s_m)}}<1-\frac{1}{2}\alpha(s_m)^2+\frac{1}{24}\alpha(s_m)^4,
\end{equation}
which comes from Taylor's expansion of the cosine function around $0$.

Also notice that by the mean value theorem, for all $k\in\ZZ$, we have for some $t'_m\in[t_m,t_{m+1}]$
$$
\frac{\phi(t_{m+1})-\phi(t_m)}{t_{m+1}-t_m}=\frac{\pi}{t_{m+1}-t_m}=\phi'(t'_m), 
$$ 
which means that 
\begin{equation}\label{proof:identifiability:torder}
(t_{m+1}-t_m)\phi'(t'_m)=\pi.
\end{equation} 

To finish the second part of the proof, we have to consider the conditions (\ref{proof:2.1:Aeps:boundcond1}) and (\ref{proof:2.1:Aeps:boundcond2}) and show that $|\alpha'(t)|\leq 3\pi\epsilon$, $|\alpha(t)|\leq \frac{4\pi^2\epsilon}{\phi'(t)}$ and $|\beta(t)|< 3\pi \epsilon$ for all $t\in\RR$. Suppose there existed $t'\in(t_m,t_{m+1})$  for some $m\in\ZZ$ so that $|\alpha'(t')|>3\pi \epsilon$. Without loss of generality, we assume $\alpha'(t')>0$.  By the fundamental theorem of calculus and (\ref{proof:2.1:Aeps:cond6}), for any $t\in (t_m,t')$, we know 
\begin{align*}
|\alpha'(t')-\alpha'(t)|&\,\leq \int^{t'}_{t}|\alpha''(u)|\ud u\leq \epsilon\int^{t'}_{t}(2\phi'(u)+\alpha'(u))\ud u	\\
&\,= \epsilon[2\phi(t')-2\phi(t)+\alpha(t')-\alpha(t)]\leq 3\pi\epsilon,
\end{align*}
where the last inequality holds due to (\ref{proof:2.1:alpha:bound}) and the fact that $\phi(t')-\phi(t)\leq \phi(t_{m+1})-\phi(t_m)=\pi$.
Thus, $\alpha'(t)>0$ for all $t\in[t_m,t_{m+1}]$, 
which contradicts the fact that $\alpha'(t)$ must change sign inside $[t_m,t_{m+1}]$ as shown in (\ref{proof:2.1:alpha:changesign}).

Then we show that $|\alpha(t)|\leq \frac{4\pi^2\epsilon}{\phi'(t)}$ for all $t$. Recall that $\alpha(t_m)=0$ for all $m\in\ZZ$. If there exists $t'\in(t_m,t_{m+1})$ for some $m\in\ZZ$ so that $\alpha(t')> \frac{4\pi^2\epsilon}{\phi'(t')}$, by the mean value theorem and (\ref{proof:identifiability:torder}) there exists $t'', t'''\in[t_m,t']$ so that 
\begin{align*}
\alpha'(t'')&\,=\frac{\alpha(t')-\alpha(t_m)}{t'-t_m}=\frac{\alpha(t')}{t'-t_m}> \frac{4\pi^2\epsilon}{\phi'(t')}\frac{\phi'(t''')}{\pi}>4\pi\big[1-\frac{\epsilon}{\phi'(t')}\big]\epsilon,
\end{align*}
where the last inequality holds due to 
\begin{align*}
|\phi'(t')-\phi'(t''')|&\,\leq \int^{t'}_{t'''}|\phi''(u)|\ud u\leq \epsilon\int^{t'}_{t'''}\phi'(u)\ud u= \epsilon(\phi(t')-\phi(t'''))\leq \pi\epsilon.
\end{align*}
Since $c_1\gg \epsilon$, we know $4\pi\big[1-\frac{\epsilon}{\phi'(t')}\big]>3\pi$, which is a contradiction.

Suppose there exists $t'$ so that $\beta(t')> 3\pi \epsilon$. Take $m\in\ZZ$ so that $t'\in(s_m,s_{m+1})$. Without loss of generality, we assume $t'<t_m$. Take $t\in(s_m,t')$ and we have by the fundamental theorem of calculus
\begin{align*}
|\beta(t')-\beta(t)|&~\leq \int_{t}^{t'}|\beta'(u)|\ud u\leq \epsilon[2\phi(t')-2\phi(t)+\alpha(t')-\alpha(t)]\\
&~\leq \epsilon[2\phi(s_{m+1})-2\phi(s_m)+\alpha(t')-\alpha(t)]\leq 3\pi\epsilon,
\end{align*}
which leads to, in particular, $\beta(t_{m})>0$. 
By (\ref{observation:identifiability:cos1}), we know 
$$
|\alpha(t_{m})|>\sqrt{\frac{2\beta(t_{m})}{a(t_{m})+\beta(t_{m})}}>0,
$$ 
which contradicts to the fact that $\alpha(t_m)=0$. The proof is hence completed.

\section{Proof of Theorem \ref{theorem:identifiability:multiple}}
To prove the multiple-component version of the identifiability theory, we need the following lemma rewritten from Theorem 3.3 in \cite{daubechies_lu_wu:2010}, which is actually the noiseless version of Theorem \ref{section:theorem:stability} when there is no trend. Notice that, contrary to our setup, in \cite{daubechies_lu_wu:2010} the authors considered a broader class of functions called $\mathcal{A}_{\epsilon,d}$. 
\begin{lem}\label{lemma:identifiability:Wf_expansion}
Take $f(t)=\sum_{l=1}^KA_l(t)\cos [2\pi\phi_l(t)]\in \mathcal{A}^{c_1,c_2}_{\epsilon,d}$ and the mother wavelet $\psi$ the same as that in Theorem \ref{section:theorem:stability}. Then for $a\in[\frac{1-\Delta}{c_2},\frac{1+\Delta}{c_1}]$ we have
\begin{align}
\Big|W_f(a,b)&-\sum_{l=1}^KA_l(b)e^{i2\pi \phi_l(b)}\sqrt{a}\overline{\widehat{\psi}\left(a\phi'_l(b)\right)}\Big|\leq E_W\epsilon,\nonumber\\
\Big|-i\partial_bW_f(a,b)&-2\pi\sum_{l=1}^K \phi'_l(b)A_l(b)e^{i2\pi \phi_l(b)}\sqrt{a}\overline{\widehat{\psi}\left(a\phi'_l(b)\right)}\Big|\leq E_W'\epsilon\nonumber,
\end{align}
where $E_W$ and $E_W'$ are universal constants depending on the moments of $\psi$ and $\psi'$, $c_1$, $c_2$ and $d$. When $a\in Z_k(b):=[\frac{1-\Delta}{\phi_k'(b)},\frac{1+\Delta}{\phi_k'(b)}]$ and $|W_f(a,b)|\neq 0$, we have
\[
\left|\omega_f(a,b)-\phi'_k(b)\right|\leq \frac{E_\omega}{|W_f(a,b)|}\epsilon,
\]
where $E_\omega:=(2\pi)^{-1}E'_W+c_2E_W$.
Furthermore, when $\Gamma\geq 0$
\begin{align}
\Big|\int_{Z_k(b)} W_f(a,b)a^{-3/2}{\boldsymbol{\chi}}_{|W_f(a,b)|>\Gamma} \ud a-A_k(b)e^{i2\pi\phi_k(b)}\Big|\leq 2\sqrt{c_2}\big(E_W\epsilon+\Gamma \big)\Delta.\nonumber
\end{align}
\end{lem}

\underline{Remark:} Note that there are two error terms in the reconstruction formula. The first one $2\sqrt{c_2}E_W\epsilon\Delta$ originates from the definition of $\mathcal{A}^{c_1,c_2}_{\epsilon,d}$, and the second term comes from the thresholding parameter $\Gamma$. Since $\widehat{\psi}$ is smooth and compactly supported, the second term might be improved but the improvement is limited. We thus choose the current simple bound to make the proof clear and clean. 
\begin{proof}
By Estimate 3.5 in \cite{daubechies_lu_wu:2010}, we know that 
\begin{align}
W_f(a,b)=\left\{\begin{array}{ll}
A_k(b)e^{i2\pi \phi_k(b)}\sqrt{a}\overline{\widehat{\psi}\left(a\phi'_k(b)\right)}+C(a,b)\epsilon &\mbox{when }a\in Z_k(b)\\
C(a,b)\epsilon & \mbox{otherwise },
\end{array}\right.\label{proof:Wf:for_trend}
\end{align}
where 
$C(a,b)\in\CC$ and 
\begin{equation}
|C(a,b)|\leq a^{3/2}\sum_{k=1}^K \Big\{\phi_k'(b)I_1+\frac{\epsilon c_2}{2} aI_2+\pi A_k(b) \left(aI_2|\phi'_k(b)|+\frac{\epsilon c_2}{3}a^{2}I_3\right)\Big\}.\nonumber
\end{equation}
In other words, we have
\begin{equation*}
W_f(a,b)=\sum_{k=1}^KA_k(b)e^{i2\pi \phi_k(b)}\sqrt{a}\overline{\widehat{\psi}\left(a\phi'_k(b)\right)}+
\epsilon C(a,b). 
\end{equation*}
Similarly, we have 
\begin{equation*}
-i\partial_bW_f(a,b)=2\pi\sum_{k=1}^K A_k(b)e^{i2\pi \phi_k(b)}\phi'_k(b)\sqrt{a}\overline{\widehat{\psi}\left(a\phi'_k(b)\right)}+\epsilon C'(a,b),
\end{equation*}
where 
\begin{equation}
|C'(a,b)|\leq a^{1/2}\sum_{k=1}^K\Big\{\phi_k'(b)I'_1+\frac{\epsilon c_2}{2}aI'_2+\pi A_k(b) \left(aI'_2|\phi'_k(b)|+\frac{\epsilon c_2}{3}a^{2}I'_3\right)\Big\}\nonumber.
\end{equation}
By the assumption of $\mathcal{A}^{c_1,c_2}_{\epsilon,d}$, we know that each function in $\mathcal{A}^{c_1,c_2}_{\epsilon,d}$ can have at most $\lceil\frac{\ln c_2-\ln c_1}{\ln(1+d)-\ln(1-d)}\rceil$ components. Thus, when $\frac{1-\Delta}{c_2}\leq a\leq \frac{1+\Delta}{c_1}$, $C(a,b)$ is bounded by 
\begin{align}\label{proof:definition:E_W}
E_W:=\frac{c_2}{c_1^{3/2}}\lceil\frac{\ln c_2-\ln c_1}{\ln(1+d)-\ln(1-d)}\rceil\Big\{I_1+\frac{I_2}{2c_1} +\pi \big(\frac{c_2}{c_1}I_2+\frac{c_2}{3c_1^2}I_3\big)\Big\}
\end{align}
and $C'(a,b)$ is bounded by 
\begin{align}\label{proof:definition:E_W'}
E_W':=\frac{c_2}{c_1^{1/2}}\lceil\frac{\ln c_2-\ln c_1}{\ln(1+d)-\ln(1-d)}\rceil\Big\{I'_1+\frac{I'_2}{2c_1}+\pi \big(\frac{c_2}{c_1}I'_2+\frac{c_2}{3c_1^2}I'_3\big)\Big\}.
\end{align}
The estimation of $\omega_f(a,b)$ is the same as that in Theorem 3.3 in \cite{daubechies_lu_wu:2010}, so we skip it. 
By the above expansion for $W_f(a,b)$ and a direct calculation, we have
\begin{align}
&\big|\int_{Z_k(b)} W_f(a,b)a^{-3/2}{\boldsymbol{\chi}}_{|W_f(a,b)|>\Gamma} \ud a-A_k(b)e^{i2\pi\phi_k(b)}\big|\label{proof:lemma:Wfapproximate:reconstruction}\\
\leq&\,\left|A_k(b)e^{i2\pi \phi_k(b)}\int_{Z_k(b)}\sqrt{a}\overline{\widehat{\psi}(a\phi'_k(b))}a^{-3/2}\ud a-A_k(b)e^{i2\pi\phi_k(b)}\right|\nonumber\\
&+\epsilon\int_{Z_k(b)}E_W a^{-3/2}\ud a+\int_{Z_k(b)}|W_f(a,b)|a^{-3/2}{\boldsymbol{\chi}}_{|W_f(a,b)|\leq \Gamma} \ud a\nonumber,
\end{align}
where the first term disappears since $\int_{Z_k(b)}\sqrt{a}\overline{\widehat{\psi}(a\phi'_k(b))}a^{-3/2}\ud a=1$ by assumption and the second term is bounded by $2c^{1/2}_2 E_W\epsilon\Delta$. We simply bound the third term by
$$
\int_{Z_k(b)}|W_f(a,b)|a^{-3/2}{\boldsymbol{\chi}}_{|W_f(a,b)|\leq \Gamma} \ud a\leq \Gamma\int_{Z_k(b)}a^{-3/2}\ud a\leq 2\sqrt{c_2}\Gamma \Delta
$$
As a result, (\ref{proof:lemma:Wfapproximate:reconstruction}) is bounded by $2\sqrt{c_2}\big(E_W\epsilon+\Gamma \big)\Delta$, as is claimed.
\end{proof}

Next we consider the discretized case. 
\begin{lem}\label{lemma:identifiability:Wf_expansion_discrete}
Take $\boldsymbol{f}=\{f(n\tau )\}_{n\in\ZZ}$, where $f(t)=\sum_{l=1}^KA_l(t)\cos [2\pi\phi_l(t)]\in \mathcal{A}^{c_1,c_2}_{\epsilon,d}$ and $\tau$, $0<\tau \leq \frac{1-\Delta}{1+\Delta}\frac{1}{c_2}$, is the sampling interval. Suppose further that $A_l\in C^2(\RR)$ and $\sup_{t\in\RR}|A''(t)|\leq \epsilon c_2$ for all $l=1,\ldots,K$. Take the mother wavelet the same as that in Theorem \ref{section:theorem:stability}. Then for $a\in[\frac{1-\Delta}{c_2},\frac{1+\Delta}{c_1}]$ we have
\begin{align}
\Big|W_{\boldsymbol{f}}(a,b)&-\sum_{l=1}^KA_l(b)e^{i2\pi \phi_l(b)}\sqrt{a}\overline{\widehat{\psi}\left(a\phi'_l(b)\right)}\Big|\leq (E_{\tau ,W}\tau ^2+E_W)\epsilon,\nonumber\\
\Big|-i\partial_bW_{\boldsymbol{f}}(a,b)&-2\pi\sum_{l=1}^K \phi'_l(b)A_l(b)e^{i2\pi \phi_l(b)}\sqrt{a}\overline{\widehat{\psi}\left(a\phi'_l(b)\right)}\Big|\leq (E_{\tau ,W}'\tau ^2+E_{W}')\epsilon\nonumber,
\end{align}
where $E_{\tau ,W}$ and $E_{\tau ,W}'$ are universal constants depending on the first three moments of $\psi$ and $\psi'$, $c_1$, $c_2$ and $d$. When $a\in Z_k(b):=[\frac{1-\Delta}{\phi_k'(b)},\frac{1+\Delta}{\phi_k'(b)}]$ and $|W_{\boldsymbol{f}}(a,b)|\neq 0$, we have
\[
\left|\omega_{\boldsymbol{f}}(a,b)-\phi'_k(b)\right|\leq \frac{E_{\tau ,\omega}}{|W_{\boldsymbol{f}}(a,b)|}\epsilon,
\]
where $E_{\tau ,\omega}:=E_\omega+(2\pi)^{-1}[E_{\tau ,W}'+c_2E_{\tau ,W}]\tau ^2$.
Furthermore, when $\Gamma\geq 0$, 
\begin{align}
\Big|\int_{Z_k(b)} W_{\boldsymbol{f}}(a,b)a^{-3/2}{\boldsymbol{\chi}}_{|W_{\boldsymbol{f}}(a,b)|>\Gamma} \ud a-A_k(b)e^{i2\pi\phi_k(b)}\Big|\leq 2\sqrt{c_2}\big(  \tau ^2E_{\tau ,W}\epsilon+E_W\epsilon+\Gamma \big)\Delta.\nonumber
\end{align}
\end{lem}
We further denote the following universal constant to bound the error introduced by the $\mathcal{A}_{\epsilon,d}^{c_1,c_2}$ and the finite sampling rate shown in Lemma \ref{lemma:identifiability:Wf_expansion_discrete}:
\begin{align}
E_{\tau,0}:=\max\big\{E_{\tau,W},\,E_{\tau,W'},\,E_{\tau,\omega},\,2\sqrt{c_2}\tau^2E_{\tau,W}\big\},\label{proof:definition_error_constant_E0tau}
\end{align}
which is a universal constant depending on the moments of $\psi$ and $\psi'$, $c_1$, $c_2$ and $d$.

\underline{Remark:} The remark after Lemma \ref{lemma:identifiability:Wf_expansion} holds here. Moreover, notice that when the sampling interval $\tau$ is small enough, the result is essentially the same as that in Lemma \ref{lemma:identifiability:Wf_expansion} except the error introduced by the discretization.
\begin{proof}
By the Poisson formula, since $\psi\in\mathcal{S}$, for $a>0$ and $b\in\RR$ we have
\begin{align}
W_{\boldsymbol{f}}(a,b)=&\,\tau \sum_{j\in\ZZ}f(j\tau )\overline{\psi_{a,b}(j\tau)}=\,\sum_{n\in\ZZ}\mathcal{F}\left\{f(t)\overline{\psi_{a,b}(t)}\right\}\Big(\frac{n}{\tau }\Big),
\end{align}
where $\mathcal{F}$ means the Fourier transform. Note that $W_f(a,b)=\mathcal{F}\left\{f(t)\overline{\psi_{a,b}(t)}\right\}(0)$. Thus, the difference between $W_f(a,b)$ and $W_{\boldsymbol{f}}(a,b)$ is $\sum_{n\in\ZZ,n\neq 0}W^{(n/\tau )}_f(a,b)$, where 
\[
W^{(n/\tau )}_f(a,b):=\mathcal{F}\left\{f(t)\overline{\psi_{a,b}(t)}\right\}\Big(\frac{n}{\tau }\Big).
\]

Before proceeding, we evaluate the continuous wavelet transform of the harmonic function $A_k(b)\cos(2\pi(\phi_k(b)-b\phi_k'(b))+2\pi\phi_k'(b)t)$ modulated by $e^{i2\pi \frac{-nt}{\tau}}$, where $b\in\RR$ is fixed, $k\in\{1,\ldots,K\}$ and $n\in\ZZ\backslash\{0\}$. Note that this is a harmonic function. By a direct calculation, we have
\begin{align}
&\,\int_\RR A_k(b)\cos\Big(2\pi(\phi_k(b)-b\phi_k'(b))+2\pi\phi_k'(b)t\Big)\overline{\psi_{a,b}(t)e^{i2\pi \frac{nt}{\tau }}}\ud t\nonumber\\
=&\, A_k(b)\int_\RR \cos(2\pi(\phi_k(b)-b\phi_k'(b))+2\pi x)\frac{1}{\sqrt{a}}\overline{\psi\left(\frac{x-b\phi_k'(b)}{a\phi_k'(b)}\right)e^{i2\pi \frac{nx}{\phi'_k(b)\tau }}}\frac{1}{\phi_k'(b)}\ud x\nonumber\\
=&\, A_k(b)\int_\RR \cos(2\pi(\phi_k(b)-u))\frac{1}{\sqrt{a}}\overline{\psi\left(\frac{u}{a\phi_k'(b)}\right)e^{i2\pi\frac{n(u+b\phi'_k(b))}{\phi'_k(b)\tau }}}\frac{1}{\phi_k'(b)}\ud u\nonumber\\
=&\, \frac{A_k(b)}{2}e^{-i2\pi\big(\frac{nb}{\tau }-\phi_k(b)\big)}\int_\RR e^{-i2\pi u}\frac{1}{\sqrt{a}}\overline{\psi\left(\frac{u}{a\phi_k'(b)}\right)e^{i2\pi \frac{n}{\phi'_k(b)\tau }u}}\frac{1}{\phi_k'(b)}\ud u\nonumber\\
=&\, \frac{A_k(b)}{2}e^{i2\pi \big(\phi_k(b)-\frac{nb}{\tau }\big)}\sqrt{a}\overline{\widehat{\psi}\Big(a\big(\phi_k'(b)+\frac{n}{\tau }\big)\Big)}=0,\nonumber
\end{align}
where the last equality holds since $\tau \leq \frac{1-\Delta}{1+\Delta}\frac{1}{c_2}$ and $\frac{1-\Delta}{2\Delta}\frac{1}{c_2}>\frac{1-\Delta}{1+\Delta}\frac{1}{c_2}$.

We will bound $\sum_{n\in\ZZ,n\neq 0}W^{(n/\tau )}_f(a,b)$ step by step. The main step is approximating the non-harmonic function $A_k(t)\cos\big(2\pi\phi_k(t)\big)$ by the harmonic one $A_k(b)\cos\big(2\pi(\phi_k(b)-b\phi_k'(b))+2\pi\phi_k'(b)t\big)$.
We start from preparing some bounds which measure the error when we approximate $\cos(2\pi\phi_k(t))$ by $\cos(2\pi(\phi_k(b)-b\phi_k'(b))+2\pi\phi_k'(b)t))$ and $A_k(t)$ by $A_k(b)$. Clearly when $|\delta|\geq 0$, we have $|\cos(x+\delta)-\cos(x)|=\big| \int_{x}^{x+\delta}\sin(u)\ud u\big|\leq |\delta|$, which leads to the following bound when $j=0,1,2$:
\begin{align}
&\int_\RR \big|\big(\cos(2\pi\phi(t))-\cos(2\pi(\phi_k(b)-b\phi_k'(b))+2\pi\phi_k'(b)t))\big)\overline{\psi^{(j)}_{a,b}(t)}\big|\ud t\nonumber\\
=&\,\int_\RR \Big|\Big( \cos\big[ 2\pi\big(\phi_k(b)-b\phi_k'(b)+\phi_k'(b)t+\int^{t-b}_0\left[\phi_k'(b+u)-\phi_k'(b)\right]\ud u\big)\big]\nonumber\\
&\qquad-\cos(2\pi(\phi_k(b)-b\phi_k'(b))+2\pi\phi_k'(b)t)) \Big)\overline{\psi^{(j)}_{a,b}(t)}\Big|\ud t\nonumber\\
\leq&\,\int_\RR \Big|\int^{t-b}_0[\phi_k'(b+u)-\phi_k'(b)]\ud u\Big|\big|\overline{\psi^{(j)}_{a,b}(t)}\big|\ud t\nonumber\\
\leq&\, \epsilon\int_\RR\Big( \frac{1}{2}|t-b|^2|\phi'_k(b)|+\frac{\epsilon c_2}{6}|t-b|^3\Big)\big|\overline{\psi^{(j)}_{a,b}(t)}\big|\ud t\nonumber\\
\leq&\, \big[\frac{c_2}{2}a^{5/2-j}I^{(j)}_2+\frac{c_2}{6}a^{7/2-j}I^{(j)}_3\big]\epsilon\label{Clemma1:app1},
\end{align}
where in the last inequality we use the fact that $\epsilon\ll 1$ to simply the bound.
Similarly, by the Taylor expansion, $\epsilon\ll1$ and the assumptions of $\mathcal{A}^{c_1,c_2}_{\epsilon,d}$, we have the following bounds:
\begin{align}
&\int_\RR |A_k(b)-A_k(t)|\big|\overline{\psi_{a,b}(t)}\big|\ud t\nonumber\\
\leq&\,\epsilon\int_\RR |t-b|\Big(|\phi'_k(b)|+\frac{1}{2}\epsilon c_2|t-b|\Big)\big|\overline{\psi_{a,b}(t)}\big|\ud t\nonumber\\
\leq& \Big[c_2a^{3/2}I^{(0)}_1+\frac{1}{2}\epsilon c_2 a^{5/2}I^{(0)}_2\Big]\epsilon\label{Clemma1:app2}
\end{align}
and
\begin{align}
&\int_\RR |\phi'_k(b)-\phi'_k(t)|\big|\overline{\psi_{a,b}(t)}\big|\ud t\nonumber\\
\leq&\,\epsilon\int_\RR |t-b|\Big(|\phi'_k(b)|+\frac{1}{2}\epsilon c_2|t-b|\Big)\big|\overline{\psi_{a,b}(t)}\big|\ud t\nonumber\\
\leq&\, \Big[c_2a^{3/2}I^{(0)}_1+\frac{1}{2}\epsilon c_2 a^{5/2}I^{(0)}_2\Big]\epsilon\label{Clemma1:app3}.
\end{align}

With the above bounds, we first evaluate the difference between $W^{(n/\tau )}_{\cos(2\pi\phi_k(t))}(a,b)$ and $W^{(n/\tau )}_{\cos(2\pi(\phi_k(b)-b\phi_k'(b))+2\pi\phi_k'(b)t)}(a,b)$. Since $\psi_{a,b}\in\mathcal{S}$, by the integration by part and the regularity assumption for $\phi_k(t)$, we have
\begin{align}
&W^{(n/\tau )}_{\cos(2\pi\phi_k(t))}(a,b)-W^{(n/\tau )}_{\cos(2\pi(\phi_k(b)-b\phi_k'(b))+2\pi\phi_k'(b)t)}(a,b)\nonumber\\
=&\,\frac{\tau ^2}{4\pi^2 n^2}\int_\RR \Big\{-2\pi \phi''_k(t)\sin(2\pi\phi(t))\overline{\psi_{a,b}(t)}\nonumber\\
&\qquad-4\pi^2\big({\phi'_k(t)}^2\cos(2\pi\phi(t))-{\phi'_k(b)}^2\cos(2\pi(\phi_k(b)-b\phi_k'(b))+2\pi\phi_k'(b)t))\big)\overline{\psi_{a,b}(t)}\nonumber\\
&\qquad-4\pi\big(\phi'_k(t)\sin(2\pi\phi(t))-\phi'_k(b)\sin(2\pi(\phi_k(b)-b\phi_k'(b))+2\pi\phi_k'(b)t))\big)\overline{\psi^{(1)}_{a,b}(t)}\nonumber\\
&\qquad+\big(\cos(2\pi\phi(t))-\cos(2\pi(\phi_k(b)-b\phi_k'(b))+2\pi\phi_k'(b)t))\big)\overline{\psi^{(2)}_{a,b}(t)} \Big\}e^{i2\pi \frac{nt}{\tau }}\ud t\nonumber.
\end{align}
The first term in the last integral is simply bounded by $2\pi c_2I^{(0)}_0\epsilon$. By (\ref{Clemma1:app1}) with $j=2$, the fourth term in the integral is bounded by $c_2a^{1/2}\big[\frac{1}{2}I^{(2)}_2+\frac{1}{6}aI^{(2)}_3\big]\epsilon$.
Rewrite 
\begin{align*}
&{\phi'_k(t)}^2\cos(2\pi\phi(t))-{\phi'_k(b)}^2\cos(2\pi(\phi_k(b)-b\phi_k'(b))+2\pi\phi_k'(b)t))\\
=&\,({\phi'_k(t)}-{\phi'_k(b)})({\phi'_k(t)}+{\phi'_k(b)})\cos(2\pi\phi(t))\nonumber\\
&\,+{\phi'_k(b)}^2\big(\cos(2\pi\phi(t))-\cos(2\pi(\phi_k(b)-b\phi_k'(b))+2\pi\phi_k'(b)t))\big).
\end{align*}
By (\ref{Clemma1:app1}) with $j=1$ and (\ref{Clemma1:app3}), we bound the second term in the integral by 
\[
4\pi^2 c_2^2a^{3/2}\big[2I^{(1)}_1+ aI^{(1)}_2\epsilon+\frac{c_2}{2}aI^{(2)}_2+\frac{c_2}{6}a^{2}I^{(2)}_3\big]\epsilon.
\] 
Similarly, by rewriting 
\begin{align*}
&\phi'_k(t)\sin(2\pi\phi(t))-\phi'_k(b)\sin(2\pi(\phi_k(b)-b\phi_k'(b))+2\pi\phi_k'(b)t))\\
=&\,[\phi'_k(t)-\phi'_k(b)]\sin(2\pi\phi(t))\nonumber\\
&\,+\phi'_k(b)\big[\sin(2\pi\phi(t))-\sin(2\pi(\phi_k(b)-b\phi_k'(b))+2\pi\phi_k'(b)t))\big],
\end{align*} 
and by (\ref{Clemma1:app1}) with $j=0$ and (\ref{Clemma1:app3}), we bound the third term in the integral by \[
4\pi c_2^2 a^{3/2}\big[aI^{(0)}_1+\frac{1}{2} a^{2}I^{(0)}_2\epsilon+\frac{1}{2}I^{(1)}_2+\frac{1}{6}aI^{(1)}_3\big]\epsilon. 
\]
As a result, we bound the entire integral by $\epsilon E_{\tau ,W,\phi}$, where 
\begin{align*}
E_{\tau ,W,\phi}:=&2\pi c_2I^{(0)}_0+c_2a^{1/2}\big[\frac{1}{2}I^{(2)}_2+\frac{1}{6}aI^{(2)}_3\big]\\
&+4\pi^2 c_2^2a^{3/2}\big[2I^{(1)}_1+ aI^{(1)}_2\epsilon+\frac{c_2}{2}aI^{(2)}_2+\frac{c_2}{6}a^{2}I^{(2)}_3\big]\\
&+4\pi c_2^2 a^{3/2}\big[aI^{(0)}_1+\frac{1}{2} a^{2}I^{(0)}_2\epsilon+\frac{1}{2}I^{(1)}_2+\frac{1}{6}aI^{(1)}_3\big]<\infty
\end{align*}
is an universal constant depending only on the first three moments of $\psi$ and $\psi'$, $c_1,c_2$ and $d$ since $a\in \big[\frac{1-\Delta}{c_2},\frac{1+\Delta}{c_1}\big]$.

Again, the difference between $W^{(n/\tau )}_{A_k(t)\cos(2\pi\phi_k(t))}(a,b)$ and $W^{(n/\tau )}_{A_k(b)\cos(2\pi\phi_k(t))}(a,b)$ is evaluated by the integration by part:
\begin{align}
&W^{(n/\tau )}_{A_k(t)\cos(2\pi\phi_k(t))}(a,b)-W^{(n/\tau )}_{A_k(b)\cos(2\pi\phi_k(t))}(a,b)\nonumber\\
=&\, \frac{\tau ^2}{4\pi^2n^2}\int_\RR \Big\{A''_k(t)\cos(2\pi\phi_k(t))\overline{\psi_{a,b}(t)} -4\pi A_k'(t)\phi_k'(t)\sin 2\pi\phi_k(t) \overline{\psi_{a,b}(t)} \nonumber\\
&\qquad+ 2A'_k(t)\cos2\pi\phi_k(t)\overline{\psi^{(1)}_{a,b}(t)}  -2\pi (A_k(t)-A_k(b))\phi''_k(t)\sin 2\pi\phi_k(t)\overline{\psi_{a,b}(t)}  \nonumber\\
&\qquad-4\pi^2{\phi_k'(t)}^2(A_k(t)-A_k(b))\cos2\pi\phi_k(t)\overline{\psi_{a,b}(t)}\nonumber\\
&\qquad-4\pi (A_k(t)-A_k(b))\phi_k'(t)\sin2\pi\phi_k(t)\overline{\psi^{(1)}_{a,b}(t)} \nonumber\\
&\qquad+(A_k(t)-A_k(b))\cos(2\pi\phi_k(t))\overline{\psi^{(2)}_{a,b}(t)} \Big\}e^{-i2\pi \frac{nt}{\tau }}\ud t\label{only_place_for_regularity_of_Al}.
\end{align}
By applying (\ref{Clemma1:app2}) and the same arguments as before, the last integral is bounded by 
$\epsilon E_{\tau ,W,A}$, where 
\begin{align*}
E_{\tau ,W,A}:=&c_2\big[\sqrt{a}I_1^{(0)}+4\pi \sqrt{a}I^{(0)}_0 +2a^{-1/2}I^{(1)}_0\big]\\
&+c_2\big[2\pi \epsilon c_2 \sqrt{a}I_0^{(0)}+4\pi^2c_2^2\sqrt{a}I^{(0)}_0+4\pi c_2a^{-1/2}I^{(1)}_0+a^{-3/2}I^{(2)}_0\big]\\
&\quad\times\big(a^{3/2}I^{(0)}_1+\frac{\epsilon }{2}I^{(0)}_2\big)<\infty
\end{align*}
depending only on the first two moments of $\psi$ and $\psi'$, $c_1,c_2$ and $d$ since $a\in \big[\frac{1-\Delta}{c_2},\frac{1+\Delta}{c_1}\big]$. Note that the condition $|A''|\leq c_2\epsilon$ is used only for the first term. 

In conclusion, we get
\begin{align}
\Big|\sum_{n\in\ZZ,n\neq 0}W^{(n/\tau )}_{f}(a,b) \Big|
\leq\sum_{k=1}^K\sum_{n\neq 0}\frac{\tau ^2\epsilon}{4\pi^2 n^2}(E_{\tau ,W,\phi}+E_{\tau ,W,A})\leq E_{\tau ,W}\tau ^2\epsilon,\nonumber
\end{align}
where 
\begin{align*}
E_{\tau ,W}:=\frac{1}{24}(E_{\tau ,W,\phi}+E_{\tau ,W,A})\lceil\frac{\ln c_2-\ln c_1}{\ln(1+d)-\ln(1-d)}\rceil
\end{align*} 
is an universal constant depending only on the first three moments of $\psi$ and $\psi'$, $c_1,c_2$ and $d$. Note that we use the fact that $\sum_{n=1}^\infty n^{-2}=\frac{\pi^2}{6}$ and that each function in $\mathcal{A}^{c_1,c_2}_{\epsilon,d}$ can have at most $\lceil\frac{\ln c_2-\ln c_1}{\ln(1+d)-\ln(1-d)}\rceil$ components, that is, $K\leq \lceil\frac{\ln c_2-\ln c_1}{\ln(1+d)-\ln(1-d)}\rceil$. Eventually, we have 
\begin{align*}
\big|W_{\boldsymbol{f}}(a,b)-W_f(a,b)\big|\leq \Big|\sum_{n\neq 0}W^{(n/\tau )}_f(a,b)\Big|\leq E_{\tau ,W}\tau ^2\epsilon.
\end{align*}
By a similar argument, for which the detail is skipped, we have
\begin{align*}
\big|\partial_bW_{\boldsymbol{f}}(a,b)-\partial_bW_f(a,b)\big|\leq E_{\tau ,W}'\tau ^2\epsilon,
\end{align*}
where $E_{\tau ,W}'$ is an universal constant depending only on the first three moments of $\psi$ and $\psi'$, $c_1,c_2$ and $d$. Combining the above two results with Lemma \ref{lemma:identifiability:Wf_expansion}, we get the claim.

Therefore, when $a\in Z_k(b)$ and $|W_{\boldsymbol{f}}(a,b)|\neq 0$, we have
\begin{align}
&\omega_{\boldsymbol{f}}(a,b)-\phi'_k(b)\nonumber\\
=&\,\frac{-i\partial_bW_{\boldsymbol{f}}(a,b)-2\pi\phi'_k(b)A_k(b)e^{i2\pi\phi_k(b)}\sqrt{a}\widehat{\psi}(a\phi'_k(b))}{2\pi W_{\boldsymbol{f}}(a,b)}\nonumber\\
&\qquad+\frac{\big[A_k(b)e^{i2\pi\phi_k(b)}\sqrt{a}\widehat{\psi}(a\phi'_k(b))-W_{\boldsymbol{f}}(a,b)\big]\phi'_k(b)}{W_{\boldsymbol{f}}(a,b)}\nonumber\\
\leq&\,\frac{\big|\partial_bW_{\boldsymbol{f}}(a,b)-\partial_bW_{f}(a,b)\big|+\big|\partial_bW_{f}(a,b)-2\pi\phi'_k(b)A_k(b)e^{i2\pi\phi_k(b)}\sqrt{a}\widehat{\psi}(a\phi'_k(b))\big|}{2\pi |W_{\boldsymbol{f}}(a,b)|}\nonumber\\
&\qquad+\frac{\Big[\big|A_k(b)e^{i2\pi\phi_k(b)}\sqrt{a}\widehat{\psi}(a\phi'_k(b))-W_{f}(a,b)\big|+\big|W_{f}(a,b)-W_{\boldsymbol{f}}(a,b)\big|\Big]\phi'_k(b)}{|W_{\boldsymbol{f}}(a,b)|}\nonumber\\
\leq&\,\frac{[E_{W}'+E_{\tau ,W}'\tau ^2]+2\pi c_2[E_W+E_{\tau ,W}\tau ^2]}{2\pi|W_{\boldsymbol{f}}(a,b)|}\epsilon=\frac{E_{\tau ,\omega}}{|W_{\boldsymbol{f}}(a,b)|}\epsilon\nonumber,
\end{align}
where $E_{\tau,\omega}:=(2\pi)^{-1}[E_{W}'+E_{\tau ,W}'\tau ^2]+c_2[E_W+E_{\tau ,W}\tau ^2]$, as is claimed.

By the above expansion for $W_{\boldsymbol{f}}(a,b)$ and $W_f(a,b)$ and a direct calculation as that for (\ref{proof:lemma:Wfapproximate:reconstruction}), we have
\begin{align}
&\big|\int_{Z_k(b)} W_{\boldsymbol{f}}(a,b)a^{-3/2}{\boldsymbol{\chi}}_{|W_{\boldsymbol{f}}(a,b)|>\Gamma} \ud a-A_k(b)e^{i2\pi\phi_k(b)}\big|\label{proof:lemma2:Wfapproximate:reconstruction}\\
\leq&\,\left|A_k(b)e^{i2\pi \phi_k(b)}\int_{Z_k(b)}\sqrt{a}\overline{\widehat{\psi}(a\phi'_k(b))}a^{-3/2}\ud a-A_k(b)e^{i2\pi\phi_k(b)}\right|\nonumber\\
&+\epsilon\int_{Z_k(b)}[\tau ^2E_{\tau ,W}+E_W] a^{-3/2}\ud a+\int_{Z_k(b)}|W_{\boldsymbol{f}}(a,b)|a^{-3/2}{\boldsymbol{\chi}}_{|W_{\boldsymbol{f}}(a,b)|\leq \Gamma} \ud a\nonumber,
\end{align}
where the first term disappears since $\int_{Z_k(b)}\sqrt{a}\overline{\widehat{\psi}(a\phi'_k(b))}a^{-3/2}\ud a=1$ and the second term is bounded by $2\sqrt{c_2} [\tau ^2E_{\tau ,W}+E_W]\epsilon\Delta$. We simply bound the third term by
$$
\int_{Z_k(b)}|W_{\boldsymbol{f}}(a,b)|a^{-3/2}{\boldsymbol{\chi}}_{|W_{\boldsymbol{f}(a,b)}|\leq \Gamma} \ud a\leq \Gamma\int_{Z_k(b)}a^{-3/2}\ud a\leq 2\sqrt{c_2}\Gamma \Delta.
$$
As a result, we get the claim.

\end{proof}

\begin{proof}[Theorem \ref{theorem:identifiability:multiple}]
We start from proving that $N=M$. Choose $\psi$ to be a Schwartz function so that 
\begin{align}\label{special:psi}
\widehat{\psi}(\xi)=e\exp\left\{\frac{1}{\big(\frac{1-\xi}{\Delta}\big)^2-1} \right\}{\boldsymbol{\chi}}_{[1-\Delta,1+\Delta]}(\xi),
\end{align}
where ${\boldsymbol{\chi}}$ is the indicator function and $\Delta=d/2(1+d)$. It is well known that $\widehat{\psi}(\xi)$ is real, smooth, compactly supported, monotonically decays when $\xi\in[1,1+\Delta]$, monotonically increases when $\xi\in[1-\Delta,1]$ and has support $[1-\Delta,1+\Delta]$. Note that $\widehat{\psi}$ has only one maximal point at $\xi=1$ and $\widehat{\psi}(1)=1$. Fix $f\in\mathcal{A}^{c_1,c_2}_{\epsilon,d}$. Suppose $f$ has two representations, both are in $\mathcal{A}^{c_1,c_2}_{\epsilon,d}$:
\[
f(t)=\sum_{l=1}^Na_l(t)\cos [2\pi\phi_l(t)]=\sum_{l=1}^MA_l(t)\cos[2\pi\varphi_l(t)]\in \mathcal{A}^{c_1,c_2}_{\epsilon,d}.
\]
Denote $Z_l(b):=\big[\frac{1-\Delta}{\phi'_l(b)}\,,\frac{1+\Delta}{\phi'_l(b)} \big]$, $l=1,\ldots,N$, and $Y_k(b):=\big[\frac{1-\Delta}{\varphi'_k(b)}\,,\frac{1+\Delta}{\varphi'_k(b)} \big]$, $k=1,\ldots M$. Then by Lemma \ref{lemma:identifiability:Wf_expansion} we have
\begin{align*}
\big|W_f(a,b)-\sum_{l=1}^Na_l(b)e^{i2\pi \phi_l(b)}\sqrt{a}\widehat{\psi}\left(a\phi'_l(b)\right)\big|\leq
E_m\epsilon,
\end{align*}
\[
\big|W_f(a,b)-\sum_{l=1}^MA_l(b)e^{i2\pi \varphi_l(b)}\sqrt{a}\widehat{\psi}\left(a\varphi'_l(b)\right)\big|\leq E_m\epsilon,
\]
where 
$$
E_m\leq c\frac{c_2}{c_1^{3/2}}\lceil\frac{\ln c_2-\ln c_1}{\ln(1+d)-\ln(1-d)}\rceil\Big\{1+\frac{1}{2c_1} +\pi \big(\frac{c_2}{c_1}+\frac{c_2}{3c_1^2}\big)\Big\},
$$ 
and $0<c<\infty$ is the maximum of the first three moments of the chosen $\psi$ (\ref{special:psi}). 
In other words, $E_m$ is a universal constant whose bound depends only on $c_1$, $c_2$ and $d$. Note that in this proof we do not concern ourselves with the optimal $\psi$ and $c$ but simply choose the convenient one.
Denote 
$$
L(a,b):=\sum_{l=1}^Na_l(b)e^{i2\pi \phi_l(b)}\sqrt{a}\widehat{\psi}\left(a\phi'_l(b)\right),
$$ 
$$
R(a,b):=\sum_{l=1}^MA_l(b)e^{i2\pi \varphi_l(b)}\sqrt{a}\widehat{\psi}\left(a\varphi'_l(b)\right).
$$
Since the continuous wavelet transform of $f$ is unique, we have
\begin{equation}\label{proof:identifiability:multiple:1}
|L(a,b)-R(a,b)|\leq 2E_m\epsilon.
\end{equation}

By the profile of the chosen $\psi$, there exist $N$ subintervals, $I_l\subset Z_l$, $l=1,\ldots, N$, around $\frac{1}{\phi'_l(b)}$ 
so that on $I_l$, $|L(a,b)|>a_l(b)\sqrt{a}/2>\frac{\sqrt{1-\Delta}c_1}{2\sqrt{c_2}}$. Clearly $\frac{1}{\phi'_l(b)}$, $l=1,\ldots,N$, are the maximal points of $|L(a,b)|$, which are dyadically separated by (\ref{Aepsd:cond:1}). Similarly there exist $M$  subintervals, $J_l\subset Y_l$, $l=1,\ldots,M$, around 
$\frac{1}{\varphi'_l(b)}$ so that on $J_l$, $|R(a,b)|>A_l(b)\sqrt{a}/2>\frac{\sqrt{1-\Delta}c_1}{2\sqrt{c_2}}$. Also $\frac{1}{\varphi'_l(b)}$, $l=1,\ldots,M$, are the maximal points of $|R(a,b)|$, which are dyadically separated by (\ref{Aepsd:cond:1}). Since $E_m$ is a universal constant, when $\epsilon$ is small enough, 
the equality in (\ref{proof:identifiability:multiple:1}) cannot hold if $M\neq N$ or $Z_l(b)\cap Y_l(b)\neq\emptyset$ for any $l=1,\ldots,N$. 
Thus we have
\[
f(t)=\sum_{l=1}^Na_l(t)\cos [2\pi\phi_l(t)]=\sum_{l=1}^NA_l(t)\cos[2\pi\varphi_l(t)]\in \mathcal{A}^{c_1,c_2}_{\epsilon,d},
\]
and we obtain the first part of the proof.


Next, we show the second part of the proof, that is, $|a_1(b)-A_1(b)|\leq 2\sqrt{c_2}E_m\epsilon$, $|\phi'_1(b)-\varphi'_1(b)|\leq 10(1+d)c_2^{3/2}E_m\epsilon$ and $|\phi_1(b)-\varphi_1(b)|\leq 9\sqrt{c_2}E_m\epsilon$. Note that it is clear from (\ref{proof:identifiability:multiple:1}) that the set $I_l(b)\cap J_k(b)=\emptyset$ for all $l\neq k$ and $I_l(b)\cap J_l(b)\neq \emptyset$. By (\ref{proof:identifiability:multiple:1}) and the fact that $\widehat{\psi}$ is a real function, we have on $I_1(b)\cap J_1(b)$ that 
\begin{align}
\Big|a_1(b)\sqrt{a}\widehat{\psi}\left(a\phi'_1(b)\right)-A_1(b)e^{i2\pi(\varphi_1(b)-\phi_1(b))}\sqrt{a}\widehat{\psi}\left(a\varphi'_1(b)\right)\Big|\leq2E_m\epsilon,\label{proof:identifiability:multiple:q}
\end{align}
that is, $A_1(b)e^{i2\pi(\varphi_1(b)-\phi_1(b))}\sqrt{a}\widehat{\psi}\left(a\varphi'_1(b)\right)\in B_{2E_m\epsilon}(a_1(b)\sqrt{a}\widehat{\psi}\left(a\phi'_1(b)\right))$. Thus,
\begin{eqnarray}\label{proof:identifiability:multiple:3}
\Big|a_1(b)\widehat{\psi}\left(a\phi'_1(b)\right)-A_1(b)\widehat{\psi}\left(a\varphi'_1(b)\right)\Big|\leq\frac{2E_m}{\sqrt{a}}\epsilon.
\end{eqnarray}

Without loss of generality, assume $a_1(b)>A_1(b)$. When $a=1/\phi'_1(b)$, (\ref{proof:identifiability:multiple:3}) becomes
\begin{align}
&\Big|a_1(b)\widehat{\psi}\left(1\right)-A_1(b)\widehat{\psi}\left(\frac{\varphi'_1(b)}{\phi'_1(b)}\right)\Big|
=\Big|a_1(b)-A_1(b)\widehat{\psi}\left(\frac{\Phi'_1(b)}{\phi'_1(b)}\right)\Big|\leq2\sqrt{c_2}E_m\epsilon.\label{proof:identifiability:multiple:3aA}
\end{align}
Thus, if $a_1(b)-A_1(b)>2\sqrt{c_2}E_m\epsilon$, we get a contradiction. Indeed, $a_1(b)-A_1(b)>a_1(b)-A_1(b)\widehat{\psi}\left(\frac{\varphi'_1(b)}{\phi'_1(b)}\right)=\big|a_1(b)-A_1(b)\widehat{\psi}\left(\frac{\varphi'_1(b)}{\phi'_1(b)}\right)\big|$ since $1$ is the unique maximal point of $\widehat{\psi}$. 

To show the bound of $|\phi'_1(b)-\varphi'_1(b)|$, pick $a_0\in Z_1$ so that $\widehat{\psi}(a_0\phi_1'(b))-\widehat{\psi}(a_0\varphi'_1(b))=\frac{1}{2}(a_0\phi_1'(b)-a_0\varphi_1'(b))$ by the mean value theorem. Notice that $\|\widehat{\psi}'\|_{L^\infty}>1$ and $\widehat{\psi}\in C^\infty_c$. Then by the fact that $a_0\in \big[\frac{1-\Delta}{c_2},\frac{1+\Delta}{c_1}\big]$, (\ref{proof:identifiability:multiple:3}) becomes
\begin{align*}
-\frac{2\sqrt{c_2}E_m}{\sqrt{1-\Delta}}\epsilon\leq \frac{a_1(b)}{2}(a_0\phi_1'(b)-a_0\varphi_1'(b))+(a_1(b)-A_1(b))\widehat{\psi}(a_0\Phi'_1(b))\leq \frac{2\sqrt{c_2}E_m}{\sqrt{1-\Delta}}\epsilon.
\end{align*}
Since $a_1(b)>A_1(b)$ and $\widehat{\psi}>0$, by (\ref{proof:identifiability:multiple:3aA}) we have 
\begin{align*}
-\frac{2\sqrt{c_2}E_m}{\sqrt{1-\Delta}}\epsilon-2\sqrt{c_2}E_m\epsilon \leq \frac{a_1(b)a_0}{2}(\phi_1'(b)-\varphi_1'(b))\leq \frac{2\sqrt{c_2}E_m}{\sqrt{1-\Delta}}\epsilon+2\sqrt{c_2}E_m\epsilon,
\end{align*}
\begin{align*}
\Big|\frac{a_1(b)a_0}{2}(\phi_1'(b)-\varphi_1'(b))\Big|\leq 2\sqrt{c_2(1+d)}E_m\epsilon+2\sqrt{c_2}E_m\epsilon
\end{align*}
and hence
\begin{align*}
\Big|\phi_1'(b)-\varphi_1'(b)\Big|\leq \frac{2c^{3/2}_2}{c_1}(1+d)(\sqrt{1+d}+2)E_m\epsilon,
\end{align*}
where the last inequality holds since $\frac{1}{1-\Delta}<1+d$.

Moreover, (\ref{proof:identifiability:multiple:q}) and (\ref{proof:identifiability:multiple:3}) imply that
\begin{equation*}
|A_1(b)\hat{\psi}(a\varphi'_1(b))(1-e^{i2\pi(\varphi_1(b)-\phi_1(b))})|\leq 4\sqrt{c_2}E_m\epsilon.
\end{equation*} 
Thus, on $I_1(b)\cap J_1(b)$, we obtain
\begin{equation}\label{proof:identifiability:multiple:2-1}
|1-e^{i2\pi(\varphi_1(b)-\phi_1(b))}|\leq 8\sqrt{c_2}E_m\epsilon.
\end{equation} 
It is clear from (\ref{proof:identifiability:multiple:2-1}), $\phi_1,\varphi_1\in C^2$ and Taylor's expansion of the sine function that when $\epsilon$ is small enough, 
\begin{align*}
|\phi_1(b)-\varphi_1(b)|\leq 9\sqrt{c_2}E_m\epsilon.%
\end{align*}

Similar argument holds for all $k=2,\ldots, N$, and we conclude the theorem by denoting 
\begin{align}
E_I:=\max\big\{2\sqrt{c_2}E_m,\,2c_1^{-1}c^{3/2}_2(1+d)(\sqrt{1+d}+2)E_m,\, 9\sqrt{c_2}E_m\big\}.\label{proof:identifiability:multiple:error_constant}
\end{align}
\end{proof}


\section{Proof of Theorem \ref{section:theorem:stability} }
To prove Theorem \ref{section:theorem:stability}, we need the following technical lemma. It is aimed at adapting the delta method, for approximating the mean and variance of the ratio of two random variables,  to our setup. In particular, when we show the properties of $\omega_Y$, in part (ii) of Theorem \ref{section:theorem:stability} and part (ii) of Theorem \ref{section:theorem:ARMAstability}, we will encounter the problem discussed in this lemma. Note that we only assume existence of the variance but not the higher moments.
Also note that the conditions $\var(\zeta)\leq (9/32)^3$ and $|y_0|>\var(\zeta)^{1/3}$ are not crucial but chosen simply to simplify the final bound. 

\begin{lem}[Delta method]\label{proof:thm:lemma1:deltamethod}
Given two complex-valued random variables $\zeta'$ and $\zeta$ so that $\EE \zeta'=\EE \zeta=0$, $\var(\zeta)\leq (9/32)^3$ and $\var(\zeta')<\infty$.  
Fix two complex numbers $x_0$ and $y_0$ so that $|y_0|>\var(\zeta)^{1/3}$. Define $\zeta^\Omega:= \zeta{\boldsymbol{\chi}}_{\CC\backslash B_{|y_0|/4}(-y_0)}$, where $B_{|y_0|/4}(-y_0)$ is the ball of radius $|y_0|/4$ centered at $-y_0$ and $\boldsymbol{\chi}$ is the indicator function. Then we have the following relationship:
\begin{equation}\label{proof:thm:lemma1:deltamethod:exp}
\EE \left(\frac{x_0+\zeta'}{y_0+\zeta^\Omega}\right)=\frac{x_0}{y_0}+e_1,
\end{equation}
where $e_1\in\CC$ and 
\begin{align}
|e_1|\leq&\,\frac{1}{|y_0|^2}\Big(\Big|\frac{x_0}{y_0}\Big|21\var(\zeta)+4\sqrt{\var(\zeta)\var (\zeta')}\Big)\nonumber
\end{align}
and
\begin{align}\label{proof:thm:lemma1:deltamethod:var}
\Big|\var \left( \frac{x_0+\zeta'}{y_0+\zeta^\Omega} \right)\Big|\leq \frac{1}{|y_0|^2}\Big(252\Big|\frac{x_0}{y_0}\Big|^2\var(\zeta)+136\Big|\frac{x_0}{y_0}\Big|\sqrt{\var(\zeta)\var (\zeta')}+21\var (\zeta')\Big).
\end{align}
\end{lem}

\begin{proof}
We start from preparing some bounds for the proof. Denote $\ud F_{\zeta'}(y)$, $\ud F_{\zeta}(y)$, $\ud F_{\zeta^\Omega}(y)$ and $\ud F_{\zeta',\zeta^\Omega}(x,y)$ the measures associated with $\zeta'$, $\zeta$, $\zeta^\Omega$ and the vector $(\zeta',\zeta^\Omega)$ respectively. 

When $|y_0|>\var({\zeta})^{1/3}$, since the variance of $\zeta$ exists and $B_{|y_0|/4}(-y_0)\subset\CC\backslash B_{3|y_0|/4}(0)$, we have 
\begin{align}
\int_{B_{|y_0|/4}(-y_0)}\ud F_{\zeta}\leq \frac{16\var(\zeta)}{9|y_0|^2},\label{proof:thm:lemma1:deltamethod:taildecay}
\end{align}
which is bounded by $16\var(\zeta)^{1/3}/9\leq 1/2$ since we assume $\var(\zeta)\leq (9/32)^3$. By definition and the assumption that $\EE \zeta=0$, we have 
\begin{align}
\EE \zeta^\Omega:=\int y\ud F_{\zeta^\Omega}(y)=\frac{\int_{\CC\backslash B_{|y_0|/4}(-y_0)}y\ud F_{\zeta}(y)}{\int_{\CC\backslash B_{|y_0|/4}(-y_0)}\ud F_{\zeta}(y)}=\frac{\int_{B_{|y_0|/4}(-y_0)}y\ud F_{\zeta}(y)}{1-\int_{B_{|y_0|/4}(-y_0)}\ud F_{\zeta}(y)}.\nonumber
\end{align}
Since 
\begin{align}
\big|\int_{B_{|y_0|/4}(-y_0)}y\ud F_{\zeta}(y)\big|\leq \frac{5}{4}|y_0|\int_{B_{|y_0|/4}(-y_0)}\ud F_{\zeta}(y)\leq \frac{20\var(\zeta)}{9|y_0|}, \label{proof:delta_method:exp_of_Yball}
\end{align}
by (\ref{proof:thm:lemma1:deltamethod:taildecay}), we have 
\begin{align}\label{proof:delta_method:exp_of_YOmega}
|\EE \zeta^\Omega| \leq \frac{40\var(\zeta)}{9|y_0|}.
\end{align} 
Similarly, the second moment of $\zeta^\Omega$ is bounded by
\begin{align}
&\int |y|^2\ud F_{\zeta^\Omega}(y)=\,\frac{\int_{\CC\backslash B_{|y_0|/4}(-y_0)}|y|^2\ud F_{\zeta}(y)}{\int_{\CC\backslash B_{|y_0|/4}(-y_0)}\ud F_{\zeta}}
\leq 2\int_{\CC\backslash B_{|y_0|/4}(-y_0)}|y|^2\ud F_{\zeta}(y)\leq 2\var{\zeta}.\label{proof:delta_method:var_of_YOmega}
\end{align}
With the above preparation, 
now we come back to prove the Lemma. We have
\begin{align}
\EE\Big[ \frac{x_0+\zeta'}{y_0+\zeta^\Omega}\Big]-\frac{x_0}{y_0}
=\,\frac{x_0}{y_0}\Big\{\EE\Big[ \frac{1}{1+\frac{\zeta^\Omega}{y_0}}\Big]-1\Big\}+\frac{1}{y_0}\EE\Big[ \frac{\zeta'}{1+\frac{\zeta^\Omega}{y_0}} \Big].\label{proof:delta_method:error:bound}
\end{align}
The first term on the right hand side of (\ref{proof:delta_method:error:bound}) is bounded by
\begin{align}
&\left|\EE\Big[ \frac{1}{1+\frac{\zeta^\Omega}{y_0}}\Big]-1\right|=\left|\int \frac{1}{1+\frac{y}{y_0}}\ud F_{\zeta^\Omega}(y)-1\right|\nonumber\\
=&\,\left|\int \big(1-\frac{y}{y_0}\big)\ud F_{\zeta^\Omega}(y)-1+\frac{1}{y_0^2}\int \frac{y^2}{1+\frac{y}{y_0}}\ud F_{\zeta^\Omega}(y)\right|\nonumber\\
=&\,\left|\frac{\EE \zeta^\Omega}{y_0}+\frac{1}{y_0}\int\frac{y^2}{y_0+y}\ud F_{\zeta^\Omega}(y)\right|,
\end{align}
where the second equality comes from $\frac{1}{1+z}=1-z+\frac{z^2}{1+z}$ when $z\neq 1$. We simply bound the integral in the second term by the Holder's inequality:
$$
\Big|\int\frac{y^2}{y_0+y}\ud F_{\zeta^\Omega}(y)\Big|\leq 2 \Big|\int_{\CC\backslash B_{|y_0|/4}(-y_0)}\frac{y^2}{y_0+y}\ud F_{\zeta}(y)\Big|\leq \frac{8\var(\zeta)}{|y_0|}
$$
Thus the first term in (\ref{proof:delta_method:error:bound}) is bounded by
\begin{align}
\Big|\frac{x_0}{y_0}\Big|\left|\EE\Big[ \frac{1}{1+\frac{\zeta^\Omega}{y_0}} \Big]-1\right|\leq \Big|\frac{x_0}{y_0}\Big|\frac{13\var(\zeta)}{|y_0|^2}.\nonumber
\end{align}
By the same trick, the second term in (\ref{proof:delta_method:error:bound}) is bounded by:
\begin{align}
&\frac{1}{|y_0|}\left|\EE\Big[ \frac{\zeta'}{1+\frac{\zeta^\Omega}{y_0}} \Big]\right|=\frac{1}{|y_0|}\left|\int \frac{x}{1+\frac{y}{y_0}}\ud F_{\zeta',\zeta^\Omega}(x,y)\right|\nonumber\\
=&\,\frac{1}{|y_0|}\left|\int \big(x-\frac{xy}{y_0+y}\big)\ud F_{\zeta',\zeta^\Omega}(x,y)\right|\leq\frac{4\sqrt{\var(\zeta)\var (\zeta')}}{|y_0|^2},\nonumber
\end{align}
where the inequality holds due to $\int x\ud F_{\zeta',\zeta^\Omega}(x,y)=0$ and the Holder's inequality. 
Hence we conclude the expectation (\ref{proof:thm:lemma1:deltamethod:exp}).
Next we evaluate the variance (\ref{proof:thm:lemma1:deltamethod:var}). Expand 
\begin{align}
&\var\left[ \frac{x_0+\zeta'}{y_0+\zeta^\Omega} \right]=\,\EE\Big|\frac{x_0+\zeta'}{y_0+\zeta^\Omega}\Big|^2-\Big|\EE\Big(\frac{x_0+\zeta'}{y_0+\zeta^\Omega}\Big) \Big|^2\nonumber\\
=&\,\Big|\frac{x_0}{y_0}\Big|^2\left(\EE \Big|\frac{1+\frac{\zeta'}{x_0}}{1+\frac{\zeta}{y_0}}\Big|^2 -1\right)+2\mathfrak{Re} \Big[\overline{e_1}\frac{x_0}{y_0}\Big]+|e_1|^2\label{proof:delta_method:var:01},
\end{align}
where $e_1$ is defined in (\ref{proof:thm:lemma1:deltamethod:exp}).
We rewrite the expectation in (\ref{proof:delta_method:var:01}) as
\begin{align}
&\EE \Big|\frac{1+\frac{\zeta'}{x_0}}{1+\frac{\zeta^\Omega}{y_0}}\Big|^2-1=\,\Big[\EE \frac{1}{|1+\frac{\zeta^\Omega}{y_0}|^2} -1\Big]+\mathfrak{Re}\EE \frac{2\frac{\zeta'}{x_0}}{|1+\frac{\zeta^\Omega}{y_0}|^2} 
+\EE \frac{\big|\frac{\zeta'}{x_0}\big|^2}{|1+\frac{\zeta^\Omega}{y_0}|^2}\label{proof:delta_method:var:1}
\end{align}
and bound the right-hand side term by term. Rewrite the first term in (\ref{proof:delta_method:var:1}) as
\[
 \frac{1}{|1+\frac{\zeta^\Omega}{y_0}|^2}-1=\frac{-2\mathfrak{Re}\frac{\zeta^\Omega}{y_0}-\big|\frac{\zeta^\Omega}{y_0}\big|^2}{\big|1+\frac{\zeta^\Omega}{y_0}\big|^2}=-2\mathfrak{Re}\Big(\frac{\zeta^\Omega}{y_0}-\frac{\big|\frac{\zeta^\Omega}{y_0}\big|^2}{1+\overline{\frac{\zeta^\Omega}{y_0}}}-\frac{\big(\frac{\zeta^\Omega}{y_0}\big)^2}{\big|1+\frac{\zeta^\Omega}{y_0}\big|^2}\Big)-\frac{\big|\frac{\zeta^\Omega}{y_0}\big|^2}{\big|1+\frac{\zeta^\Omega}{y_0}\big|^2},
\]
where we use the equality $\frac{z}{|1+z|^2}=\frac{z}{1+\bar{z}}\big(1-\frac{z}{1+z}\big)=\frac{z}{1+\bar{z}}-\frac{z^2}{|1+z|^2}=z-\frac{|z|^2}{1+\bar{z}}-\frac{z^2}{|1+z|^2}$ when $z\neq -1$.
Thus, by (\ref{proof:delta_method:exp_of_Yball}) and (\ref{proof:delta_method:exp_of_YOmega}), the first term in (\ref{proof:delta_method:var:1}) is bounded by
\begin{align}
&\left|\EE\frac{1}{|1+\frac{\zeta}{y_0}|^2}-1\right|=\,\left|\int2\mathfrak{Re}\Big(\frac{y}{y_0}-\frac{\big|\frac{y}{y_0}\big|^2}{1+\overline{\frac{y}{y_0}}}-\frac{\big(\frac{y}{y_0}\big)^2}{\big|1+\frac{y}{y_0}\big|^2}\Big)+\frac{\big|\frac{y}{y_0}\big|^2}{\big|1+\frac{y}{y_0}\big|^2}\ud F_{\zeta^\Omega}\right|\nonumber\\
\leq&\, \frac{2}{|y_0|}\int |y|\ud F_{\zeta^\Omega}+\frac{2}{|y_0|}\int\frac{|y|^2}{|y_0+y|}\ud F_{\zeta^\Omega} +2\int \frac{|y|^2}{|y_0+y|^2}\ud F_{\zeta^\Omega}\leq\frac{85\var(\zeta)}{|y_0|^2}.\nonumber
\end{align}
The third term in (\ref{proof:delta_method:var:1}) is simply bounded by
\begin{align}
\frac{|y_0|^2}{|x_0|^2}\int \frac{|x|^2}{|y_0+y|^2}\ud F_{\zeta',\zeta^\Omega}(x,y)
\leq\,\frac{16\var (\zeta')}{|x_0|^2}\label{proof:delta_method:var:qq}.
\end{align}
Next, since
\[
 \frac{\zeta'}{|1+\frac{\zeta^\Omega}{y_0}|^2}=\frac{\zeta'}{1+\overline{\frac{\zeta^\Omega}{y_0}}}\Big(1-\frac{\frac{\zeta^\Omega}{y_0}}{1+\frac{\zeta^\Omega}{y_0}}\Big)=\zeta'-\frac{\zeta'\overline{\frac{\zeta^\Omega}{y_0}}}{1+\overline{\frac{\zeta^\Omega}{y_0}}}-\frac{\zeta'\frac{\zeta^\Omega}{y_0}}{\big|1+\frac{\zeta^\Omega}{y_0}\big|^2},
\]
by the Holder's inequality and (\ref{proof:delta_method:var:qq}), the second term in (\ref{proof:delta_method:var:1}) is bounded by
\begin{align}
&\left|\frac{2}{x_0}\EE\frac{\zeta'}{\big|1+\frac{\zeta^\Omega}{y_0}\big|^2}\right|=\left| \frac{2}{x_0}\int \Big(x-\frac{x\overline{y}}{y_0+\overline{y}}-\frac{y_0xy}{\big|y_0+y\big|^2}\Big) \ud F_{\zeta',\zeta^\Omega}(x,y)\right|\leq80\frac{\sqrt{\var(\zeta)\var (\zeta')}}{|x_0||y_0|}\nonumber,
\end{align}
where we use the fact that $\EE \zeta'=0$. 
As a result, (\ref{proof:delta_method:var:1}) is bounded by
\begin{align}
\left|\EE \Big|\frac{1+\frac{\zeta'}{x_0}}{1+\frac{\zeta^\Omega}{y_0}}\Big|^2-1\right|\leq \frac{85\var(\zeta)}{|y_0|^2}+80\frac{\sqrt{\var(\zeta)\var (\zeta')}}{|x_0||y_0|}+\frac{16\var (\zeta')}{|x_0|^2}.\nonumber
\end{align}
Hence with the bound of $e_1$ the variance (\ref{proof:thm:lemma1:deltamethod:var}) is bounded by
\begin{align}
\Big|\var \left( \frac{x_0+\zeta'}{y_0+\zeta^\Omega} \right)\Big|\leq&\,127\Big|\frac{x_0}{y_0}\Big|^2\frac{\var(\zeta)}{|y_0|^2}+88\Big|\frac{x_0}{y_0}\Big|\frac{\sqrt{\var(\zeta)\var (\zeta')}}{|y_0|^2}+16\frac{\var (\zeta')}{|y_0|^2}  +|e_1|^2\nonumber.
\end{align}
Since $|y_0|>\var(\zeta)^{1/3}$ and $\var(\zeta)^{1/3}\leq 9/32$ by assumption, we can roughly bound $|e_1|^2$ by
$$
\frac{1}{|y_0|^2}\Big(125\big|\frac{x_0}{y_0}\big|^2\var(\zeta)+48\big|\frac{x_0}{y_0}\big|\sqrt{\var(\zeta)\var(\zeta')}+5\var(\zeta')\Big)
$$
and hence conclude
\begin{align}
\Big|\var \left( \frac{x_0+\zeta'}{y_0+\zeta^\Omega} \right)\Big|\leq&\,252\Big|\frac{x_0}{y_0}\Big|^2\frac{\var(\zeta)}{|y_0|^2}+136\Big|\frac{x_0}{y_0}\Big|\frac{\sqrt{\var(\zeta)\var (\zeta')}}{|y_0|^2}+21\frac{\var (\zeta')}{|y_0|^2}.\nonumber
\end{align}
\end{proof}


\begin{proof}[Theorem \ref{section:theorem:stability}] 
We denote the GRP $\Phi_\sigma:=\sigma\Phi$ to simplify the notation. Fix a mother wavelet $\psi\in\mathcal{S}$. Since $\sigma\psi^{(l)}_{a,b}\in\mathcal{S}$ for $l=0,1$, the CWT of $Y$, $W_{Y}(a,b)$, and $-i\partial_{b}W_{Y}(a,b)$ are understood as random variables $W_{f}(a,b)+\Phi_\sigma(\psi_{a,b})$ and $-i\partial_{b}W_{f}(a,b)+\Phi_\sigma(-i\psi_{a,b}')$ respectively. 
Note that both $\Phi_\sigma(\psi_{a,b})$ and $\Phi_\sigma(\psi'_{a,b})$ are in general complex-valued. 
By assumption we know that for all $a>0$ and $b\in\RR$, $\EE(\Phi_\sigma(\psi_{a,b}))=\EE(\Phi_\sigma(\psi'_{a,b}))=0$, $\var(\Phi_\sigma(\psi_{a,b}))<\infty$ and $\var(\Phi_\sigma(\psi'_{a,b}))<\infty$. 
Notice that the quantities $\var(\Phi_\sigma(\psi_{a,b}))$, $\var(\Phi_\sigma(\psi'_{a,b}))$ are independent of the time $b$ when $\sigma$ is constant since $\Phi$ is stationary. When $\sigma$ is not constant, $\var(\Phi_\sigma(\psi_{a,b}))$ and $\var(\Phi_\sigma(\psi'_{a,b}))$ depend on $b$ and we have to handle this dependence. 
\newline\newline\underline{\textbf{Step 0: Handling the heteroscedastic term $\sigma(t)$.}}\newline

By definition 
\begin{align}
\var[\Phi_\sigma(\psi_{a,b})]\,=\,\int |\widehat{\sigma\psi_{a,b}}(\xi)|^2\ud\eta(\xi),\nonumber
\end{align} 
where $\ud \eta$ is the power spectrum of $\Phi$.
Fix $b\in\RR$. By the assumption of $\sigma$ and the integration by part, we have that 
\begin{align}
&E(\xi):=\widehat{\sigma\psi_{a,b}}(\xi)-\sigma(b)\widehat{\psi_{a,b}}(\xi)\nonumber\\
=&\,\int (\sigma(x)-\sigma(b))\psi_{a,b}(x)e^{-i2\pi\xi x}\ud x\leq \,\frac{\epsilon_\sigma C_l}{(1+|\xi|)^l},\nonumber
\end{align} 
where $C_l$ defined below is a universal constant depending on the moments of $\psi,\psi',\ldots,\psi^{(l)}$, $c_1$, $c_2$ and $d$. The last inequality comes from bounding the following term coming from the integration by part:
\begin{align}
(\sigma\psi_{a,b})^{(l)}(x)-\sigma(b)\psi_{a,b}^{(l)}(x)=&\,(\sigma(x)-\sigma(b))\psi_{a,b}^{(l)}(x)+\sum_{k=1}^l\left(\begin{array}{c}l\\ l-k\end{array}\right)\sigma^{(k)}(x)\psi_{a,b}^{(l-k)}(x)\nonumber,
\end{align}
which leads to
\begin{align}
&\,\int |(\sigma\psi_{a,b})^{(l)}(x)-\sigma(b)\psi_{a,b}^{(l)}(x)|\ud x\nonumber\\
\leq&\,\epsilon_\sigma\int \Big\{|x-b|\big|\psi_{a,b}^{(l)}(x)\big|+\sum_{k=1}^l\left(\begin{array}{c}l\\ l-k\end{array}\right)\big|\psi_{a,b}^{(l-k)}(x)\big|\Big\}\ud x\nonumber.
\end{align}
Thus $C_l$ is
\begin{align}
&\,\frac{1}{(2\pi)^l}\int \Big\{|x-b|\big|\psi_{a,b}^{(l)}(x)\big|+\sum_{k=1}^l\left(\begin{array}{c}l\\ l-k\end{array}\right)\big|\psi_{a,b}^{(l-k)}(x)\big|\Big\}\ud x\nonumber\\
\leq&\,\frac{1}{(2\pi)^l}\Big\{a^{-l+3/2}I_1^{(l)}+\sum_{k=1}^l\left(\begin{array}{c}l\\ l-k\end{array}\right) a^{k-l+1/2}I_0^{(l-k)} \Big\}=:C_l\label{proof:thm:stability:sigma_varying:Rbbound}.
\end{align} 
Note that we can bound $a$ by $(1+\Delta)/c_1<2/c_1$ when $l<3/2$ or $c_2/(1-\Delta)\leq (1+d)c_2$ when $l>3/2$, so $C_l$ is an universal constant depending on $l,c_1,c_2,d$ and the zeros and first moments of $\psi^{(k)}$, $k=1,\ldots,l$.
As a result, $E(\xi)\in C^\infty$ and $E(\xi)$ decays polynomially as fast as $|\xi|^{-l}$.

Thus, by a direct expansion we have
\begin{align}
&\var[\Phi_\sigma(\psi_{a,b})]-\sigma(b)^2\var[\Phi(\psi_{a,b})]\nonumber\\
=&\,\int \big[\sigma(b)\widehat{\psi_{a,b}}(\xi)+E(\xi)\big]\overline{\big[\sigma(b)\widehat{\psi_{a,b}}(\xi)+E(\xi)\big]}-\sigma(b)^2|\widehat{\psi_{a,b}}(\xi)|^2\ud \eta(\xi)\nonumber\\
=&\,2\sigma(b)\mathfrak{Re} \int \overline{E(\xi)}\widehat{\psi_{a,b}}(\xi)\ud \eta(\xi)+\int |E(\xi)|^2 \ud\eta(\xi).\nonumber
\end{align}
By Holder's inequality we thus obtain
\begin{align}
&\Big|\var[\Phi_\sigma(\psi_{a,b})]-\sigma(b)^2\var[\Phi(\psi_{a,b})]\Big|\nonumber\\
\leq&\,2\sigma(b) \epsilon_\sigma C_l\sqrt{C_\eta\int |\widehat{\psi_{a,b}}(\xi)|^2\ud \eta(\xi)}+\epsilon_\sigma^2C_l^2C_\eta\nonumber\\
\leq&\,2\sigma(b) \epsilon_\sigma C_l\sqrt{C_\eta}\sqrt{\var(\Phi(\psi_{a,b}))}+\epsilon_\sigma^2C_l^2C_\eta\nonumber,
\end{align} 
which leads to 
\begin{align}
\var[\Phi_\sigma(\psi_{a,b})]\leq \big(\sqrt{\var[\Phi(\psi_{a,b})]}\sigma(b)+C_l\sqrt{C_\eta}\epsilon_\sigma\big)^2\label{proof:thm:stability:sigma_varying:Rbbound2}.
\end{align}
Similarly we have
\begin{align}
\var[\Phi_\sigma(\psi'_{a,b})]\leq \big(\sqrt{\var[\Phi(\psi'_{a,b})]}\sigma(b)+C'_l\sqrt{C_\eta}\epsilon_\sigma\big)^2\label{proof:stability:Var_Phi_sigma2},
\end{align}
where $C'_l$ is an universal constant depending on $l,c_1,c_2,d$ and the zeros and first moments of $\psi^{(k)}$, $k=2,\ldots,l+1$.
From (\ref{proof:thm:stability:sigma_varying:Rbbound2}) and (\ref{proof:stability:Var_Phi_sigma2}) we see how $\var[\Phi_\sigma(\psi_{a,b})]$ and $\var[\Phi_\sigma(\psi'_{a,b})]$ depend on $\sigma(b)$ and $\epsilon_\sigma$. 

In this proof, we focus on tracing the interaction between the noise level $\sigma(b)$, the non-stationarity level $\epsilon_\sigma$ and the model bias $\epsilon$. We thus simply bound the influence of $a$ inside $\var[\Phi(\psi_{a,b})]$ and $\var[\Phi(\psi'_{a,b})]$ by the following. Define $\theta_0:=\max_{\big[\frac{1-\Delta}{c_2},\frac{1+\Delta}{c_1}\big]}\frac{\|\hat{\psi}(a\xi)\|_{L^2(\RR,\eta)}^2}{\|\hat{\psi}(\xi)\|_{L^2(\RR,\eta)}^2}$. 
By using $|\widehat{\psi_{a,b}}(\xi)|=|\sqrt{a}\widehat{\psi}(a\xi)|$, we bound 
\[
\var[\Phi(\psi_{a,b})]=a\|\hat{\psi}(a\xi)\|^2_{L^2(\RR,\eta)}\leq a\theta_0\leq \frac{1+2d}{c_1}\theta_0,
\]
since $\Delta\leq d/(1+d)$.
By the assumption $\textup{supp}\widehat{\psi}\subset[1-\Delta,1+\Delta]$ we have
\begin{align}
&\,\var(\Phi(\psi'_{a,b}))=\|\xi\widehat{\psi_{a,b}}(\xi)\|_{L^2(\RR,\eta)}=\int  a|\xi^2\widehat{\psi}(a\xi)|^2\ud\eta(\xi)\nonumber\\
\leq&\, (1+d)^2c_2^2\int a|\widehat{\psi}(a\xi)|^2\ud\eta(\xi)= (1+d)^2c_2^2\var(\Phi(\psi_{a,b}))\leq \frac{(1+d)^2c_2^2(1+2d)}{c_1}\theta_0\nonumber,
\end{align}
where the first inequality holds since we focus on $a\in \big[ \frac{1-\Delta}{c_2}, \frac{1+\Delta}{c_1}\big]$. With the simplified bounds for $\var[\Phi(\psi_{a,b})]$ and $\var[\Phi(\psi'_{a,b})]$, we further simplify the bounds of $\var[\Phi_\sigma(\psi_{a,b})]$ and $\var[\Phi_\sigma(\psi'_{a,b})]$ by
\begin{align}
\max\big\{\var(\Phi_\sigma(\psi_{a,b})),\,\var(\Phi_\sigma(\psi'_{a,b}))\big\}\leq (E_1\sigma(b)+E_2\epsilon_\sigma)^2,\label{proof:stability:Var_Phi_sigma3}
\end{align}
where 
\begin{align}
&E_1:=\max\big\{\sqrt{(1+2d)c_1^{-1}\theta_0},\, \sqrt{(1+d)^2c_2^2c_1^{-1}(1+2d)\theta_0}\big\}\nonumber\\ 
&E_2:=\max\big\{ C_l\sqrt{C_\eta},\,C'_l\sqrt{C_\eta}\big\}.\label{proof:definition_error_constant_E1E2}
\end{align} 
Here $E_1$ and $E_2$ are constants depending on the power spectrum of $\Phi$, $c_1,c_2,d$ and the zeros and first moments of $\psi^{(k)}$, $k=1,\ldots,l+1$. 

We comment that in the special case where $\Phi$ is the Gaussian white noise and $\sigma$ is constant, we can highly simply the representation of $\var(\Phi_\sigma(\psi_{a,b}))$ and $\var(\Phi_\sigma(\psi'_{a,b}))$ since $\ud\eta(\xi)=\ud\xi$ and $\|\widehat{\psi}(a\xi)\|_{L^2(\RR,\eta)}^2=a^{-1}$. See \cite{brevdo_fuckar_thakur_wu:2012} for example.  
\newline\newline\underline{\textbf{Step 1: Handling the trend term $T(t)$.}}\newline

By Assumption (A1), we have
\begin{align}\label{proof:WT:step1}
W_T(a,b)=T(\psi_{a,b}) 
\end{align}
bounded by $C_T\epsilon$ when $a\in\big(0,\frac{1+\Delta}{c_1}\big]$. Since $T$ is a real-valued function and
$\widehat{\psi}$ is compactly supported on $[1-\Delta,1+\Delta]$ so that the support of $\widehat{\psi}(a\xi)$ is always away from $\big(0,\frac{1-\Delta}{1+\Delta}c_1\big)$ when $a\in\big(0,\frac{1+\Delta}{c_1}\big]$, as a distribution in general $\widehat{T}$ is ``essentially'' supported in $\big(-\frac{1-\Delta}{1+\Delta}c_1,\frac{1-\Delta}{1+\Delta}c_1\big)$ for all time $b\in \RR$. Similarly, we have by the integration by part that
\begin{align}\label{proof:WT:step2}
\partial_bW_T(a,b)=\int T(t)\partial_b\psi_{a,b}(t)\ud t=\int \partial_tT(t)\psi_{a,b}(t)\ud t, 
\end{align}
which by Assumption (A1) is again bounded by $C_T\epsilon$ when $a\in\big(0,\frac{1+\Delta}{c_1}\big]$.
Thus, the existence of the trend $T$ does not play a significant role in the following analysis since we focus ourselves in the region $a\in\big[\frac{1-\Delta}{c_2},\frac{1+\Delta}{c_1}\big]$. 
\newline\newline\underline{\textbf{Step 2: Approximating the random variable $W_Y(a,b)$.}} \newline

Fix $b\in\RR$.
For $\gamma>1$, by the Chebychev's inequality we know that 
\begin{equation}\label{proof:stability:cYa_concentration}
\textup{Pr}\left\{|W_Y(a,b)-\EE W_Y(a,b)|>\gamma(E_1\sigma(b)+E_2\epsilon_\sigma) \right\}\leq \frac{\var(\Phi_\sigma(\psi_{a,b}))}{\gamma^2(E_1\sigma(b)+E_2\epsilon_\sigma)^2}\leq \gamma^{-2}.
\end{equation}
By Lemma \ref{lemma:identifiability:Wf_expansion}, we have 
\begin{align}
\Big|W_f(a,b)&-\sum_{l=1}^KA_l(b)e^{i2\pi \phi_l(b)}\sqrt{a}\overline{\widehat{\psi}\left(a\phi'_l(b)\right)}\Big|\leq E_{W}\epsilon\nonumber 
\end{align}

With (\ref{proof:stability:cYa_concentration}) and (\ref{proof:WT:step1}) we conclude that with probability higher than $1-\gamma^{-2}$, 
\begin{align}
\Big|W_Y(a,b)&-\sum_{l=1}^KA_l(b)e^{i2\pi \phi_l(b)}\sqrt{a}\overline{\widehat{\psi}\left(a\phi'_l(b)\right)}\Big|\leq \gamma(E_1\sigma(b)+E_2\epsilon_\sigma)+E_0\epsilon, \nonumber
\end{align}
where 
\begin{align}
E_0:=\max\big\{E_W+C_T,\,E_W'+C_T,\, E_\omega +C_T(1+2c_2),\,2\sqrt{c_2}(E_W+4C_T)\big\}.\label{proof:definition_error_constant_E0}
\end{align}
Here, $E_0$ is a universal constant depending on the moments of $\psi$ and $\psi'$, $c_1$, $c_2$ and $d$.
It bounds the model bias introduced by the $\mathcal{A}_{\epsilon,d}^{c_1,c_2}$ shown in Lemma \ref{lemma:identifiability:Wf_expansion} and that introduced by the trend $T$.
\newline\newline\underline{\textbf{Step 3: Approximating the random variable $\omega_Y(a,b)$.}}\newline

In the following proof we denote 
$$
\zeta_a:=\Phi_\sigma(\psi_{a,b})\,\,\mbox{ and }\,\,\zeta_a':=\Phi_\sigma(\psi'_{a,b}).
$$ 
To evaluate $\omega_Y(a,b)$ when $a\in Z_k(b)$, we apply Lemma \ref{proof:thm:lemma1:deltamethod}, where we take $x_0=\partial_bW_{f+T}(a,b)$, $y_0=W_{f+T}(a,b)$ 
and $\zeta_a^\Omega:=\zeta_a{\boldsymbol{\chi}}_{\CC\backslash B_{|W_{f+T}(a,b)|/4}(-W_{f+T}(a,b))}$. Notice that by (\ref{proof:thm:lemma1:deltamethod:taildecay}), when $|W_{f+T}(a,b)|>\gamma\sqrt{2}(E_1\sigma(b)+E_2\epsilon_\sigma)$, we have  
\begin{align}
\int_{B_{|W_{f+T}(a,b)|/4}(-W_{f+T}(a,b))} \ud F_{\zeta_a}(y)\leq \frac{16\var({\zeta_a})}{9|W_{f+T}(a,b)|^2}\leq\frac{8}{9}\gamma^{-2}.\label{proof:thmCARMA:conditional}
\end{align}
Denote $\omega_{f+T}(a,b):=\frac{-i\partial_bW_{f+T}(a,b)}{2\pi W_{f+T}(a,b)}$.
By Lemma \ref{proof:thm:lemma1:deltamethod}  
we have  
\begin{align}
\Big|-i\EE \Big(\frac{\partial_bW_{f+T}(a,b)+{\zeta_a}'}{W_{f+T}(a,b)+{\zeta_a^\Omega}}\Big)-\omega_{f+T}(a,b)\Big|\leq&\, \frac{12|\omega_{f+T}(a,b)|\var(\zeta_a)+4\sqrt{\var(\zeta)\var (\zeta_a')}}{|W_{f+T}(a,b)|^2}\nonumber
\end{align}
and 
\begin{align}
\var  \Big(\frac{\partial_bW_{f+T}(a,b)+{\zeta_a}'}{W_{f+T}(a,b)+{\zeta_a^\Omega}}\Big) \leq&\, \frac{\big(16|\omega_{f+T}(a,b)|\sqrt{\var(\zeta_a)}+5\sqrt{\var (\zeta_a')}\big)^2}{|W_{f+T}(a,b)|^2}.\nonumber
\end{align}
Under Assumption (A1), we evaluate $\omega_{f+T}$ by  
\begin{align}
&\,|\omega_{f+T}(a,b)-\omega_f(a,b)|\nonumber\\
=&\,\left|\frac{\partial_bW_{f+T}(a,b)W_f(a,b)-\partial_bW_{f}(a,b)W_{f+T}(a,b)}{2\pi W_{f+T}(a,b)W_f(a,b)} \right|\nonumber\\
=&\,\left|\frac{\partial_bW_{T}(a,b)W_f(a,b)-\partial_bW_{f}(a,b)W_T(a,b) }{2\pi W_{f+T}(a,b)W_f(a,b)} \right|\nonumber\\
\leq&\,\left( \frac{ 1 }{2\pi |W_{f+T}(a,b)| } + \frac{ |\omega_{f}(a,b)| }{|W_{f+T}(a,b)| }  \right)C_T\epsilon.\nonumber
\end{align}
To bound $\omega_f(a,b)$, notice that by Lemma \ref{lemma:identifiability:Wf_expansion} we have
$$
\left|\omega_f(a,b)\right|\leq \phi'_k(b)+\frac{E_\omega}{|W_f(a,b)|}\epsilon,
$$
which is bounded by $2\phi'_k(b)\leq 2c_2$ when $|W_{f}(a,b)|>\epsilon^{1/3}$ and $\epsilon$ is small enough. Hence we obtain the bound
\begin{align}
 |\omega_{f+T}(a,b)-\omega_f(a,b)|\leq \frac{C_T(1+2c_2) }{ |W_{f+T}(a,b)|}\epsilon \label{proof:continuous:omega_fT_omega_f}
\end{align}
and hence 
\begin{align}
 |\omega_{f+T}(a,b)|\leq\, \frac{C_T(1+2c_2) }{ |W_{f+T}(a,b)|}\epsilon+2c_2\leq 3c_2, \label{proof:continuous:omega_fT_omega_f2}
\end{align}
where the second inequality holds when $|W_{f+T}(a,b)|>\gamma\sqrt{2}(E_1\sigma(b)+E_2\epsilon_\sigma)+ \epsilon^{1/3}$  and $\epsilon$ is small enough.

Combine the bound of $\omega_{f+T}(a,b)$ (\ref{proof:continuous:omega_fT_omega_f2}) and (\ref{proof:stability:Var_Phi_sigma3}) and we get
\begin{align}
\Big|-i\EE \Big(\frac{\partial_bW_{f+T}(a,b)+{\zeta_a}'}{W_{f+T}(a,b)+{\zeta_a^\Omega}}\Big)-\omega_{f+T}(a,b)\Big|\leq\,\frac{(36c_2+4)(E_1\sigma(b)+E_2\epsilon_\sigma)^2}{|W_{f+T}(a,b)|^2}\label{proof:thm:omega:expectation}
\end{align}
and 
\begin{align}
\var  \Big(\frac{\partial_bW_{f+T}(a,b)+{\zeta_a}'}{W_{f+T}(a,b)+{\zeta_a^\Omega}}\Big) \leq\, \frac{(48c_2+5)^2(E_1\sigma(b)+E_2\epsilon_\sigma)^2}{|W_{f+T}(a,b)|^2}.\label{proof:thm:omega:variance}
\end{align}

Thus, by the Chebyshev's inequality and (\ref{proof:thm:omega:variance}), when $|W_f(a,b)|>\gamma\sqrt{2}(E_1\sigma(b)+E_2\epsilon_\sigma)+(C_T+1)\epsilon^{1/3}$ we have
\begin{align}
 \textup{Pr}\Big\{\Big| \frac{\partial_bW_{f+T}(a,b)+{\zeta_a}'}{W_{f+T}(a,b)+{\zeta_a^\Omega}}-&\EE \Big(\frac{\partial_bW_{f+T}(a,b)+{\zeta_a}'}{W_{f+T}(a,b)+{\zeta_a^\Omega}}\Big)\Big|>\nonumber\\&\frac{\gamma(144c_2+15)(E_1\sigma(b)+E_2\epsilon_\sigma)}{|W_{f+T}(a,b)|}\Big\}
 \leq\,\frac{1}{9}\gamma^{-2}.\label{proof:thm:omega:expectation:probability1}
\end{align}
As a result, with (\ref{proof:thmCARMA:conditional}), (\ref{proof:thm:omega:expectation}) and (\ref{proof:thm:omega:expectation:probability1}), we know that when $|W_f(a,b)|>\gamma\sqrt{2}(E_1\sigma(b)+E_2\epsilon_\sigma) + \epsilon^{1/3}$,
$$
|\omega_Y(a,b)-\omega_{f+T}(a,b)|\leq \frac{\gamma(144c_2+15)(E_1\sigma(b)+E_2\epsilon_\sigma)}{|W_{f+T}(a,b)|}+\frac{(36c_2+4)(E_1\sigma(b)+E_2\epsilon_\sigma)^2}{|W_{f+T}(a,b)|^2}
$$ 
with probability higher than 
\begin{align}
\big(1-\frac{1}{9}\gamma^{-2}\big)\big(1-\frac{8}{9}\gamma^{-2}\big)\geq \,1-\gamma^{-2}.\label{proof:stability:omega:probability}
\end{align}

To simplify the final representation, we define 
\begin{align}
&E_3:=\max\{(144c_2+15),\,\sqrt{36c_2+4}\}2E_1\nonumber\\
&E_4:=\max\{(144c_2+15),\,\sqrt{36c_2+4}\}2E_2\label{proof:definition_error_constant_E3E4}.
\end{align}
Note that $|W_f(a,b)|>\gamma(E_3\sigma(b)+E_4\epsilon_\sigma) +(C_T+1)\epsilon^{1/3}$ implies $|W_{f+T}(a,b)|>\gamma(E_1\sigma(b)+E_2\epsilon_\sigma) + \epsilon^{1/3}$ since $|W_{f+T}(a,b)|>|W_f(a,b)|-|W_T(a,b)|$. In conclusion, when $(a,b)\in Z_k(b)$ and $|W_f(a,b)|>\gamma(E_3\sigma(b)+E_4\epsilon_\sigma) +(C_T+1)\epsilon^{1/3}$, by Lemma \ref{lemma:identifiability:Wf_expansion} and (\ref{proof:continuous:omega_fT_omega_f}), with probability greater than $1-\gamma^{-2}$, we have
\begin{align}
&|\omega_Y(a,b)-\phi'_k(b)|\nonumber\\
\leq& |\omega_Y(a,b)-\omega_{f+T}(a,b)|+|\omega_{f+T}(a,b)-\omega_{f}(a,b)|+|\omega_f(a,b)-\phi'_k(b)|\nonumber\\
\leq&\, \frac{(E_3\sigma(b)+E_4\epsilon_\sigma)^2}{4|W_f(a,b)|^2}+\frac{\gamma(E_3\sigma(b)+E_4\epsilon_\sigma)}{2|W_f(a,b)|}+ \frac{C_T(1+2c_2) \epsilon}{ |W_{f}(a,b)|}+\frac{E_\omega\epsilon}{|W_f(a,b)|}\nonumber\\
\leq&\, \frac{\gamma(E_3\sigma(b)+E_4\epsilon_\sigma)+[C_T(1+2c_2)+E_\omega] \epsilon}{|W_f(a,b)|}\leq \frac{\gamma(E_3\sigma(b)+E_4\epsilon_\sigma)+E_0\epsilon}{|W_f(a,b)|}\nonumber,
\end{align}
where the second inequality comes from $|W_{f+T}(a,b)|\geq |W_f(a,b)|$, the third inequality holds since $\frac{(E_3\sigma(b)+E_4\epsilon_\sigma)}{2|W_f(a,b)|}<1$ and $\gamma>1$ leads to $\frac{(E_3\sigma(b)+E_4\epsilon_\sigma)^2}{4|W_f(a,b)|^2}<\frac{\gamma(E_3\sigma(b)+E_4\epsilon_\sigma)}{2|W_f(a,b)|}$, and the last inequality comes from the definition (\ref{proof:definition_error_constant_E0}). We thus conclude the second part of the Theorem.
\newline\newline\underline{\textbf{Step 4: Reconstruction of each seasonal component.}} \newline

Fix $b\in\RR$. Denote $Q:=E_1\sigma(b)+E_2\epsilon_\sigma$. Recall the reconstruction formula:
\begin{align}
\widetilde{f}^{Q,\CC}_k(b)=\int_{Z_k(b)} W_Y(a,b){\boldsymbol{\chi}}_{|W_Y(a,b)| >Q} a^{-3/2} \ud a,\label{proof:stability:reconstruction_rv}
\end{align}
where ${\boldsymbol{\chi}}$ is the indicator function. Note that $\widetilde{f}^{Q,\CC}_k(b)$ is a complex-valued random variable. To simplify the notation, we denote $\mathbf{1}_{Y,y,\alpha}(a):={\boldsymbol{\chi}}_{|W_Y(a,b)+y| >\alpha Q}$, where $y\in\RR$ and $\alpha>0$. We use $\mathbf{1}_{Y,y,\alpha}$ when there is no confusion. Recall that we bound $\var(\zeta_a)\leq Q^2$ for all $a$ in (\ref{proof:stability:Var_Phi_sigma3}).

We start from preparing a bound. When $|W_{f+T}(a,b)|\geq\frac{3}{2}Q$ and $r\leq Q$, since the variance of $\Psi_\sigma(\psi_{a,b})$ exists and $B_{r}(-W_{f+T}(a,b))\in\CC\backslash B_{\frac{1}{2}Q}(0)$, we have 
\begin{align}
&\int_{B_{r}(-W_{f+T}(a,b))}\ud F_{\zeta_a}\leq \int_{\CC\backslash B_{\frac{1}{2}Q}(0)}\ud F_{\zeta_a}
\leq \frac{\var({\zeta_a})}{(|W_{f+T}(a,b)|-Q)^2}\leq \frac{9\var({\zeta_a})}{|W_{f+T}(a,b)|^2},\label{proof:thm2:taildecay}
\end{align}
where we apply the Chebychev inequality and the fact that $|W_{f+T}(a,b)|-Q\geq \frac{|W_{f+T}(a,b)|}{3}$. 

Note that for a fixed $a>0$ and $b\in\RR$ we have
\begin{align}
&\EE W_Y(a,b)\mathbf{1}_{Y,0,1}-W_{f+T}(a,b)\mathbf{1}_{{f+T},0,1}\nonumber\\
=&\,W_{f+T}(a,b)\int \big(\mathbf{1}_{{f+T},y,1}-\mathbf{1}_{{f+T},0,1}\big)\ud F_{\zeta_a}(y)+\,\int y\mathbf{1}_{{f+T},y,1}\ud F_{\zeta_a}(y).\label{proof:stability:reconstruction:exp1}
\end{align}
To bound the first term of (\ref{proof:stability:reconstruction:exp1}), we consider two different cases for $W_{f+T}(a,b)$. 
When $|W_{f+T}(a,b)|>\frac{3}{2}Q$, $\mathbf{1}_{{f+T},y,1}-\mathbf{1}_{{f+T},0,1}\neq 0$ only when $y\in B_{Q}(-W_{f+T}(a,b))$, that is, 
\[
\int \big(\mathbf{1}_{{f+T},y,1}-\mathbf{1}_{{f+T},0,1}\big)\ud F_{\zeta_a}(y)\leq\int_{B_{Q}(-W_{f+T}(a,b))}\ud F_{\zeta_a}(y),
\]
which, due to (\ref{proof:thm2:taildecay}), is bounded by
$\frac{9\var({\zeta_a})}{|W_{f+T}(a,b)|^2}$. Thus, the first term of (\ref{proof:stability:reconstruction:exp1}) is bounded by
$\frac{9\var({\zeta_a})}{|W_{f+T}(a,b)|}$;
since $\EE{\zeta_a}=0$, we have
\begin{align}
\int y\mathbf{1}_{{f+T},y,1}\ud F_{\zeta_a}(y)=&\,\int_{\CC\backslash B_{Q}(-W_{f+T}(a,b))} y\ud F_{\zeta_a}(y)\nonumber\\
=&\,\int_{B_{Q}(-W_{f+T}(a,b))} y\ud F_{\zeta_a}(y),\nonumber
\end{align}
where the right hand side can be bounded by $(|W_{f+T}(a,b)|+Q)\frac{9\var({\zeta_a})}{|W_{f+T}(a,b)|^2}\leq \frac{18\var({\zeta_a})}{|W_{f+T}(a,b)|}$. 
On the other hand, when $|W_{f+T}(a,b)|\leq\frac{3}{2}Q$, we simply bound the first term in (\ref{proof:stability:reconstruction:exp1}) by $|W_{f+T}(a,b)|$ and the second term is bounded by
$$
\int_{B_{Q}(-W_{f+T}(a,b))}|y|\ud F_{\zeta_a}(y)\leq \int_{B_{\frac{5}{2}Q}(0)}|y|\ud F_{\zeta_a}(y)\leq \frac{5}{2}Q,
$$
where the last inequality comes from a simple bound. 
As a result, the first term in (\ref{proof:stability:reconstruction:exp1}) is bounded by
\begin{align}
&\Big|W_{f+T}(a,b)\int \big(\mathbf{1}_{{f+T},y,1}-\mathbf{1}_{{f+T},0,1}\big)\ud F_{\zeta_a}(y)\Big|\label{proof:stability:reconstruction_chi_diff}\\
\leq&\, \frac{9\var({\zeta_a})}{|W_{f+T}(a,b)|}\mathbf{1}_{{f+T},0,3/2}+ |W_{f+T}(a,b)|(1-\mathbf{1}_{{f+T},0,3/2})\nonumber
\end{align}
and the second term is bounded by
\begin{align}
&\Big|\int y\mathbf{1}_{{f+T},y,1}\ud F_{\zeta_a}(y)\Big|\leq\, \frac{18\var({\zeta_a})}{|W_{f+T}(a,b)|}\mathbf{1}_{{f+T},0,3/2}+\frac{5}{2}Q(1-\mathbf{1}_{{f+T},0,3/2})\label{proof:stability:reconstruction_ychi}.
\end{align}
By (\ref{proof:stability:reconstruction_chi_diff}) and (\ref{proof:stability:reconstruction_ychi}) we have the following bound for (\ref{proof:stability:reconstruction:exp1}):
\begin{align}
&\big|\EE W_Y(a,b)\mathbf{1}_{Y,0,1}-W_{f+T}(a,b)\mathbf{1}_{{f+T},0,1}\big|\nonumber\\
\leq& \,\frac{27\var({\zeta_a})}{|W_{f+T}(a,b)|}\mathbf{1}_{{f+T},0,3/2}+4Q(1-\mathbf{1}_{{f+T},0,3/2})\leq 18Q\nonumber.
 \end{align}
Here, in order to obtain a cleaner expression of the proof, we simply bound $\frac{\var({\zeta_a})}{|W_{f+T}(a,b)|}$ by $2Q/3$ by taking into account the facts that $\var(\zeta_a)\leq Q^2$ and $|W_{f+T}(a,b)|>3Q/2$.
By exchanging the expectation and integration in (\ref{proof:stability:reconstruction_rv}) we have
\begin{align}
&\Big|\EE \widetilde{f}_k(b)- \int_{Z_k(b)} W_{f+T}(a,b)\mathbf{1}_{{f+T},0,1}a^{-3/2} \ud a\Big|\nonumber\\
=&\, \Big|\int_{Z_k(b)} \big[\EE W_Y(a,b)\mathbf{1}_{Y,0,1}-W_{f+T}(a,b)\mathbf{1}_{{f+T},0,1}\big]a^{-3/2}\ud a\Big|\nonumber\\
\leq&\, 18Q\int_{Z_k(b)}a^{-3/2}\ud a=\, 36\sqrt{c_2}\Delta Q\label{proof:stability:reconstruction_expectation_error},
\end{align}
where we use $c_1\leq\phi'_k(b)\leq c_2$.

Next we evaluate the variance of $\widetilde{f}_k(b)$. 
By definition, $\var \widetilde{f}_k(b)$ becomes
\begin{align}
&\int_{Z_k(b)}\int_{Z_k(b)}  \EE\big(W_Y(a,b)\mathbf{1}_{Y,0,1}(a)-\EE W_Y(a,b)\mathbf{1}_{Y,0,1}(a)\big) \nonumber\\
&\times \overline{\big(W_Y(a',b)\mathbf{1}_{Y,0,1}(a')-\EE W_Y(a',b)\mathbf{1}_{Y,0,1}(a')\big)} (aa')^{-3/2} \ud a\ud a'\label{proof:stability:recon:variance}.
\end{align}
By a direct expansion, the expectation inside the double integral of (\ref{proof:stability:recon:variance}) becomes  
\begin{align}
&\int W_f(a,b)\big(\mathbf{1}_{{f+T},y,1}(a)-\mathbf{1}_{{f+T},0,1}(a)\big)\nonumber\\
&\quad\times\overline{W_{f+T}(a',b)}\big(\mathbf{1}_{{f+T},y',1}(a')-\mathbf{1}_{{f+T},0,1}(a')\big)\ud F_{Y_a,Y_{a'}}(y,y')\label{proof:stability:reconstruction_variance1}\\
+&\,\int W_{f+T}(a,b)\big(\mathbf{1}_{{f+T},y,1}(a)-\mathbf{1}_{{f+T},0,1}(a)\big) \overline{y'}\mathbf{1}_{{f+T},y',1}(a')\ud F_{Y_a,Y_{a'}}(y,y')\label{proof:stability:reconstruction_variance2}\\
+&\,\int \overline{W_{f+T}(a',b)}\big(\mathbf{1}_{{f+T},y',1}(a')-\mathbf{1}_{{f+T},0,1}(a')\big) y\mathbf{1}_{{f+T},y,1}(a)\ud F_{Y_a,Y_{a'}}(y,y')\label{proof:stability:reconstruction_variance3}\\
+&\,\int y\overline{y'}\mathbf{1}_{{f+T},y,1}(a)\mathbf{1}_{{f+T},y',1}(a')\ud F_{Y_a,Y_{a'}}(y,y')\label{proof:stability:reconstruction_variance4}\\
-&\,\Big|\EE \widetilde{f}_k(b)- \int_{Z_k(b)} W_{f+T}(a,b)a^{-3/2}\mathbf{1}_{{f+T},0,1}(a) \ud a\Big|^2.\nonumber
\end{align}
We simply bound (\ref{proof:stability:reconstruction_variance1}), (\ref{proof:stability:reconstruction_variance2}), (\ref{proof:stability:reconstruction_variance3}) and (\ref{proof:stability:reconstruction_variance4}) by the Holder's inequality. We prepare two simple bounds. When $|W_{f+T}(a,b)|>\frac{3}{2}Q$, $\mathbf{1}_{{f+T},y,1}-\mathbf{1}_{{f+T},0,1}\neq 0$ only when $y\in B_{Q}(-W_{f+T}(a,b))$, that is, 
\[
\int \big(\mathbf{1}_{{f+T},y,1}-\mathbf{1}_{{f+T},0,1}\big)^2\ud F_{\zeta_a}(y)\leq\int_{B_{Q}(-W_{f+T}(a,b))}\ud F_{\zeta_a}(y),
\]
which again, due to (\ref{proof:thm2:taildecay}), is bounded by $\frac{9\var({\zeta_a})}{|W_{f+T}(a,b)|^2}$; when $|W_{f+T}(a,b)|\leq\frac{3}{2}Q$, we simply bound $\big(\mathbf{1}_{{f+T},y,1}-\mathbf{1}_{{f+T},0,1}\big)^2$ by $1$. Therefore we obtain
\begin{align}
&\Big(\int |W_{f+T}(a,b)|^2\big|\mathbf{1}_{{f+T},y,1}(a)-\mathbf{1}_{{f+T},0,1}(a)\big|^2\ud F_{\zeta_a}\Big)^{1/2}\nonumber\\
\leq&\,3\sqrt{\var({\zeta_a})}\mathbf{1}_{{f+T},0,3/2}(a)+|W_{f+T}(a,b)|(1-\mathbf{1}_{{f+T},0,3/2}(a))\leq 3Q,\label{proof:stability:reconstruction_chi_diff2}
\end{align}
where the first inequality holds since the support of $\mathbf{1}_{{f+T},0,3/2}(a)$ is disjoint from that of $1-\mathbf{1}_{{f+T},0,3/2}(a)$.
Next we bound
\begin{align}
&\int \mathbf{1}_{{f+T},y,1}(a) |y|^2\ud F_{\zeta_a}(y)\leq \var(\zeta_a)\label{proof:stability:reconstruction_ychi2}
\end{align}
simply by $\mathbf{1}_{{f+T},y,1}(a)\leq1$.
By the Holder's inequality and (\ref{proof:stability:reconstruction_chi_diff2}), (\ref{proof:stability:reconstruction_variance1}) is bounded by
\begin{align}
&\,\Big(3\sqrt{\var({\zeta_a})}\mathbf{1}_{{f+T},0,3/2}(a)+|W_{f+T}(a,b)|(1-\mathbf{1}_{{f+T},0,3/2}(a))\Big)\nonumber\\
&\quad\times\Big(3\sqrt{\var({\zeta_{a'}})}\mathbf{1}_{{f+T},0,3/2}(a')+|W_{f+T}(a',b)|(1-\mathbf{1}_{{f+T},0,3/2}(a'))\Big)\leq 9Q^2\nonumber
\end{align}
Similarly, by (\ref{proof:stability:reconstruction_chi_diff2}) and (\ref{proof:stability:reconstruction_ychi2}), (\ref{proof:stability:reconstruction_variance2}) and (\ref{proof:stability:reconstruction_variance3}) are together bounded by
\begin{align}
&\,\Big(3\sqrt{\var({\zeta_a})}\mathbf{1}_{{f+T},0,3/2}(a)+|W_{f+T}(a,b)|(1-\mathbf{1}_{{f+T},0,3/2}(a))\Big)\sqrt{\var(\zeta_a)}\nonumber\\
+&\,\Big(3\sqrt{\var(\zeta_{a'})}\mathbf{1}_{{f+T},0,3/2}(a)+|W_{f+T}(a,b)|(1-\mathbf{1}_{{f+T},0,3/2}(a))\Big)\sqrt{\var(\zeta_{a'})}\leq 6Q^2\nonumber
\end{align}
We simply bound (\ref{proof:stability:reconstruction_variance4}) by $\sqrt{\var(\zeta_a)\var(\zeta_{a'})}\leq Q^2$ by the Holder's inequality and bound the last term by (\ref{proof:stability:reconstruction_expectation_error}).
Putting the above bounds together, the expectation inside the double integral of (\ref{proof:stability:recon:variance}) is bounded by 
$(16+324c_1^{-2}\Delta^2)Q^2$.
Hence, the bound of $\var \widetilde{f}_k(b)$ in (\ref{proof:stability:recon:variance}) becomes:
\begin{align}
\var \widetilde{f}_k(b)\leq(16+36^2c_2\Delta^2)Q^2\Big[ \int_{Z_k(b)}  a^{-3/2}\ud a\Big]^2\leq\,16(1+81c_2\Delta^2)c_2\Delta^2Q^2.\label{proof:stability:reconstruction_variance_error}
\end{align}

As a result, by (\ref{proof:stability:reconstruction_variance_error}) and the Chebychev's inequality, we know that 
\[
\textup{Pr}\left\{\big|\widetilde{f}_k(b)-\EE\widetilde{f}_k(b)\big|>4\gamma\sqrt{1+81c_2\Delta^2}\sqrt{c_2}\Delta Q\right\}\leq\,\gamma^{-2}.\nonumber
\]
Together with (\ref{proof:stability:reconstruction_expectation_error}) this yields that,   
with probability greater than $1-\gamma^{-2}$,
\begin{align}\label{proof:continuous:reconstruction:tildef:f+T}
\big|\widetilde{f}_k(b)-\int_{Z_k(b)} W_{f+T}(a,b)\mathbf{1}_{{f+T},0,1}a^{-3/2} \ud a\big|\leq\big[4\gamma\sqrt{1+81c_2\Delta^2}+36\big]\sqrt{c_2}\Delta Q. 
\end{align}
Here we have the following simple bound
\begin{align}
&\Big|\int_{Z_k(b)} W_{f+T}(a,b)\mathbf{1}_{f+T,0,1}a^{-3/2} \ud a-\int_{Z_k(b)} W_{f }(a,b)\mathbf{1}_{f,0,1}a^{-3/2} \ud a\Big|\nonumber\\
=&\Big|\int_{Z_k(b)} W_{f}(a,b)(\mathbf{1}_{f+T,0,1}-\mathbf{1}_{f,0,1})a^{-3/2} \ud a+\int_{Z_k(b)} W_{T}(a,b)\mathbf{1}_{f+T,0,1}a^{-3/2} \ud a\Big|\nonumber\\
\leq&\int_{Z_k(b)}|W_f(a,b)|{\boldsymbol{\chi}}_{|W_f(a,b) | \leq  Q+C_T\epsilon}a^{-3/2}\ud a + C_T\epsilon\int_{Z_k(b)} a^{-3/2} \ud a\nonumber\\
\leq&4\sqrt{c_2} \Delta(Q+2C_T\epsilon),\nonumber
\end{align}
which when combined with Lemma \ref{lemma:identifiability:Wf_expansion} leads to 
\begin{align}\label{proof:continuous:reconstruction:f+T:f}
&\Big|\int_{Z_k(b)} W_{f+T}(a,b)\mathbf{1}_{f+T,0,1}a^{-3/2} \ud a-A_k(b)e^{2\pi i \phi_k(b)}\Big|\nonumber\\
\leq&6\sqrt{c_2} Q\Delta +2\sqrt{c_2}(E_W+4C_T)\Delta\epsilon.
\end{align}
With (\ref{proof:continuous:reconstruction:tildef:f+T}) and (\ref{proof:continuous:reconstruction:f+T:f}) we conclude that with probability higher than $1-\gamma^{-2}$, 
\begin{align}
\big|\widetilde{f}_k(b)-A_k(b)e^{2\pi i \phi_k(b)}\big|\leq&\, \big[4\gamma\sqrt{1+81c_2\Delta^2}+42\big]\sqrt{c_2}\Delta Q+2\sqrt{c_2}(E_W+4C_T)\Delta\epsilon.\nonumber\\
\leq&\, \gamma\big[4\sqrt{1+81c_2\Delta^2}+42\big]\sqrt{c_2}\Delta Q+2\sqrt{c_2}(E_W+4C_T)\Delta\epsilon \nonumber 
\end{align}
since $\gamma>1$. 
To simplify the expression, we define 
\begin{align}
E_5:=(4\sqrt{1+81c_2d^2}+42)\sqrt{c_2}E_1,\quad E_6:=(4\sqrt{1+81c_2d^2}+42)\sqrt{c_2}E_2,\label{proof:definition_error_constant_E5E6}
\end{align}
which together with (\ref{proof:definition_error_constant_E0}) lead to 
\begin{align}
\big|\widetilde{f}_k(b)-A(b)e^{2\pi i \phi_k(b)}\big|\leq\,\big[\gamma(E_5\sigma(b)+E_6\epsilon_\sigma)+E_0\epsilon\big]\Delta\nonumber.
\end{align}
Note that $(4\sqrt{1+81c_2\Delta^2}+42)\sqrt{c_2}E_1\leq E_5$ and $(4\sqrt{1+81c_2\Delta^2}+42)\sqrt{c_2}E_2\leq E_6$ since $\Delta<d$. The proof of the (iii) is thus finished.
\newline\newline\underline{\textbf{Step 5: Reconstruction of the trend.}} \newline

Recall the following trend estimator, which is a generalized random process in general:
\[
\widetilde{T}=Y-\mathfrak{Re}\int_{\frac{1-\Delta}{c_2}}^{\frac{1+\Delta}{c_1}} W_Y(a,b)a^{-3/2}\ud a.
\]
To simplify the notation, denote 
\[
R(b):=\mathfrak{Re}\int_{\frac{1-\Delta}{c_2}}^{\frac{1+\Delta}{c_1}} W_Y(a,b)a^{-3/2}\ud a.
\]
Intuitively, $\mathfrak{Re}\int_{\frac{1-\Delta}{c_2}}^{\frac{1+\Delta}{c_1}} W_Y(a,b)a^{-3/2}\ud a$ approximates the seasonal components, which cancels the seasonal components in $Y$. Thus $\widetilde{T}$ approximates $T$ in the general distribution sense. To prove this, take $\varphi\in\mathcal{S}$ and define $\varphi_{h,x}(t):=\frac{1}{h}\varphi(\frac{x-t}{h})$, where $h>0$ and $x\in\RR$. Assume $\varphi_{h,x}\to \delta_x$ in the distribution sense when $h\to 0$, where $\delta_x$ is the delta measure at $x\in\RR$. We evaluate the expectation of $\widetilde{T}$ directly:
\begin{align*}
\EE\widetilde{T}(\varphi_{h,x})&=\EE Y(\varphi_{h,x})-\EE R(\varphi_{h,x})\\
&=\int (T+f)\varphi_{h,x}\ud b-\mathfrak{Re}\int \int_{\frac{1-\Delta}{c_2}}^{\frac{1+\Delta}{c_1}} \EE W_Y(a,b)a^{-3/2}\ud a \varphi_{h,x}(b)\ud b\\
&=T(\varphi_{h,x})+\int \left[f(b)-\mathfrak{Re}\int_{\frac{1-\Delta}{c_2}}^{\frac{1+\Delta}{c_1}} W_{f+T}(a,b)a^{-3/2}\ud a\right] \varphi_{h,x}(b)\ud b 
\end{align*}
where we use the facts that $T\in\mathcal{S}'$ and $\EE W_Y(a,b)=W_{T+f}(a,b)$. 
We bound $\EE\widetilde{T}(\varphi_{h,x})-T(\varphi_{h,x})$ by the following. By Lemma  \ref{lemma:identifiability:Wf_expansion}  we have 
\begin{align}
&\left|\int \Big[f(b)-\mathfrak{Re}\int_{\frac{1-\Delta}{c_2}}^{\frac{1+\Delta}{c_1}} W_{f+T}(a,b)a^{-3/2}\ud a\Big] \varphi_{h,x}(b)\ud b\right|\nonumber\\
\leq &\epsilon  \int \int_{\frac{1-\Delta}{c_2}}^{\frac{1+\Delta}{c_1}} (E_W+C_T)a^{-3/2} \ud a |\varphi_{h,x}(b)|\ud b \leq  E_T\epsilon  \|\varphi\|_{1},\label{proof:definition:ET}
\end{align}
where $E_T$ is a constant depending on $C_T, c_1,c_2,d$ and $I_{k}$, $k=1,2,3$, and the first inequality comes from (\ref{proof:definition:E_W}) and the fact that
\begin{align*}
&\mathfrak{Re}\int_{\frac{1-\Delta}{c_2}}^{\frac{1+\Delta}{c_1}} W_f(a,b)a^{-3/2}\ud a\\
=&\mathfrak{Re}\int_{\frac{1-\Delta}{c_2}}^{\frac{1+\Delta}{c_1}} \sum_{l=1}^K A_l(b)e^{i2\pi\phi_l(b)}\sqrt{a}\overline{ \widehat{\psi}(a\phi_l'(b))})a^{-3/2}\ud a+\epsilon\int_{\frac{1-\Delta}{c_2}}^{\frac{1+\Delta}{c_1}}C(a,b)a^{-3/2}\ud a\\
=&f(b)+\epsilon\int_{\frac{1-\Delta}{c_2}}^{\frac{1+\Delta}{c_1}}C(a,b)a^{-3/2} \ud a.
 \end{align*}
Hence, we have
 \[
|\EE\widetilde{T}(\varphi_{h,x})-T(\varphi_{h,x})|\leq E_T\|\varphi\|_{1}\epsilon.
\]
Note that since $\varphi_{h,x}$ is a approximate identity and $T\in C^\infty\cap \mathcal{S}'$, as $h\to0$, we have $T(\varphi_{h,x})\to T(x)$. As result, since $E_T\|\varphi\|_{1}$ is independent of $h$, when $h$ is small enough, we obtain
 \[
|\EE\widetilde{T}(\varphi_{h,x})-T(x)|\leq 2E_T \|\varphi\|_{1}\epsilon .
\] 

\end{proof}


\section{Proof of Theorem \ref{section:theorem:ARMAstability}}

The proof is essentially the same as that of Theorem \ref{section:theorem:stability} once we get the following approximations specific to the discretization.
\newline\newline\underline{\textbf{Step 1: Approximating the discretized signal}}\newline
 
By the Poisson formula, given $\boldsymbol{T}:=\{T(n\tau )\}_{n\in\ZZ}$, where $\tau >0$, we have
\begin{align}
W_{\boldsymbol{T}}(a,b)=&\,\tau \sum_{j\in\ZZ}T(j\tau )\overline{\psi_{a,b}(j\tau)}=\,\sum_{n\in\ZZ}\mathcal{F}\left\{T(t)\overline{\psi_{a,b}(t)}\right\}\Big(\frac{n}{\tau }\Big)
\end{align}
where $\mathcal{F}$ means Fourier transform. Note that $W_{T}(a,b)=\mathcal{F}\left\{T(t)\overline{\psi_{a,b}(t)}\right\}(0)=0$ when $a\in \big[ \frac{1-\Delta}{c_2}, \frac{1+\Delta}{c_1}\big]$. Thus, in order to analyze the error introduced by the finite sampling, we focus on analyzing $\sum_{n\in\ZZ,n\neq 0}W^{(n/\tau )}_{T}(a,b)$, where we denote $W^{(n/\tau )}_{T}(a,b):=\mathcal{F}\left\{T(t)\overline{\psi_{a,b}(t)}\right\}\big(\frac{n}{\tau }\big)$.
By the extra regularity assumption (A3) on $T$, we control \\$\sum_{n\in\ZZ,n\neq 0}W^{(n/\tau )}_{T}(a,b)$ by the same arguments as that in the proof of Lemma \ref{lemma:identifiability:Wf_expansion_discrete} and get
\begin{align}\label{proof:discrete:CWT_of_T}
|W_{\boldsymbol{T}}(a,b)|=\big|W_{T}(a,b)+\sum_{n\in\ZZ,n\neq 0}W^{(n/\tau )}_{T}(a,b)\big| \leq  C_{\tau,T}\epsilon.
\end{align}
when $a\in \big(0, \frac{1+\Delta}{c_1}\big]$, where $C_{\tau,T}$ depends on the first three moments of $\psi$ and $\psi'$, $C_T, c_1,c_2$ and $d$. 
In other words, when the sampling rate is high enough, the trend component is not significant and we can focus on the seasonal component analysis. In the case that the sampling rate is not that high, the aliasing effect will show up and extra error term will come into play.
\newline\newline\underline{\textbf{Step 2: Estimate $W_{\boldsymbol{Y}}(a,b)$, $\omega_{\boldsymbol{Y}}(a,b)$ and the reconstruction}}\newline

We need the following approximations. 
First, the CWT of $\boldsymbol{Y}$ becomes
\begin{align*}
W_{\boldsymbol{Y}}(a,b)&:=\,\tau \sum_{n\in\ZZ} Y_n\psi_{a,b}(n\tau)\\
&=\,W_{\boldsymbol{f}+\boldsymbol{T}}(a,b)+\tau \sum_{n\in\ZZ}\sigma(n\tau) \Phi_n\psi_a(n\tau -b).
\end{align*}
Similarly, we have
\begin{align*}
\partial_bW_{\boldsymbol{Y}}(a,b)&:=\,\tau \sum_{n\in\ZZ} Y_n\frac{1}{a^{3/2}}\psi'\big(\frac{n\tau -b}{a}\big)\\
&=\,\partial_bW_{\boldsymbol{f}+\boldsymbol{T}}(a,b)+\tau \sum_{n\in\ZZ}\sigma(n\tau) \Phi_n\psi^{(1)}_a(n\tau -b).
\end{align*}
Denote 
\[
\Psi_\sigma(a):= \tau \sum_{n\in\ZZ}\sigma(n\tau)\Phi_n\psi_a(n\tau -b),\,\,\,\Psi(a):= \tau \sum_{n\in\ZZ}\Phi_n\psi_a(n\tau -b)
\] 
and
\[
\Psi'_\sigma(a):= \tau \sum_{n\in\ZZ}\sigma(n\tau)\Phi_n\psi^{(1)}_a(n\tau -b),\,\,\,\Psi'(a):= \tau \sum_{n\in\ZZ}\Phi_n\psi^{(1)}_a(n\tau -b).
\] 
Note that $\Psi_\sigma(a)$ and $\Psi'_\sigma(a)$ are both complex-valued random variables. Clearly $\EE \Psi_\sigma(a)=0$ and $\EE \Psi'_\sigma(a)=0$. To evaluate variances of $\Psi_\sigma(a)$ and $\Psi'_\sigma(a)$, recall that the covariance function $\gamma(n):=\EE X_n\overline{X_0}$ is a positive definite sequence. Thus, since $\psi\in\mathcal{S}$ and $X$ is stationary, we are allowed to exchange $\EE$ and $\sum_{n\in\ZZ}$ and have 
\begin{align}
\var (\Psi_\sigma(a))&:=\,\tau ^2\EE\sum_{n\in\ZZ}\sigma(n\tau)X_n\psi_a(n\tau -b)\overline{\sum_{m\in\ZZ}\sigma(m\tau)X_m\psi_a(m\tau -b)}\nonumber\\
&=\,\tau ^2\sum_{n,m\in\ZZ} \EE X_n\overline{X_m}\sigma(n\tau)\psi_a(n\tau -b)\overline{\sigma(m\tau)\psi_a(m\tau -b)}\nonumber\\
&=\,\tau ^2\sum_{n,m\in\ZZ} \gamma(n-m)\sigma(n\tau)\psi_a(n\tau -b)\overline{\sigma(m\tau)\psi_a(m\tau -b)}\nonumber\\
&=\,\int_{0}^{2\pi} \Big|\tau \sum_{n\in\ZZ}\sigma(n\tau)\psi_a(n\tau -b)e^{-i n\xi}\Big|^2\ud \mu(\xi)\label{proof:discrete:poisson:in:integration},
\end{align}
where the last equality comes from the Bochner theorem. Similarly we have
\begin{align}
\var (\Psi(a))=\int_{0}^{2\pi} \Big|\tau \sum_{n\in\ZZ}\psi_a(n\tau -b)e^{-i n\xi}\Big|^2\ud \mu(\xi)\nonumber.
\end{align}
 Since $\psi\in\mathcal{S}$, by the Poisson formula, the summation in (\ref{proof:discrete:poisson:in:integration}) becomes
$$
\tau \sum_{n\in\ZZ}\sigma(n\tau)\psi_a(n\tau -b)e^{-in\xi}=\sum_{n\in\ZZ}\mathcal{F}\big\{\sigma(\cdot)\psi_a(\cdot-b)\big\}\Big(\frac{\xi}{2\pi \tau}+\frac{n}{\tau }\Big),
$$
where $\mathcal{F}$ means Fourier transform. By the same arguments as in the proof of Lemma \ref{lemma:identifiability:Wf_expansion_discrete} and by the assumption (A4) on $\sigma$, we have
$$
\Big|\sum_{n\in\ZZ}\mathcal{F}\big\{\sigma(\cdot)\psi_a(\cdot-b)\big\}\Big(\frac{\xi}{2\pi \tau}+\frac{n}{\tau }\Big)-\sigma(b)\sum_{n\in\ZZ}\mathcal{F}\big\{\psi_a(\cdot-b)\big\}\Big(\frac{\xi}{2\pi \tau}+\frac{n}{\tau }\Big)\Big|\leq \tau^2e_{\tau,0}\epsilon_\sigma,
$$
and hence
$$
 \var (\Psi_\sigma(a))\leq (\sigma(b)\sqrt{\var (\Psi(a))}+\tau^2e_{\tau,1}\epsilon_\sigma)^2,
$$
where $e_{\tau,0}$ and $e_{\tau,1}$  are constants depending on $\Phi$, $c_1,c_2,d$ and the first two moments of $\psi$ and $\psi'$.
To simplify the bound, we take the following parameters into account. Since $0<\tau <\frac{1-\Delta}{(1+\Delta)c_2}$, $a\in\big[\frac{1-\Delta}{c_2},\frac{1+\Delta}{c_1}\big]$ and $\widehat{\psi_a}(\xi)=\sqrt{a}\widehat{\psi}(a\xi)$, we have $\frac{a}{\tau }>1+\Delta$ and  
\begin{align}
\var (\Psi(a))=&\,\int_0^{2\pi} \Big|\sum_{n\in\ZZ}\widehat{\psi_a}\Big(\frac{\xi}{2\pi \tau }+\frac{n}{\tau }\Big)\Big|^2\ud \mu(\xi)\nonumber\\
=&\,\int_{0}^{2\pi} \Big|\widehat{\psi_a}\Big(\frac{\xi}{2\pi \tau }\Big)\Big|^2\ud \mu(\xi)=a\int_{0}^{2\pi} \Big|\widehat{\psi}\Big(\frac{a\xi}{2\pi \tau }\Big)\Big|^2\ud \mu(\xi)
\leq \frac{2}{c_1}\theta_\tau,\nonumber
\end{align}
where we use the fact that 
$$
\int_{0}^{2\pi} \big|\widehat{\psi}(\frac{a\xi}{2\pi \tau })\big|^2\ud \mu(\xi)\leq \max_{a\in\big[\frac{1-\Delta}{c_2},\frac{1+\Delta}{c_1}\big]}\int_{0}^{2\pi} \big|\widehat{\psi}(\frac{a\xi}{2\pi \tau })\big|^2\ud \mu(\xi)=:\theta_\tau.
$$ 
Similarly, we have
\begin{align}
&\var (\Psi'(a))=\int_{0}^{2\pi} \Big|\widehat{\psi_a^{(1)}}\Big(\frac{\xi}{2\pi \tau }\Big)\Big|^2\ud \mu(\xi)=\frac{a}{\tau ^2}\int_{\frac{2\pi \tau (1-\Delta)}{a}}^{\frac{2\pi \tau (1+\Delta)}{a}} \Big|\xi\widehat{\psi}\big(\frac{a\xi}{2\pi \tau }\big)\Big|^2\ud \mu(\xi)\nonumber\\
\leq&\,\frac{a}{\tau ^2}\Big|\frac{2\pi \tau (1+\Delta)}{a}\Big|^2\int_{0}^{2\pi} \Big|\widehat{\psi}\big(\frac{a\xi}{2\pi \tau }\big)\Big|^2\ud \mu(\xi)\leq \frac{16\pi^2}{a}\theta_\tau\leq 16\pi^2c_2(1+d)\theta_\tau.\nonumber
\end{align}
We thus bound $\var(\Psi_\sigma(a))$ and $\var(\Psi'_\sigma(a))$ simultaneously by 
\begin{align}
\max\big\{\var(\Psi_\sigma(a)),\var(\Psi'_\sigma(a))\big\}\leq (E_{\tau,1}\sigma(b)+\tau^2E_{\tau,2}\epsilon_\sigma)^2,\label{proof:notationE1E2:discrete}
\end{align} 
where 
$E_{\tau,1}$ and $E_{\tau,2}$ are constants depending on $\Phi$, $c_1,c_2,d$ and the zeros and first moments of $\psi$ and $\psi'$. 

With the above estimations and Lemma \ref{lemma:identifiability:Wf_expansion_discrete}, we can finish the proof of (i), (ii) and (iii) by following the same lines as in the proof of Theorem \ref{section:theorem:stability}. 
\newline\newline\underline{\textbf{Step 3: Reconstruction of the trend.}} \newline

Recall that we reconstruct the trend at time $n\tau$, $n\in\ZZ$, by the following estimator:
\[
\widetilde{T}_n:=\boldsymbol{Y_n}-\mathfrak{Re}\int_{\frac{1-\Delta}{c_2}}^{\frac{1+\Delta}{c_1}} W_{\boldsymbol{Y}}(a,n\tau)a^{-3/2}\ud a.
\]
To simplify the notation, denote 
\[
R_n:=\mathfrak{Re}\int_{\frac{1-\Delta}{c_2}}^{\frac{1+\Delta}{c_1}} W_{\boldsymbol{Y}}(a,n\tau)a^{-3/2}\ud a,
\]
which is a random variable.
We now show $\widetilde{T}_n$ approximates $T(n\tau)$. First, the expectation of $\widetilde{T}_n$ satisfies:
\begin{align*}
\EE\widetilde{T}_n 
 =T(n\tau)+ f(n\tau)-\mathfrak{Re}\int_{\frac{1-\Delta}{c_2}}^{\frac{1+\Delta}{c_1}} W_{\boldsymbol{f}+\boldsymbol{T}}(a,n\tau)a^{-3/2}\ud a.   
\end{align*}
 By Lemma \ref{lemma:identifiability:Wf_expansion_discrete} and (\ref{proof:discrete:CWT_of_T}) we have
\begin{align*}
&\left|\int_{\frac{1-\Delta}{c_2}}^{\frac{1+\Delta}{c_1}} W_{\boldsymbol{f}+\boldsymbol{T}}(a,n\tau)a^{-3/2}\ud a-\int_{\frac{1-\Delta}{c_2}}^{\frac{1+\Delta}{c_1}} \sum_{l=1}^K A_l(n\tau)e^{i2\pi\phi_l(n\tau)}\sqrt{a}\overline{ \widehat{\psi}(a\phi_l'(n\tau))})a^{-3/2}\ud a\right|\\
\leq\,& \epsilon\int_{\frac{1-\Delta}{c_2}}^{\frac{1+\Delta}{c_1}} (E_{\tau,W}\tau^2+E_W+C_{\tau,T})a^{-3/2}\ud a \\
\leq\,& 2(E_{\tau,W}\tau^2+E_W+C_{\tau,T})\left( \sqrt{\frac{ c_2}{1-\Delta}}-\sqrt{\frac{ c_1}{1+\Delta}}\right) \epsilon.
 \end{align*}
By a direct calculation, the following equality holds  
 $$
 f(n\tau)=\int_{\frac{1-\Delta}{c_2}}^{\frac{1+\Delta}{c_1}} \sum_{l=1}^K A_l(n\tau)e^{i2\pi\phi_l(n\tau)}\sqrt{a}\overline{ \widehat{\psi}(a\phi_l'(n\tau))})a^{-3/2}\ud a,
 $$ 
which leads to the bound  
 \begin{align*}
&\left|  f(n\tau)-\mathfrak{Re}\int_{\frac{1-\Delta}{c_2}}^{\frac{1+\Delta}{c_1}} W_{\boldsymbol{f}+\boldsymbol{T}}(a,n\tau)a^{-3/2}\ud a \right|\leq E_{T,0}\epsilon  
\end{align*}
where $E_{T,0}:=2(E_{\tau,W}\tau^2+E_W+C_{\tau,T})\left( \sqrt{\frac{ c_2}{1-\Delta}}-\sqrt{\frac{ c_1}{1+\Delta}}\right) $ is a constant depending on $C_T,c_1,c_2,d$ and $I_{k}$, $k=1,2,3$. 
As a result, the expectation of the trend estimator is deviated from the truth by
\begin{align}\label{proof:discrete:trend:expectation}
|\EE\widetilde{T}_n-T(n\tau)|\leq E_{T,0}\epsilon .
\end{align}

By the assumption that $\textup{var}(\Phi_n)=1$, the variance of $\widetilde{T}_n$ is directly evaluated by:
\begin{align*}
\textup{var}\widetilde{T}_n&=\EE|\widetilde{T}_n|^2-|\EE\widetilde{T}_n|^2=\EE|\boldsymbol{Y_n} -R_n|^2-|\EE(\boldsymbol{Y_n}-R_n)|^2\\
&=\sigma(n\tau)^2\textup{var}(\Phi_n)+\textup{var}(R(b))-2\mathfrak{Re}\EE[\sigma(n\tau)\Phi_n\overline{(R_n-\EE R_n)}]\\
&\leq(\sigma(n\tau)+\sqrt{\textup{var}(R_n)})^2,
\end{align*}
where we use the Holder's inequality. By a direct calculation with (\ref{proof:stability:Var_Phi_sigma3}), the following bound can be achieved:
\begin{align*}
\text{var}(R_n)\leq 2\left( \sqrt{\frac{ c_2}{1-\Delta}}-\sqrt{\frac{ c_1}{1+\Delta}}\right) (E_{\tau,1}\sigma(n\tau)+E_{\tau,2}\epsilon_\sigma)^2,
\end{align*} 
which comes from (\ref{proof:notationE1E2:discrete}) and the following fact:
\begin{align*}
R_n-\EE R_n&= t\int_{\frac{1-\Delta}{c_2}}^{\frac{1+\Delta}{c_1}}\big[ W_{\boldsymbol{Y}}(a,n\tau)- W_{\boldsymbol{f}+\boldsymbol{T}}(a,n\tau) \big] a^{-3/2}\ud a =  \int_{\frac{1-\Delta}{c_2}}^{\frac{1+\Delta}{c_1}} \Psi_\sigma(a)a^{-3/2}\ud a  .
\end{align*}
As a result, we obtain
\begin{align*}
\textup{var}\widetilde{T}_n &\leq(\sigma(n\tau)+\sqrt{2}\sqrt{ \sqrt{\frac{ c_2}{1-\Delta}}-\sqrt{\frac{ c_1}{1+\Delta}} }  (E_{\tau,1}\sigma(n\tau)+E_{\tau,2}\epsilon_\sigma)  )^2\\
&\leq (E_{T,1}\sigma(n\tau) +E_{T,2}\epsilon_\sigma )^2,
\end{align*}
where $E_{T,1}=1+\sqrt{2}\sqrt{ \sqrt{\frac{ c_2}{1-\Delta}}-\sqrt{\frac{ c_1}{1+\Delta}} }E_{\tau,1} $ and $E_{T,2}= \sqrt{2}\sqrt{ \sqrt{\frac{ c_2}{1-\Delta}}-\sqrt{\frac{ c_1}{1+\Delta}} }E_{\tau,2} $  are constants depending on $\Phi$, $c_1,c_2,d$ and the zeros and first moments of $\psi$ and $\psi'$.

Now we put everything together. By the Chebychev's inequality, we know that 
\[
\textup{Pr}\left\{ \big|\widetilde{T}_n-\EE\widetilde{T}_n\big|> \gamma (E_{T,1}\sigma(n\tau) +E_{T,2}\epsilon_\sigma ) \right\} \leq\,\gamma^{-2}.\nonumber
\]
Combining this with (\ref{proof:discrete:trend:expectation}), we conclude that 
with probability greater than $1-\gamma^{-2}$: 
\[
\big|\widetilde{T}_n-T(n\tau)\big|\leq \gamma (E_{T,1}\sigma(n\tau) +E_{T,2}\epsilon_\sigma )+E_{T,0}\epsilon. 
\]


\section{Discrete- and continuous-time Autoregressive Moving Average processes}

\subsection{Discrete-time Autoregressive Moving Average Processes}\label{section:ARMA}
Take a time series $X_t$. Denote $B$ as the lag operator, that is, $BX_n=X_{n-1}$. Denote $\omega_t$ as an white noise, i.e., an independent random variable with mean $0$ and variance $\sigma^2$. Note that $\omega_t$ might have fat tail or the variance of $\omega_t$ might not be finite. An Autoregressive and Moving Average Process of order $(p,q)$, denoted as ARMA(p,q), where $p,q\in\NN\cup\{0\}$, is defined as a complex-valued strictly stationary solution $\Phi_n$, $n\in\ZZ$, of the difference equation:
\begin{equation}\label{section:model:ARMA}
\phi(B)\Phi_t=\theta(B)\omega_t,
\end{equation}
where $\phi(\cdot)$ and $\theta(\cdot)$ are polynomials defined as
$$
\phi(z)=1-\phi_1z-\ldots-\phi_pz^p,\quad  \theta(z)=1-\theta_1z-\ldots-\theta_qz^q
$$  
with $\phi_i\in\CC$, $\theta_j\in\CC$ for $i=1,\ldots,p$ and $j=1,\ldots,q$. In \cite{Brockwell_Lindner:2010} the following result was provided: suppose $\omega_n$ is a nondeterministic independent white noise sequence. Then (\ref{section:model:ARMA}) admits a strictly stationary solution $\Phi_n$ if and only if (i) all singularities of $\theta(z)/\phi(z)$ on the unit circle are removable and $\EE\log^+|\omega_1|<\infty$, or (ii) all singularities of $\theta(z)/\phi(z)$ are removable. When $\sigma^2<\infty$, the power spectral density $\ud\mu(\xi)$ of the strictly stationary solution $\Phi_n$ satisfies 
$$
\ud\mu(\xi)=\frac{\sigma^2|\theta(e^{-i\xi})|^2}{2\pi|\phi(e^{-i\xi})|^2}\ud \xi.
$$

\subsection{Continuous-time Autoregressive Moving Average Processes}\label{section:CARMA}

In this subsection we recall the basic definition of the continuous-time autoregressive and moving average (CARMA) process. CARMA of order $(p,q)$, denoted as $\text{CARMA}(p,q)$, has been studied in \cite{Brockwell:2001} and \cite{Brockwell:2010}. 
Consider the following stochastic differential equation
\begin{align}\label{background:CARMA:equation}
a(D)\Phi_t=\sigma b(D)DW_t,
\end{align}
where $a(z)=z^p+a_1z^{p-1}+\ldots+a_p$ is the autoregression polynomial which is nonzero on the imaginary axis, $b(z)=b_0+b_1z+\ldots+b_qz^q$ is the moving average polynomial, $\sigma>0$, and $DW$ is the Gaussian white noise. 
In \cite{Brockwell:2001}, the case $p>q$ was considered and the solution is a stationary random process in the ordinary sense. The strictly stationary solution is unique if all singularities of the meromorphic function $b(z)/a(z)$ on the imaginary axis are removable \citep{Brockwell:2001,Brockwell:2010}. Also the necessary and sufficient condition for the existence of the CARMA$(p,0)$ process is that $a(z)$ has no zeros on the imaginary axis. 

In \cite{Brockwell:2010}, the constraint $p>q$ was removed by taking the generalized random process (GRP) \citep{Gelfand:1964} into consideration\footnote{Note that the standard Guassian white noise is $\text{CARMA}(0,0)$ GRP with $a(z)=1$ and $b(z)=1$, which is a special $p=q$ case.}. Indeed, with the assumption on $a(z)$ of order $p\geq0$ and $b(z)$ of order $q\geq0$, where $p$ may be less than or equal to $q$, the $\text{CARMA}(p,q)$ GRP $\Phi$ is defined as:
\begin{equation}\label{definition:CARMA:Y}
\Phi:=\left\{\begin{array}{ll}
\sum_{j=0}^qb_jX^{(j)}&\mbox{if }p>0\\
\sum_{j=0}^qb_jW^{(j+1)}&\mbox{if }p=0
\end{array}\right.,
\end{equation}
where
\begin{equation}\label{definition:CARMA:Xj}
X^{(j)}:=\ve_1^TA^jW^{(1)}\star \vg+\sum^{j-1}_{k=p-1} W^{(j-k)}\ve_1^TA^k\ve_p,
\end{equation}
$\ve_i$ is the unit $p$-vector with the $i$-th entry 1,
$$
A:= \left[
\begin{array}{ccccc}
0 & 1 & 0 & \ldots & 0 \\
0 & 0 & 1 & \ldots & 0 \\
\vdots & \vdots & \vdots &  & \vdots \\
0 & 0 & 0 & \ldots & 1 \\
-a_p & -a_{p-1} & -a_{p-2} & \ldots & -a_1
\end{array}
\right],
$$
$$
\vg=\int_{\gamma_l}\frac{e^{zt}[1~z~\ldots~z^{p-1}]^T}{a(z)}\ud z{\boldsymbol{\chi}}_{[0,\infty)}(t)-\int_{\gamma_r}\frac{e^{zt}[1~z~\ldots~z^{p-1}]^T}{a(z)}\ud z{\boldsymbol{\chi}}_{(-\infty,0]}(t),
$$ 
and $\gamma_l$ and $\gamma_r$ are close curves in $\CC$ enclosing the zeroes of $a(\cdot)$ with strictly negative and strictly positive real parts, respectively. We call $A$ the {\it companion matrix} of the autoregression polynomial $a(z)$. $\vg$ can be understood as the kernel of the CARMA process, which is clear when $p>q$. Notice that when $p>q$, $Y$ is causal when the real part of the eigenvalues of $A$ are all negative, which can be seen from (\ref{definition:CARMA:Y}) \citep{Brockwell:2009}. Also notice that when $p>q$, the summation term in (\ref{definition:CARMA:Xj}) does not exist; when $p=q$, the summation term is reduced to $W^{(1)}$. The algebraic forms of the residue of $e^{zt}[1~z~\ldots~z^{p-1}]^T/a(z)$ at the zeros $\lambda$ of $a(z)$ are
\begin{align}
\vg(t)=&\,\sum_{\lambda:\mathfrak{Re}\lambda<0}\sum_{k=1}^{\mu(\lambda)-1}\frac{ D^{\mu(\lambda)-1}_z\big[(z-\lambda)^{\mu(\lambda)}[1~z~\ldots~z^{p-1}]^Te^{zt}a^{-1}(z)\big]\Big|_{z=\lambda}}{2\pi i(\mu(\lambda)-1)!}\mathbf{1}_{[0,\infty)}(t)\nonumber\\
&\,-\sum_{\lambda:\mathfrak{Re}\lambda>0}\sum_{k=1}^{\mu(\lambda)-1}\frac{ D^{\mu(\lambda)-1}_z\big[(z-\lambda)^{\mu(\lambda)}[1~z~\ldots~z^{p-1}]^Te^{zt}a^{-1}(z)\big]\Big|_{z=\lambda}}{2\pi i(\mu(\lambda)-1)!}\mathbf{1}_{(-\infty,0]}(t)\nonumber\\
=&\,\sum_{\lambda:\mathfrak{Re}\lambda<0}\sum_{k=1}^{\mu(\lambda)-1}\boldsymbol{\alpha}_{\lambda k}t^ke^{\lambda t}{\boldsymbol{\chi}}_{(0,\infty)}(t)-\sum_{\lambda:\mathfrak{Re}\lambda>0}\sum_{k=1}^{\mu(\lambda)-1}\boldsymbol{\beta}_{\lambda k}t^ke^{\lambda t}{\boldsymbol{\chi}}_{(-\infty,0)}(t),\label{definition:CARMA:vg:expansion}
\end{align}
where $\mu(\lambda)$ is the multiplicity of the zero $\lambda$, $\boldsymbol{\alpha}_{\lambda k}, \boldsymbol{\beta}_{\lambda k}\in\CC^p$. 
Note that the entries of $\vg(t)$ are bounded and decay exponentially. 
Moreover, there exist vectors $\boldsymbol{\ell}(0)$ and $\boldsymbol{r}(0)$ \cite[Proposition 3.2]{Brockwell_Lindner:2009} so that
\begin{equation}\label{definition:CARMA:g:bound}
\vg=e^{At}\boldsymbol{\ell}(0){\boldsymbol{\chi}}_{[0,\infty)}(t)-e^{At}\boldsymbol{r}(0){\boldsymbol{\chi}}_{(-\infty,0]}(t).
\end{equation}

The relationship between the $\text{CARMA}(p,q)$ GRP defined in (\ref{background:CARMA:equation}) and discrete-time ARMA process is discussed in \cite{Brockwell:2010} and we briefly state it here. Define $\phi(z):=a(\delta^{-1}(1-z))$ and $\theta(z):=b(\delta^{-1}(1-z))$, where $\delta>0$ is small enough so that none of the zeros of $\phi(z)$ lies on the unit circle. Let $(\Phi_n)_{n\in\ZZ}$ be the stationary solution to the following ARMA equation:
\[
\phi(B)\Phi_n=\theta(B)\delta^{-1/2}Z_n,\,\,n\in\ZZ,
\]
where $Z_n$ is i.i.d. Gaussian with mean 0 and variance 1. Define a new GRP $\Phi^\delta_t:=\Phi_{[t/\delta]}$. It is shown in \cite{Brockwell:2010} that $\Phi^\delta_t$ converges to $\Phi_t$ in the sense of finite dimensional distribution. In this sense, we can approximate a CARMA(p,q) GRP by a ARMA time series in the distribution sense. 
  
We mention that the innovation process $W_t$ can be generalized to the Levi process, as is discussed in \cite{Brockwell_Lindner:2009}. To simply the discussion, we focus ourselves on the Brownian motion case. Also the CARMA can be generalized to the fractal CARMA process by introducing the fractal integral part. Again, to simplify the discussion, we focus ourselves on the CARMA case. We mention that the proof for Theorem \ref{section:theorem:stability} and \ref{section:theorem:ARMAstability} for the L\'{e}vy-driven CARMA process can be carried out in the same way when $p>q$. 

Below we show the power spectrum of a given CARMA$(p,q)$ GRP so that we can see how it depends on the scaling of the measurement function.

\begin{lem}\label{lemma:powerspectrum}
Fix a CARMA(p,q) GRP $\Phi$, where $p,q\in \NN\cup\{0\}$, defined in (\ref{background:CARMA:equation}). For $\psi\in\mathcal{S}$, we have
\begin{align}
\var \Phi(\psi_{a})=\int \Big|\sum_{\lambda}\frac{c_\lambda}{i2\pi\xi/a+\lambda}+d_{p,q}(a,\xi)\Big|^2|\widehat{\psi}(\xi)|^2\ud \xi,\nonumber
\end{align} 
where $\lambda$ are the roots of $a(z)$,
\[
d_{p,q}(a,\xi):=\sum_{i=p}^q\sum_{k=p-1}^{i-1}\frac{b_i\ve_1^TA^k\ve_p (2\pi i \xi)^{i-k-1}}{a^{i-k-1}},
\]
and $d_{0,0}(a,\xi)=1$. 
In particular, the power spectrum of $\Phi$ is
\begin{align}
\ud\eta(\xi) = \Big|\sum_{\lambda}\frac{c_\lambda}{i2\pi\xi+\lambda}+d_{p,q}(1,\xi)\Big|^2\ud \xi.\nonumber
\end{align}
\end{lem}

\underline{Remark:} Note that when $q<p$, $\ud\eta$ decays as fast as $|\xi|^{-1}$, so $C_\eta:=\int (1+|\xi|)^{-2l}\ud\eta$ is finite when $l\geq 1$. When $q\geq p$, $d_{p,q}$ is a polynomial of degree $q-p$. In this case, $C_\eta:=\int (1+|\xi|)^{-2l}\ud\eta$ is finite when $2l\geq q-p+2$. In conclusion, CARMA(p,q) process fits our Theorem.

\begin{proof}[Lemma \ref{lemma:powerspectrum}]
Take $\psi\in\mathcal{S}$. By (\ref{definition:CARMA:Y}), we have 
\begin{align}
&\Phi(\psi_{a,b})\,=\,\sum_{j=0}^q b_j\ve_1^T\big[A^jW^{(1)}\star \vg +\sum_{k=p-1}^{j-1}A^k\ve_p W^{(j-k)}\big](\psi_{a,b}) \nonumber\\
=&\,\sum_{j=0}^qb_j\ve_1^TA^j\int(\widetilde{\vg}\star\psi_{a,b})(t)\ud W(t)+\sum_{j=p}^q \sum_{k=p-1}^{j-1}(-1)^{j-k}b_j\ve_1^TA^k\ve_p W(\psi_{a,b}^{(j-k)}),\nonumber
\end{align}
where $\widetilde{\vg}(t)=\vg(-t)$ and $W(\psi_{a,b}^{(j-k)})$ is a complex-valued Gaussian random variable by the assumption. 
By a direct calculation we know 
\begin{align}
\var\Phi(\psi_{a,b})\,=&\,\EE\Big|\sum_{j=0}^qb_j\ve_1^TA^j\int(\widetilde{\vg}\star\psi_{a,b})(t)\ud W(t)\Big|^2\label{proof:varY:1}\\
&\,+\EE\Big| \sum_{j=p}^q\sum_{k=p-1}^{j-1}(-1)^{j-k}b_j\ve_1^TA^k\ve_p W(\psi_{a,b}^{(j-k)})\Big|^2\label{proof:varY:2}\\
&\,+2\mathfrak{Re}\EE\Big(\sum_{j=0}^q b_j\ve_1^TA^j\int(\widetilde{\vg}\star\psi_{a,b})(t)\ud W(t)\Big)\nonumber\\
&\,\qquad \times\Big(\sum_{j=p}^q \sum_{k=p-1}^{j-1}(-1)^{j-k}\overline{b_j\ve_1^TA^k\ve_p W(\psi_{a,b}^{(j-k)})}\Big).\label{proof:varY:3}
\end{align} 
Note that when $p>q$, (\ref{proof:varY:2}) and (\ref{proof:varY:3}) disappear. Recall the fact that 
\begin{equation}\label{proof:thm:EWf1Wg1}
\EE W(f^{(1)})\overline{W(g^{(1)})}=\int f\bar{g}\ud t
\end{equation}
when $f,g\in\mathcal{S}$. By (\ref{proof:thm:EWf1Wg1}), (\ref{proof:varY:1}) becomes
\begin{align}
&\,\EE\Big|\sum_{j=0}^qb_j\ve_1^TA^j\int(\widetilde{\vg}\star\psi_{a,b})(t)\ud W(t)\Big|^2\nonumber\\
=&\,\sum_{j,k=0}^q b_j\overline{b_k}\int \ve_1^TA^j(\widetilde{\vg}\star\psi_{a,b})(t)\overline{\ve_1^TA^k(\widetilde{\vg}\star\psi_{a,b})(t)}\ud t\nonumber\\
=&\,\sum_{j,k=0}^q b_j\overline{b_k}\tr \Big\{(A^j)^T\ve_1\ve_1^T\overline{A^k}\int (\widetilde{\vg}\star\psi_{a,b})(t)\overline{(\widetilde{\vg}\star\psi_{a,b})(t)}^T\ud t\Big\},\label{proof:thm:Var_term1:1}
\end{align}
where in the last equality we use the fact that $v^Twu^Tw=\tr (vu^Tww^T)$ when $v,w,u\in\CC^p$. Since each entry of $\vg$ is bounded, $C^\infty$ except at $0$ and exponentially decay, each entry of $\widetilde{\vg}\star \phi$ is a Schwartz function.

Thus, by the Plancheral theorem the integral in (\ref{proof:thm:Var_term1:1}) becomes
\begin{align}
&\,\int (\widetilde{\vg}\star\psi_{a,b})(t)\overline{(\widetilde{\vg}\star\psi_{a,b})(t)}^T\ud t\,=\,\int \widehat{\widetilde{\vg}\star\psi_{a,b}}(\xi)\overline{\widehat{\widetilde{\vg}\star\psi_{a,b}}(\xi)}^T\ud \xi\nonumber\\
=&\,\int \widehat{\widetilde{\vg}}(\xi)\widehat{\psi_{a,b}}(\xi)\overline{\widehat{\widetilde{\vg}}(\xi)^T\widehat{\psi_{a,b}}(\xi)}\ud \xi\,=\,a\int \widehat{\vg}(-\xi)\overline{\widehat{\vg}(-\xi)^T}|\widehat{\psi}(a\xi)|^2\ud \xi\label{proof:thm:Var_term1:2},
\end{align}
where the last equality holds due to $\widehat{\widetilde{\vg}}(\xi)=\widehat{\vg}(-\xi)$ and $\widehat{\psi_{a,b}}(\xi)=\sqrt{a}\widehat{\psi}(a\xi)e^{-i2\pi b\xi}$. By plugging (\ref{proof:thm:Var_term1:2}) into (\ref{proof:thm:Var_term1:1}), $\EE\Big|\sum_{j=0}^qb_j\ve_1^TA^j\int(\widetilde{\vg}\star\psi_{a,b})(t)\ud W(t)\Big|^2$ becomes
\begin{align}
&\,\sum_{j,k=0}^q b_j\overline{b_k}\tr \Big\{(A^j)^T\ve_1\ve_1^T\overline{A^k}\int \widehat{\vg}(-\xi)\overline{\widehat{\vg}(-\xi)^T}a|\widehat{\psi}(a\xi)|^2\ud \xi\Big\}\nonumber\\
=&\,\sum_{j,k=0}^q b_j\overline{b_k} \int \ve_1^TA^j\widehat{\vg}(-\xi)\overline{\ve_1^TA^k\widehat{\vg}(-\xi)}a|\widehat{\psi}(a\xi)|^2\ud \xi,\label{proof:thm:Var_term1:3}
\end{align}
By a direct calculation, we know when $\text{sign}(\mathfrak{Re}\lambda)<0$, the Fourier transform of $e^{\lambda t}\mathbf{1}_{[0,\infty)}$ is $\frac{1}{i2\pi\xi-\lambda}$; when $\text{sign}(\mathfrak{Re}\lambda)>0$, the Fourier transform of $e^{\lambda t}\mathbf{1}_{(-\infty,0]}$ is $\frac{1}{\lambda-i2\pi\xi}$. Thus by the assumption that the eigenvalues of $A$ are all simple, we know that
\begin{align}
\widehat{\vg}(-\xi)\,=&\,-\sum_{\lambda}\boldsymbol{\gamma}_{\lambda}\frac{1}{i2\pi\xi+\lambda}
\label{proof:thm:Var_term1:4},
\end{align}
where $\boldsymbol{\gamma}_\lambda=\boldsymbol{\alpha}_\lambda$ when $\mathfrak{Re}\lambda<0$ and $\boldsymbol{\gamma}_\lambda=\boldsymbol{\beta}_\lambda$ when $\mathfrak{Re}\lambda>0$.
Plugging (\ref{proof:thm:Var_term1:4}) into (\ref{proof:thm:Var_term1:3}) leads to
\begin{align}
&\EE\Big|\sum_{j=0}^qb_j\ve_1^TA^j\int(\widetilde{\vg}\star\psi_{a,b})(t)\ud W(t)\Big|^2\nonumber\\
=&\,a\sum_{\lambda,\mu}\Big[\sum_{j,k=0}^qb_j\ve_1^TA^j\boldsymbol{\gamma}_\lambda\overline{b_j\ve_1^TA^k\boldsymbol{\gamma}_\mu}\Big]\int \frac{|\widehat{\psi}(a\xi)|^2}{(\lambda+i2\pi\xi)(\bar{\mu}-i2\pi\xi)}\ud \xi\nonumber\\
=:&\,\int \sum_{\lambda}\frac{c_{\lambda}}{\lambda+i2\pi\xi/a}\overline{\sum_{\mu}\frac{c_{\mu}}{\mu+i2\pi\xi/a}}|\widehat{\psi}(\xi)|^2\ud \xi,\label{proof:thm:Var_term1:5}
\end{align}
where $c_{\lambda}:=\sum_{j=0}^qb_j\ve_1^TA^j\boldsymbol{\gamma}_\lambda\in\CC$. 
Next, (\ref{proof:varY:2}) becomes
\begin{align}
&\,\EE\Big| \sum_{i=p}^{q}\sum_{k=p-1}^{j-1}b_i\ve_1^TA^k\ve_p (-1)^{i-k}W(\psi_{a,b}^{(i-k)})\Big|^2\nonumber\\
=&\,\sum_{i,j=p}^q\sum_{k=p-1}^{i-1}\sum_{l=p-1}^{j-1}\Big[ b_i\ve_1^T A^k\ve_p\overline{b_j\ve_1^T A^l\ve_p}(-1)^{i+j-k-l}\EE W(\psi_{a,b}^{(i-k)})\overline{W(\psi_{a,b}^{(j-l)})}\Big]\nonumber\\
=&\,\sum_{i,j=p}^q\sum_{k=p-1}^{i-1}\sum_{l=p-1}^{j-1}(-1)^{i+j-k-l}\Big[ b_i\ve_1^T A^k\ve_p\overline{b_j\ve_1^T A^l\ve_p}\int\psi_{a,b}^{(i-k-1)}(t)\overline{\psi_{a,b}^{(j-l-1)}(t)}\ud t\Big],\label{proof:thm:Var_term2:0}
\end{align}
where due to $\widehat{\psi^{(k)}_{a,b}}(\xi)=(i2\pi \xi)^k\sqrt{a}\widehat{\psi}(a\xi)e^{-i2\pi b\xi}$ the integral in the last equality becomes
\begin{align}
(-1)^{j-l-1}\frac{(2\pi i)^{i+j-k-l-2}}{a^{i+j-k-l-2}}\int |\widehat{\psi}(\xi)|^2\xi^{i+j-k-l-2}\ud \xi\nonumber
\end{align}
and thus (\ref{proof:varY:2}) becomes
\begin{align}
\int |d_{p,q}(a,\xi)|^2|\widehat{\psi}(\xi)|^2\ud \xi.\label{proof:thm:Var_term2:1}
\end{align}

Similarly, by (\ref{proof:thm:EWf1Wg1}), (\ref{proof:varY:3}) becomes
\begin{align}
&\,\EE\Big[\sum_{j=0}^qb_j\ve_1^TA^j\int(\widetilde{\vg}\star\psi_{a,b})(t)\ud W(t)\Big]\Big[ \sum_{l=p}^q\sum_{k=p-1}^{l-1}(-1)^{l-k}\overline{b_l\ve_1^TA^k\ve_p W(\psi_{a,b}^{(l-k)})}\Big]\nonumber\\
=&\,\sum_{j=0}^q\sum_{l=p}^q\sum_{k=p-1}^{l-1}(-1)^{l-k}\overline{b_l\ve_1^TA^k\ve_p}b_j\ve_1^TA^j\EE \Big[\big(\int\widetilde{\vg}\star\psi_{a,b}(t)\ud W(t)\big)\overline{W(\psi_{a,b}^{(l-k)})}\Big]\nonumber\\
=&\,\sum_{j=0}^q\sum_{l=p}^q\sum_{k=p-1}^{l-1}(-1)^{l-k-1}\overline{b_l\ve_1^TA^k\ve_p}b_j\ve_1^TA^j \int\widetilde{\vg}\star\psi_{a,b}(t)\overline{\psi_{a,b}^{(l-k-1)}(t)}\ud t.\label{proof:thm:Var_term3:1}
\end{align}
By the same arguments as those for (\ref{proof:thm:Var_term1:1}), the integral term in (\ref{proof:thm:Var_term3:1}) becomes 
\begin{align}
&\,\int\widetilde{\vg}\star\psi_{a,b}(t)\overline{\psi_{a,b}^{(l-k-1)}(t)}\ud t\,=\,\int \widehat{\widetilde{\vg}}(\xi)\widehat{\psi_{a,b}}(\xi)\overline{(2\pi i\xi)^{l-k-1}\widehat{\psi_{a,b}}(\xi)}\ud \xi\nonumber\\
=&\,a(-2\pi i)^{l-k-1}\int \xi^{l-k-1}\widehat{\vg}(-\xi)|\widehat{\psi}(a\xi)|^2\ud \xi\nonumber\\
=&\,-\frac{(-2\pi i)^{l-k-1}}{a^{l-k-1}}\sum_{\lambda}\boldsymbol{\gamma}_{\lambda}\int \frac{\xi^{l-k-1}}{i2\pi\xi/a+\lambda}|\widehat{\psi}(\xi)|^2\ud \xi\nonumber
\end{align}
and (\ref{proof:thm:Var_term3:1}) becomes 
\begin{align}
&-\sum_{j=0}^q\sum_{l=p}^q\sum_{k=p-1}^{l-1}\frac{(2\pi i)^{l-k-1}}{a^{l-k-1}}\overline{b_l\ve_1^TA^k\ve_p}b_j\ve_1^TA^j \sum_{\lambda}\boldsymbol{\gamma}_{\lambda}\int \frac{\xi^{l-k-1}}{i2\pi\xi/a+\lambda}|\widehat{\psi}(\xi)|^2\ud \xi\nonumber\\
=&\,-\int\sum_{l=p}^q\sum_{k=p-1}^{l-1}\frac{(2\pi i\xi)^{l-k-1}}{a^{l-k-1}}\overline{b_l\ve_1^TA^k\ve_p}\sum_{\lambda}\sum_{j=0}^q b_j\ve_1^TA^j \boldsymbol{\gamma}_{\lambda} \frac{1}{i2\pi\xi/a+\lambda}|\widehat{\psi}(\xi)|^2\ud \xi\nonumber\\
=&\,-\int \overline{d_{p,q}(a,\xi)}\sum_{\lambda}\frac{c_\lambda}{i2\pi\xi/a+\lambda}|\widehat{\psi}(\xi)|^2\ud \xi.\label{proof:thm:Var_term3:3}
\end{align}
With (\ref{proof:thm:Var_term1:5}), (\ref{proof:thm:Var_term2:1}) and (\ref{proof:thm:Var_term3:3}), $\var\Phi(\psi_{a,b})$ becomes
\begin{align}
\var\Phi(\psi_{a,b})=\int \Big|\sum_{\lambda}\frac{c_\lambda}{i2\pi\xi/a+\lambda}+d_{p,q}(a,\xi)\Big|^2|\widehat{\psi}(\xi)|^2\ud \xi
\end{align} 
Thus, by polarization, we know that the unique positive tempered measure associated with $\Phi$ is
\begin{align}
\ud\eta(\xi) = \Big|\sum_{\lambda}\frac{c_\lambda}{i2\pi\xi+\lambda}+d_{p,q}(1,\xi)\Big|^2\ud \xi.
\end{align}
Notice that $d_{p,q}(a,\xi)=0$ when $p>q$, and by definition we have
\begin{align*}
\sum_{\lambda}\frac{c_\lambda}{i2\pi\xi+\lambda}=\sum_{j=0}^qb_j\ve_1^TA^j\sum_{\lambda}\frac{\boldsymbol{\gamma}_\lambda}{i2\pi\xi+\lambda}=\sum_{j=0}^qb_j\ve_1^TA^j\widehat{\vg}(-\xi)=b^T\widehat{\widetilde{\vg}}(\xi).
\end{align*}
As a result we have
$$
\ud \eta(\xi)=\Big|\sum_{\lambda}\frac{c_\lambda}{i2\pi\xi+\lambda}\Big|^2\ud\xi=\Big|\frac{b(i\xi)}{a(i\xi)}\Big|^2\ud\xi,
$$
which coincides with the result in \cite{Brockwell:2001}.

\end{proof}

\section{More simulation reports}
In this section we provide more simulation results for reference.

\subsection{The SST tested on signals with different error types}

We considered two more random error processes in addition to $X_1,X_2$ and $X_3$ studied in Section \ref{simulation}. Let $X_{\textup{ARMA3}}$ (resp. $X_{\textup{ARMA4}}$) be an ARMA(1,1) time series determined by the autoregression polynomial $a(z)=0.5z+1$ (resp. $a(z)=-0.2z+1$) and the moving averaging polynomial $b(z)=0.4z+1$ (resp. $b(z)=0.51z+1$), with the innovations taken as i.i.d. N$(0,1)$ random variables. Note that the only difference between $X_{\textup{ARMA3}}$ (resp. $X_{\textup{ARMA4}}$) and $X_{\textup{ARMA1}}$ (resp.  $X_{\textup{ARMA2}}$) is the innovation process. Define 
\begin{align*}
&X_4(n):=\sigma(n\tau)\big(4X_{\textup{ARMA3}}(n)\boldsymbol{\chi}_{n\in[1,N/2]}(n)+X_{\textup{ARMA4}}(n)\boldsymbol{\chi}_{n\in[N/2+1,N]}(n)\big),\\
&X_5(n):=\sigma(n\tau)\big(4X_{\textup{ARMA1}}(n)\boldsymbol{\chi}_{n\in[1,N/2]}(n)+X_{\textup{ARMA4}}(n)\boldsymbol{\chi}_{n\in[N/2+1,N]}(n)\big),
\end{align*}
where $\sigma(\cdot)$ is given in Section \ref{simulation} and $\boldsymbol{\chi}$ denotes the indicator function. Here $X_5$ differs from $X_4$ in the first half and differs from $X_2$ in the second half. 

The RRASE, and the standard deviation, of the results by SST for $\boldsymbol{Y_{j,k,\sigma_0}}:=\boldsymbol{s}_2+\boldsymbol{T}_j+\sigma_0\boldsymbol{X}_k$ with different combinations of $j=1,2$, $k=4,5$ and $\sigma_0=0.5,1, \sqrt{2}$, are shown in Table \ref{table:simulation:Y4Y5}. The average computational time (in seconds) and the standard deviation are reported as well. Note that the performance of SST on $\boldsymbol{Y_{j,2,1}}$ is in general worse than on $\boldsymbol{Y_{j,4,1}}$ but better than on $\boldsymbol{Y_{j,5,1}}$. This can be explained as below. 
The standard deviations of the student $t_4$ and N$(0,1)$ distributions are $\sqrt{2}$ and $1$ respectively, and by Theorem \ref{section:theorem:ARMAstability} the larger the standard deviation is, the worse the estimator is. Furthermore, from Table \ref{table:simulation:Y4Y5} we see that the performance of SST on $\boldsymbol{Y_{j,2,1}}$ is similar to that on $\boldsymbol{Y_{j,4,\sqrt{2}}}$, which again reflects the theoretical result.

\begin{table}
\caption{\label{table:simulation:Y4Y5}$\boldsymbol{Y_{j,k,\sigma_0}}$, $j=1,2, k=4,5$: Two dynamic seasonal components and trend, contaminated by different kinds of heteroscedastic, dependent noise.}
\centering
{\small
\begin{tabular}{| c | c | c | c | c | c |}
\hline
\cline{1-6}
$(j,k,\sigma_0)$  & $\widetilde{\boldsymbol{s_{2,1}}}$ & $\widetilde{\boldsymbol{s_{2,2}}}$ &  $\widetilde{\boldsymbol{T_j}}$ & $\widetilde{\boldsymbol{r}}$ & Time\\
\hline
$(1,4,0.5)$ & $0.111\pm 0.021$ & $0.103 \pm 0.013$ &  $ 0.013\pm 0.004$ & $ 0.212\pm 0.022$ & $ 8.91\pm 1.15$ \\  
\hline
$(2,4,0.5)$ & $ 0.120\pm 0.023$ & $ 0.103\pm 0.013$ &  $ 0.012\pm 0.002$ & $ 0.225\pm 0.021$ & $ 8.91\pm 0.92$ \\	
\hline
$(1,4,1)$ & $0.197 \pm 0.044$ & $0.175 \pm 0.027$ &  $0.025 \pm 0.007$ & $0.189 \pm 0.023$ & $8.89 \pm 1.11$ \\  
\hline
$(2,4,1)$ & $ 0.203\pm 0.046$ & $ 0.175\pm 0.027$ &  $ 0.019\pm 0.004$ & $ 0.193\pm 0.023$ & $ 8.94\pm1.12 $ \\	
\hline
$(1,4,\sqrt{2})$ & $0.268 \pm 0.064$ & $0.234 \pm 0.040$ &  $0.035 \pm 0.011$ & $0.182 \pm 0.024$ & $8.61 \pm 1.47$ \\  
\hline
$(2,4,\sqrt{2})$ & $ 0.272\pm 0.065$ & $ 0.234\pm 0.040$ &  $ 0.025\pm 0.007$ & $ 0.184\pm 0.023$ & $ 8.84\pm1.11 $ \\	
\hline\hline
$(1,5,0.5)$ & $0.137\pm0.030 $ & $ 0.127\pm 0.020$ &  $0.017 \pm 0.005$ & $0.196 \pm 0.023$ & $ 8.96\pm 1.13 $ \\  
\hline
$(2,5,0.5)$ & $ 0.145\pm 0.031$ & $ 0.127\pm 0.020$ &  $ 0.014\pm 0.002$ & $ 0.203\pm 0.022$ & $9.03 \pm 0.93$ \\	
\hline
$(1,5,1)$ & $0.249\pm 0.062$ & $ 0.223\pm 0.040$ &  $ 0.033\pm 0.009$ & $0.179 \pm 0.023$ & $ 9.16\pm 1.01$ \\  
\hline
$(2,5,1)$ & $ 0.253\pm 0.063$ & $ 0.223\pm 0.040$ &  $ 0.024\pm 0.006$ & $ 0.181\pm 0.023$ & $ 8.99\pm 1.08$ \\	
\hline
$(1,5,\sqrt{2})$ & $0.341 \pm 0.091$ & $0.298 \pm 0.055$ &  $0.047 \pm 0.013$ & $0.173 \pm 0.023$ & $8.88 \pm 1.00$ \\  
\hline
$(2,5,\sqrt{2})$ & $ 0.345\pm 0.094$ & $ 0.298\pm 0.056$ &  $ 0.033\pm 0.008$ & $ 0.174\pm 0.023$ & $ 8.70\pm1.09 $ \\	
\hline
\end{tabular}
}
\end{table}

\subsection{The sensitivity of SST tested on signals with dynamics and heteroscedastic, dependent noise}

Next we demonstrate the sensitivity issue of SST. We ran SST on 200 realizations of $\boldsymbol{Y^I_{j,k,\sigma_0}}:=\boldsymbol{Y_{j,k,\sigma_0}}|_{[0,9]}$, $\boldsymbol{Y_{j,k,\sigma_0}}$ restricted to the sub-interval $[0,9]$, to estimate the components restricted to $[0,9]$, denoted as $\boldsymbol{s^I_{2,1}}$, $\boldsymbol{s^I_{2,2}}$, $\boldsymbol{T^I_j}$ and $\boldsymbol{X^I_k}$. The RRASE of the resulting estimators, denoted as $\widetilde{\boldsymbol{s^I_{2,1}}}$, $\widetilde{\boldsymbol{s^I_{2,2}}}$, $\widetilde{\boldsymbol{T_j^I}}$ and the residual $\widetilde{\boldsymbol{r^I}}$, are reported in Table \ref{table:simulation:sensitivity}. In addition, we report in Table \ref{table:simulation:sensitivity} the difference between $\widetilde{\boldsymbol{s^I_{2,1}}}$ (resp. $\widetilde{\boldsymbol{s^I_{2,2}}}$, $\widetilde{\boldsymbol{T_j}^I}$ and $\widetilde{\boldsymbol{X_k}^I}$) and $\widetilde{\boldsymbol{s_{2,1}}}|_{[0,9]}$, i.e. $\widetilde{\boldsymbol{s_{2,1}}}$ restricted to $[0,9]$ (resp. $\widetilde{\boldsymbol{s_{2,2}}}|_{[0,9]}$, $\widetilde{\boldsymbol{T_j}}|_{[0,9]}$ and $\widetilde{\boldsymbol{X_k}}|_{[0,9]}$) measured by the relative root average square difference (RRASD), that is
\[
\text{RRASD}:=\frac{\sqrt{\big\|\widetilde{\boldsymbol{s^I_{2,1}}}-\widetilde{\boldsymbol{s_{2,1}}}|_{[0,9]}\big\|_2^2}}{\sqrt{\big\|\widetilde{\boldsymbol{s_{2,1}}}|_{[0,9]}\big\|_2^2}}.
\]
Among all the 200 realizations, in Figure \ref{fig:simulation:sst_recon_9sec} we demonstrate the results for the realization of $\boldsymbol{Y^I_{1,2,1}}$ which gave the median RRASD.

\begin{table}
\caption{\label{table:simulation:sensitivity} $\boldsymbol{Y^I_{j,k,\sigma_0}}$: 
Sensitivity of SST tested on two dynamic seasonal components and trend, contaminated by different kinds of noise. 
In the parentheses are the RRASD and its standard deviation. 
}
\centering
{\small
\begin{tabular}{| c | c | c | c | c | c |}
\hline
\cline{1-6}
 $(j,k,\sigma_0)$ & $\widetilde{\boldsymbol{s^I_{2,1}}}$ & $\widetilde{\boldsymbol{s^I_{2,2}}}$ &  $\widetilde{\boldsymbol{T^I_j}}$ & $\widetilde{\boldsymbol{r^I}}$ & Time\\
\hline
$(1,2,0.5)$  & $ 0.150\pm 0.032$ & $0.132 \pm 0.021$ &  $0.018 \pm 0.005$ & $0.198 \pm 0.023$ & $8.37 \pm 0.95$ \\	
                   & ($ 0.087\pm 0.009$) & ($0.017 \pm 0.006$) &  ($0.005 \pm 0.001$) & ($0.040 \pm 0.006$) &  \\
\hline
$(2,2,0.5)$  & $ 0.169 \pm 0.035 $ & $ 0.132 \pm 0.021$ &  $ 0.011 \pm 0.001$ & $ 0.242 \pm 0.022 $ & $ 8.47 \pm 0.84 $ \\	
                   & ($ 0.102 \pm 0.009 $) & ($ 0.017 \pm 0.006 $) &  ($ 0.545 \pm 0.001 $) & ($ 0.145 \pm 0.013 $) &  \\
\hline
$(1,2,1)$  & $ 0.272\pm 0.066$ & $0.233 \pm 0.041$ &  $0.036 \pm 0.010$ & $0.182 \pm 0.023$ & $8.5 \pm 1.15$ \\	
                & ($ 0.097\pm 0.024$) & ($0.027 \pm 0.012$) &  ($0.007 \pm 0.003$) & ($0.032 \pm 0.008$) &  \\
\hline
$(2,2,1)$  & $ 0.289 \pm 0.077 $ & $ 0.233 \pm 0.041$ &  $ 0.015 \pm 0.003$ & $ 0.197 \pm 0.023 $ & $ 8.68 \pm 0.77 $ \\	
                   & ($ 0.120 \pm 0.068 $) & ($ 0.027 \pm 0.011 $) &  ($ 0.545 \pm 0.002 $) & ($ 0.080 \pm 0.016 $) &  \\
\hline\hline
$(1,3,0.5)$  & $ 0.135 \pm 0.025 $ & $ 0.115 \pm 0.016$ &  $ 0.015 \pm 0.004$ & $ 0.204 \pm 0.022 $ & $  8.34\pm 1.1 $ \\	
                   & ($ 0.090 \pm 0.011 $) & ($ 0.021 \pm 0.009 $) &  ($ 0.006 \pm 0.002 $) & ($ 0.053 \pm 0.010 $) &  \\
\hline
$(2,3,0.5)$  & $  0.156\pm 0.027 $ & $ 0.116 \pm 0.016$ &  $ 0.010 \pm 0.001$ & $ 0.263 \pm 0.018 $ & $  8.54\pm 1.08 $ \\	
                   & ($ 0.104 \pm 0.012 $) & ($ 0.021 \pm 0.008 $) &  ($ 0.545 \pm 0.001 $) & ($ 0.174 \pm 0.017 $) &  \\
\hline
$(1,3,1)$  & $ 0.242 \pm 0.054 $ & $ 0.201 \pm 0.033$ &  $ 0.029 \pm 0.008$ & $ 0.187 \pm 0.022 $ & $ 8.92 \pm 0.99 $ \\	
                   & ($ 0.106 \pm 0.029 $) & ($ 0.035 \pm 0.016 $) &  ($ 0.008 \pm 0.003 $) & ($ 0.043 \pm 0.012 $) &  \\
\hline
$(2,3,1)$  & $ 0.256 \pm 0.057 $ & $ 0.202 \pm 0.033$ &  $ 0.013 \pm 0.002$ & $ 0.205 \pm 0.021 $ & $ 8.92 \pm 0.92 $ \\	
                   & ($ 0.119 \pm 0.039 $) & ($ 0.035 \pm 0.016 $) &  ($ 0.545 \pm 0.002 $) & ($ 0.096 \pm 0.016 $) &  \\
\hline
\end{tabular}
}
\end{table}

\begin{figure}[ht]
\includegraphics[width=1\textwidth]{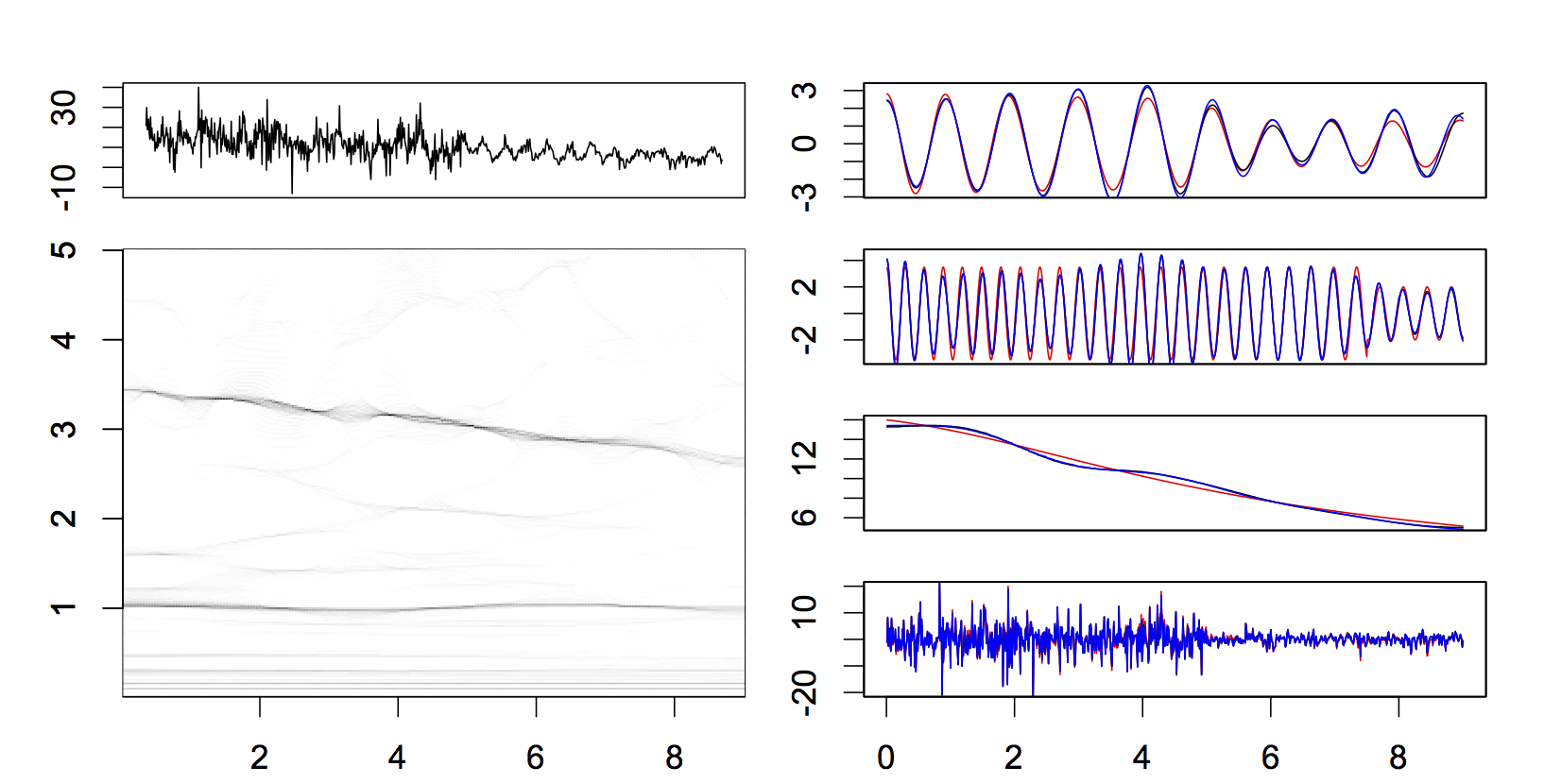}
\caption{{\it Sensitivity of SST tested on $\boldsymbol{Y^I_{1,2,1}}$.}  Left column: The realization of $\boldsymbol{Y^I_{1,2,1}}$ which yielded the median RRASD (upper panel) and SST of that realization (lower panel). Visually we can see that the SST is close to that in Figure \ref{fig:simulation:sst_recon}. Right column: From top to bottom the black curves are $\widetilde{\boldsymbol{s^I_{2,1}}}$, $\widetilde{\boldsymbol{s^I_{2,2}}}$, $\widetilde{\boldsymbol{T^I_1}}$ and $\widetilde{\boldsymbol{X^I_2}}$, with the respective $\boldsymbol{s_{2,1}^I}$, $\boldsymbol{s_{2,2}^I}$, $\boldsymbol{T^I_1}$ and $\boldsymbol{X^I_2}$ superimposed as red curves and $\widetilde{\boldsymbol{s_{2,1}}}|_{[0,9]}$, $\widetilde{\boldsymbol{s_{2,2}}}|_{[0,9]}$, $\widetilde{\boldsymbol{T_1}}|_{[0,9]}$ and $\widetilde{\boldsymbol{X_2}}|_{[0,9]}$ as blue curves. Note that the differences between black curves and blue curves are not significant.}
\label{fig:simulation:sst_recon_9sec}
\end{figure}

According to the simulation results, we see that the estimates obtained from the observations in the time interval $0$ to $9$ seconds do not differ much from those obtained from the observations in the interval $0$ to $10$ seconds. In particular, from Figure \ref{fig:simulation:sst_recon_9sec} visually we can see that the SST based on $\boldsymbol{Y^I_{1,2,1}}$ is close to that in Figure \ref{fig:simulation:sst_recon} when restricted to the time interval $[0,9]$. Also, visually the respective reconstructed components are close to each other. Notice that this visual closeness quantified and summarized in Table \ref{table:simulation:sensitivity} is much smaller than the estimation error.

\subsection{The SST tested on signals with local bursts}

Since local bursts are common in real data, although we do not provide a formal analysis, we empirically show the robustness of SST to them by considering the following deterministic delta peaks modeling local bursts: 
\[
\boldsymbol{O}(n)=18\delta_{4/\tau,n}-20\delta_{7/\tau,n},
\]
where $\delta$ is the Kronecker delta.
Then we analyzed $\boldsymbol{Y_{j,k,\sigma_0}}+\boldsymbol{O}$ and
the RRASE, and the standard deviation, of the results by SST are shown in Table \ref{table:simulation:outlier}. The average computation time (in seconds) and the standard deviation are reported as well.
Among all the 200 realizations of $\boldsymbol{Y_{1,2,1}}+\boldsymbol{O}$ and $\boldsymbol{Y_{1,3,1}}+\boldsymbol{O}$, we demonstrate the results for those gave the median RRASE in Figure \ref{fig:simulation:sst_outlier} and Figure \ref{fig:simulation:sst_outlier2}. Visually we can see that the SST in Figure \ref{fig:simulation:sst_outlier} (reps. Figure \ref{fig:simulation:sst_outlier2})  is similar to that in Figure \ref{fig:simulation:sst_recon} (reps. Figure \ref{fig:simulation:sst_recon2}), that is, the local burst does not influence much on the time frequency representation. Also, from these figures, we see that the reconstructions of the seasonal components and the trend are not significantly changed by the local burst, while the reconstruction of the error process is influenced. Indeed, the local burst is viewed as part of the noise process by SST.

Note that the local bursts are classified by SST as the noise in the final decomposition, and hence the error process estimation is deteriorated. To the best of our knowledge, unless a priori knowledge is available, how to distinguish them from the error process solely based on a single observed time series is by far an open problem. We mention that the local bursts analysis we convey here only covers a small portion of the whole field. We will report a further analysis when local bursts exist as an a priori knowledge in a future paper.

\begin{table}
\caption{\label{table:simulation:outlier}$\boldsymbol{Y_{j,k,\sigma_0}}+\boldsymbol{O}$: Two dynamic seasonal components and trend, contaminated by different kinds of heteroscedastic, dependent noise and local bursts.}
\centering
{\small
\begin{tabular}{| c | c | c | c | c | c |}
\hline
\cline{1-6}
$(j,k,\sigma_0)$  & $\widetilde{\boldsymbol{s}_{2,1}}$ & $\widetilde{\boldsymbol{s}_{2,2}}$ &  $\widetilde{\boldsymbol{T_j}}$ & $\widetilde{\boldsymbol{r}}$ & Time\\
\hline
$(1,2,0.5)$ & $0.159\pm 0.031$ & $0.143 \pm 0.020$ &  $ 0.019\pm 0.005$ & $ 0.447\pm 0.027$ & $ 9.6\pm 1.124$ \\  
\hline
$(2,2,0.5)$ & $0.165\pm0.033 $ & $ 0.143\pm 0.020$ &  $0.016 \pm 0.003$ & $0.450 \pm 0.028$ & $ 9.01\pm 0.83 $ \\  
\hline
$(1,2,1)$ & $0.272 \pm 0.062$ & $0.239 \pm 0.039$ &  $0.036 \pm 0.009$ & $0.272 \pm 0.020$ & $8.17 \pm 0.55$ \\  
\hline
$(2,2,1)$ & $0.276\pm 0.064$ & $ 0.239\pm 0.039$ &  $ 0.026\pm 0.006$ & $0.274 \pm 0.020$ & $ 9.16\pm 0.9$ \\  
\hline\hline
$(1,3,0.5)$ & $0.149\pm 0.026$ & $ 0.132\pm 0.017$ &  $ 0.017\pm 0.004$ & $ 0.499\pm 0.020$ & $ 9.61\pm1.11 $ \\  
\hline
$(2,3,0.5)$ & $0.157\pm 0.027$ & $ 0.132\pm 0.017$ &  $0.014 \pm 0.002$ & $ 0.503\pm 0.020$ & $ 9.57\pm 0.97$ \\  
\hline
$(1,3,1)$ & $ 0.250\pm 0.050$ & $0.216 \pm 0.033$ &  $0.031 \pm 0.008$ & $0.294 \pm 0.016$ & $8.82 \pm 0.33$ \\	
\hline
$(2,3,1)$ & $0.255\pm 0.052$ & $ 0.217\pm 0.033$ &  $ 0.022\pm 0.005$ & $ 0.296\pm 0.016$ & $ 8.81\pm 0.38$ \\  
\hline
\end{tabular}
}
\end{table}

\begin{figure}[ht]
\includegraphics[width=1\textwidth]{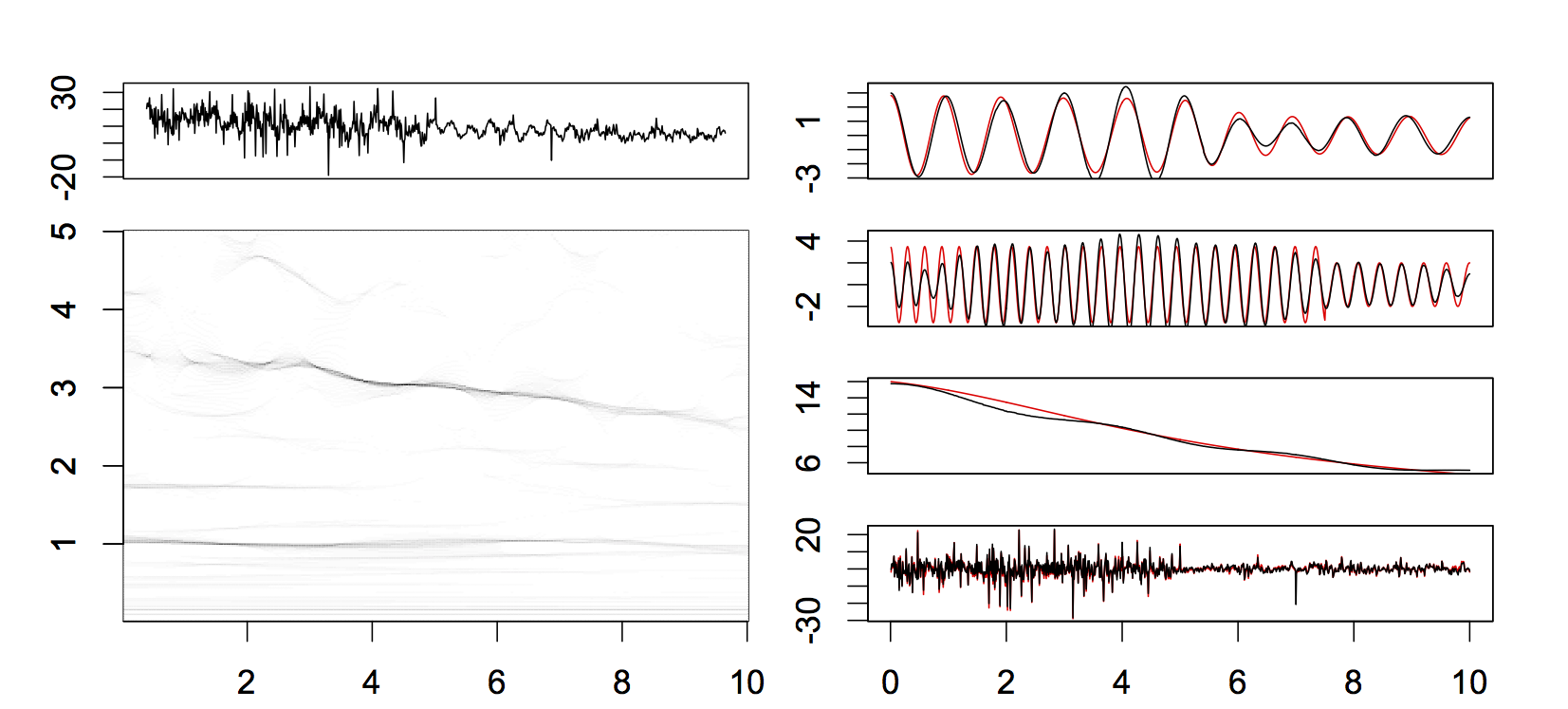}
\caption{{\it SST results of $\boldsymbol{Y_{1,2,1}}+\boldsymbol{O}$.}  Left column: The realization of $\boldsymbol{Y_{1,2,1}}+\boldsymbol{O}$ which yielded the median RRASD (upper panel) and the SST of that realization (lower panel). 
Right column: From top to bottom the black curves are $\widetilde{\boldsymbol{s_{2,1}}}$, $\widetilde{\boldsymbol{s_{2,2}}}$, $\widetilde{\boldsymbol{T_1}}$ and $\widetilde{\boldsymbol{r}}$, with the respective $\boldsymbol{s_{2,1}}$, $\boldsymbol{s_{2,2}}$, $\boldsymbol{T_1}$ and $\boldsymbol{X_2}$ superimposed as red curves. }
\label{fig:simulation:sst_outlier}
\end{figure}

\begin{figure}[ht]
\includegraphics[width=1\textwidth]{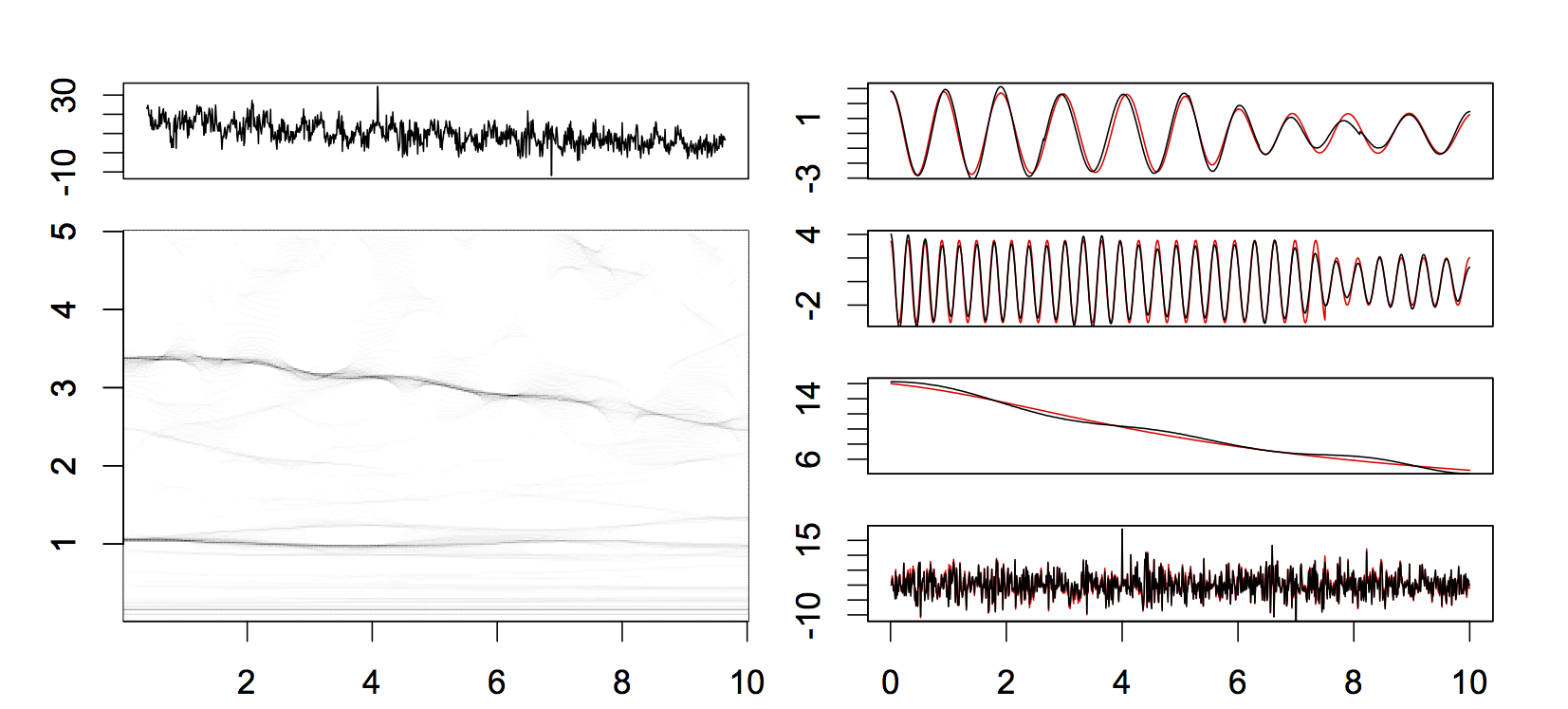}
\caption{{\it SST results of $\boldsymbol{Y_{1,3,1}}+\boldsymbol{O}$.}  Left column: The realization of $\boldsymbol{Y_{1,3,1}}+\boldsymbol{O}$ which yielded the median RRASD (upper panel) and the SST of that realization (lower panel). 
Right column: From top to bottom the black curves are $\widetilde{\boldsymbol{s_{2,1}}}$, $\widetilde{\boldsymbol{s_{2,2}}}$, $\widetilde{\boldsymbol{T_1}}$ and $\widetilde{\boldsymbol{r}}$, with the respective $\boldsymbol{s_{2,1}}$, $\boldsymbol{s_{2,2}}$, $\boldsymbol{T_1}$ and $\boldsymbol{X_3}$ superimposed as red curves. 
}
\label{fig:simulation:sst_outlier2}
\end{figure}

\end{document}